\documentclass[12pt,letterpaper,titlepage]{amsart}
\usepackage{amsmath, amssymb, amsthm, amsfonts,amscd,amsaddr,wasysym,enumerate}
 \usepackage[
    paper=a4paper,
    portrait=true,
    textwidth=425pt,
    textheight=650pt,
    tmargin=3cm,
    marginratio=1:1
            ]{geometry}

\theoremstyle{plain}
\newtheorem{theorem}{Theorem}[section]

\newtheorem{cor}[theorem]{Corollary}

\newtheorem{prop}[theorem]{Proposition}
\newtheorem{lemma}[theorem]{Lemma}

\theoremstyle{definition}
\newtheorem{definition}[theorem]{Definition}

\newtheorem{rmk}[theorem]{Remark}
\numberwithin{equation}{section}
\newtheorem*{theoremA*}{Theorem A}
\newtheorem*{theoremB*}{Theorem B}
\newtheorem*{theorem1*}{Theorem A'}
\newtheorem*{theoremC*}{Theorem C}
\newtheorem*{theoremD*}{Theorem D}
\newtheorem*{theoremE*}{Theorem E}
\newtheorem*{theoremF*}{Theorem F}
\newtheorem*{theoremE2*}{Theorem E2}
\newtheorem*{theoremE3*}{Theorem E3}
\newcommand{\bs}{\backslash}
\newcommand{\Cc}{\mathcal{C}}
\newcommand{\Mcal}{\mathcal{M}}
\newcommand{\C}{\mathbb{C}}
\newcommand{\G}{\mathbb{G}}

\newcommand{\E}{\mathcal{E}}

\newcommand{\Z}{\mathbb{Z}}

\newcommand{\R}{\mathbb{R}}
\newcommand{\N}{\mathbb{N}}

\newcommand{\Gr}{\operatorname{Gr}}

\newcommand{\Ind}{\operatorname{Ind}}
\newcommand{\Int}{\operatorname{int}}

\newcommand{\Hom}{\operatorname{Hom}}
\newcommand{\End}{\operatorname{End}}

\newcommand{\tr}{\operatorname{tr}}
\newcommand{\im}{\operatorname{Im}}

\newcommand{\Ad}{\operatorname{Ad}}

\newcommand{\ad}{\operatorname{ad}}

\newcommand{\vol}{\operatorname{vol}}

\newcommand{\supp}{\operatorname{supp}}
\newcommand{\Span}{\operatorname{span}}

\newcommand{\rank}{\operatorname{rank}}

\newcommand{\re}{\operatorname{Re}}

\def\dotvar{\, \cdot\,}

\def\hat{\widehat}
\def\af{\mathfrak{a}}
\def\bfrak{\mathfrak{b}}

\def\gf{\mathfrak{g}}

\def\cf{\mathfrak{c}}

\def\ef{\mathfrak{e}}
\def\hf{\mathfrak{h}}
\def\kf{\mathfrak{k}}
\def\lf{\mathfrak{l}}
\def\mf{\mathfrak{m}}
\def\nf{\mathfrak{n}}

\def\pf{\mathfrak{p}}

\def\qf{\mathfrak{q}}

\def\sf{\mathfrak{s}}

\def\tf{\mathfrak{t}}
\def\uf{\mathfrak{u}}
\def\vf{\mathfrak{v}}

\def\la{\langle}
\def\ra{\rangle}
\def\1{{\bf1}}

\def\U{\mathcal{U}}

\def\Cc{\mathcal{C}}
\def\D{\mathcal {D}}

\def\G{\mathcal{G}}
\def\Oc{\mathcal{O}}
\def\P{\mathcal{P}}

\def\M{\mathcal{M}}
\def\oline{\overline}

\def\W{\mathcal{W}}

\def\tilde{\widetilde}

\def\v{\mathbf{v}}
\def\w{\mathbf {w}}

\usepackage[usenames]{color}

\title[Discrete series]{The infinitesimal characters of discrete series for real spherical spaces}

\begin{document}
\date{\today}

\begin{abstract}
Let $Z=G/H$ be the homogeneous space of a real reductive group and a unimodular
real spherical subgroup, and consider the regular representation of $G$ on $L^2(Z)$.
It is shown that all representations of the discrete series, that is, the irreducible
subrepresentations of $L^2(Z)$, have infinitesimal characters which are real
and  belong to a lattice.
Moreover, let $K$ be a maximal compact subgroup of $G$.
Then each irreducible representation of $K$ occurs in a finite set of such
discrete series representations only. Similar results are obtained for the
twisted discrete series, that is, the discrete components of the space of square integrable
sections of a line bundle, given by a unitary character on an abelian extension of~$H$.
\end{abstract}

\author[Kr\"otz]{Bernhard Kr\"{o}tz}
\email{bkroetz@gmx.de}
\address{Universit\"at Paderborn, Institut f\"ur Mathematik\\Warburger Stra\ss e 100,
33098 Paderborn}

\author[Kuit]{Job J. Kuit}
\email{j.j.kuit@gmail.com}
\address{Universit\"at Paderborn, Institut f\"ur Mathematik\\Warburger Stra\ss e 100,
33098 Paderborn}

\author[Opdam]{Eric M. Opdam}
\email{e.m.opdam@uva.nl}
\address{University of Amsterdam, Korteweg-de Vries Institute for Mathematics\\P.O. Box 94248, 1090 GE Amsterdam}

\author[Schlichtkrull]{Henrik Schlichtkrull}
\email{schlicht@math.ku.dk}
\address{University of Copenhagen, Department of Mathematics\\Universitetsparken 5,
DK-2100 Copenhagen \O}

\thanks{JJK was funded by the Deutsche Forschungsgemeinschaft grant 262362164}

\maketitle

\section{Introduction}
Let $Z=G/H$ be a homogeneous space attached to a real reductive group $G$ and a closed subgroup
$H$.
A principal objective in the harmonic analysis of $Z$
is the understanding of the $G$-equivariant spectral
decomposition of the space $L^2(Z)$ of square integrable half-densities.
The irreducible components of $L^2(Z)$ are of particular interest,
they comprise the discrete series for $Z$. We will assume that
$Z$ is unimodular, that is, it
carries a positive $G$-invariant Radon measure. Then $L^2(Z)$ is identified as
the space of square integrable functions  with respect to this measure.

Later on we shall restrict ourselves to the case where $Z$ is real spherical, that is,
the action of a minimal parabolic
subgroup $P\subseteq G$ on $Z$ admits an open orbit.  Symmetric spaces are real spherical, as well
as real forms of complex spherical spaces.  We mention that a classification of real
spherical spaces $G/H$ with $H$ reductive became recently available, see \cite{KKPS} and \cite{KKPS2}.

For symmetric spaces it is known (see \cite{Delorme}, \cite{vdBS})
that the spectral components of $L^2(Z)$ are
built by means of induction from certain parabolic subgroups of $G$.
The inducing representations belong to the discrete series of a symmetric space
of the Levi subgroup, twisted by unitary characters on its center.
For real spherical spaces the results on tempered
representations obtained in \cite{KKS2} suggest
similarly that the spectral decomposition of $L^2(Z)$ will be
built from the twisted discrete spectrum of a certain
finite set of satellites $Z_I=G/H_I$ of $Z$, which are again
unimodular real spherical spaces. A first step towards obtaining a spectral
decomposition is then to obtain key properties of the
twisted discrete series for all unimodular real spherical spaces.

As usual we write $\hat G$ for the unitary dual of $G$ and disregard the distinction
between equivalence classes  $[\pi]\in \hat G$ and their representatives $\pi$.
Representations $\pi \in \hat G$ which occur in $L^2(Z)$ discretely will be called representations of the {\it discrete series
for $Z$}.  This notion distinguishes a subset of $\hat G$ which we denote by $\hat G_{H, {\rm d}}$.
We write $\hat G_{\rm d}$ for the discrete series
of $G$, i.e., $\hat G_{\rm d}=\hat G_{\{e\},{\rm d}}$.
Note that in general there is no relation between the sets $\hat G_{\rm d}$ and $\hat G_{H,{\rm d}}$ if $H$ is non-trivial.

To explain the notion of being twisted we recall the automorphism group
$N_G(H)/ H$ of $Z$, where $N_G(H)$ denotes the normalizer of $H$.
It gives rise to a right action of $N_G(H)/ H$ on $L^2(Z)$ commuting with the left regular action of $G$.
For a real spherical space
$N_G(H)/ H$ is fairly well behaved:
$N_G(H)/H$ is a product of a compact group
and a non-compact torus \cite{KKS}.
It is easy to see that in this case
there exists no discrete spectrum unless $N_G(H)/ H$ is compact.
Let ${\mathcal A}$ be a maximal non-compact torus in $N_G(H)/H$. Hence
if ${\mathcal A}$ is non-trivial, there exist no discrete series representations for $Z$.
In this case we generalize the notion of discrete series as follows.
We have an equivariant disintegration into $G$-modules
$$
L^2(Z) \simeq \int_{ \hat {\mathcal A}}^\oplus  L^2(Z; \chi) \ d \chi\,.
$$
Here $\hat {\mathcal A}$ denotes the set of unitary characters
$\chi$ of $\mathcal A$,
and $L^2(Z;\chi)$ denotes the space of functions on $Z$,
which transform
by $\chi$ (times a modular character) and are square integrable
modulo $\mathcal A$ (as half-densities, since in general $G/N_G(H)$
is not unimodular).
The set of representations $\pi \in \hat G$ which are in the discrete spectrum of
$L^2(Z;\chi)$ is called the {\it $\chi$-twisted discrete series} and is denoted $\hat G_{H,\chi}$. The union $\hat G_{H, {\rm td}}$
of these sets over all  $\chi\in\hat{\mathcal A}$
is referred to as the
{\it twisted discrete series for $Z$}.

Let $P=MAN$ be a Langlands decomposition of the minimal parabolic
subgroup $P$. Denote by $\mf$ and $\af$
the Lie algebras of $M$ and $A$ respectively.  Choose a maximal torus $\tf \subseteq \mf$  and set
$\cf:= \af +i \tf$.  We note that $\cf_\C$ is Cartan subalgebra of $\gf_\C$
and denote by $W_{\cf}$ the Weyl group
of the root system $\Sigma(\gf_\C, \cf)\subseteq \cf^*$.  For every $\pi \in \hat G$ we denote by
$\chi_\pi\in \cf_\C^*/ W_\cf$ its infinitesimal character and recall a theorem of Harish-Chandra
(\cite[Thm.~7]{H-C1956}), which asserts
that the map
\begin{equation}\label{infinitesimal character map}
 {\mathfrak X}:  \hat G \to \cf_\C^*/ W_\cf,  \ \ \pi\mapsto \chi_\pi
\end{equation}
has uniformly finite fibers. Note that ${\mathfrak X}$ is continuous if $\hat G$ is endowed with the Fell topology.

A priori it is not clear that ${\mathfrak X}(\hat G_{H, {\rm d}})$ or ${\mathfrak X}(\hat G_{H, \chi})$ is a discrete subset of $\cf_{\C}^{*}/W_{\cf}$.
However, we believe this to be true for general real algebraic homogeneous spaces $Z$.
For real spherical spaces $Z$ it is a
consequence of the main theorem, Theorem \ref{main thm} below, which slightly simplified can be phrased as follows.

\begin{theorem} \label{thm intro} Let $Z=G/H$ be a
unimodular real spherical space. Assume that the pair $(G,H)$ is real algebraic. Then there exists a $W_\cf$-invariant lattice $\Lambda_Z \subseteq \cf^*$, rational with respect to the root system in $\cf$,
 such that:
\begin{enumerate}[(i)]
\item\label{Thm intro item 1} ${\mathfrak X} (\hat G_{H, {\rm d}}) \subseteq \Lambda_Z/ W_\cf$,
\item\label{Thm intro item 2} $\re {\mathfrak X} (\hat G_{H, {\rm td}}) \subseteq \Lambda_Z/ W_\cf$.
\end{enumerate}
\end{theorem}

A few remarks related to this theorem are in order.

\begin{rmk}\label{rmk intro}
\quad\begin{enumerate}[(1)]
\item
The statement in (\ref{Thm intro item 1}) implies
that the infinitesimal characters $\chi_\pi$ are {\it real} and
{\it discrete} for
$\pi\in \hat G_{H, {\rm d}}$.
Furthermore (see Corollary \ref{K-type} below),
these properties of $\chi_\pi$ lead to the following.
Let $K\subseteq G$ be a maximal compact subgroup.
For all $\tau\in\hat K$ and $\chi\in\hat{\mathcal A}$
the set
$$
\{ \pi \in \hat G_{H,\chi}\mid \Hom_K (\pi\big|_K, \tau)\neq 0\}
$$
is finite. In other words,
there are only finitely many $\chi$-twisted discrete series representations containing a
given $K$-type.
For p-adic spherical spaces of wavefront type this was shown by Sakellaridis and Venkatesh in
\cite[Theorem 9.2.1]{SV}.

\item
There is a simple relation between the leading exponents of generalized matrix coefficients attached to
$\pi \in \hat G_{H, {\rm td}}$
and the infinitesimal character $\chi_\pi$ of $\pi$
(cf.~Lemma \ref{disc cover}).
Further, twisted discrete
series can be described by inequalities satisfied by the leading exponents (cf. \cite{KKS2} or (\ref{DS1})-(\ref{DS2}) below).
The inclusion $\re {\mathfrak X} (\hat G_{H, {\rm td}}) \subseteq \Lambda_Z/ W_\cf$ then implies that
all real parts of leading exponents are uniformly bounded away from "rho". Phrased differently,
Theorem \ref{thm intro}(\ref{Thm intro item 2}) implies a  spectral gap for twisted discrete series.
In \cite{SV}, Prop.~9.4.8,  this is called "uniform boundedness
of exponents" and is a key fact for establishing the Plancherel
formula for p-adic spherical spaces of wavefront type.

\item
The lattice $\Lambda_{Z}$ can be taken of the form $\frac{1}{N}\Sigma(\gf_{\C},\cf)$, where $N$ is an integer which only depends on $\gf$. (We may use the integer $N$ from Theorem \ref{main thm}, which is the product of the integers from Theorem \ref{fundamental lemma} and Proposition \ref{Prop intertwining operators}. The latter two integers only depend on $\gf$.)
\end{enumerate}
\end{rmk}

Theorem \ref{thm intro} is the crucial ingredient for the uniform constant term approximation for tempered eigenfunctions in \cite{DKS}. Thus it lies at the heart of the Plancherel
theorem for $L^2(Z)$ in terms of Bernstein-morphisms, established in \cite{DKKS} and motivated by
\cite{SV}, Section 11.
Notice that the strategy of proof designed in \cite{SV} for the Plancherel
theorem differs from the earlier approach where the discrete
spectrum is classified first (see \cite{HC1} for groups and
\cite{vdBS}, \cite{Delorme} for symmetric spaces). In \cite{SV}
the discrete series is taken as a black
box which features a spectral gap, and the Plancherel theorem is established
without knowing the discrete spectrum explicitly.

For reductive groups an explicit parametrization of the discrete series $\hat G_{\rm d}$
was obtained  by Harish-Chandra \cite{HC}.
More generally, for symmetric spaces $G/H$ discrete series were constructed  by
Flensted-Jensen \cite{FJ}, and his work
was completed by Matsuki and Oshima \cite{MO} to a full classification of $\hat G_{H,{\rm d}}$.
For a general real spherical
space such an explicit parameter description appears currently to be out of reach
and for non-symmetric spaces the existence or non-existence of discrete series is known only in a few cases. See \cite[Corollary 5.6]{Kob} and in \cite[Corollary 4.5]{Huang}.

More importantly, the existence of discrete series can be  characterized
geometrically by the existence of a compact Cartan subalgebra in the group case,
and of a compact Cartan subspace in $\hf^\perp$ in the more general case of
symmetric spaces.
One can phrase this uniformly as:
\begin{equation}\label{d-conj}
\hat G_{H, {\rm d}}\neq \emptyset \iff
\Int\{ X\in \hf^\perp\mid X \ \text{elliptic}\} \neq \emptyset\,,
\end{equation}
where the interior $\Int$ is taken in $\hf^\perp$.
We expect that (\ref{d-conj}) is true for all algebraic homogeneous spaces $Z$.
A geometric
characterization for the existence of twisted discrete series is less clear;
in the real spherical
case we expect
\begin{equation} \label{td-conj}
\hat G_{H, {\rm td}}\neq \emptyset \iff
\Int\{ X\in N_{\gf}(\hf)^\perp\mid X \ \text{weakly elliptic}\} \neq \emptyset
\end{equation}
with $N_\gf(\hf)$ the normalizer of $\hf$ in $\gf$.

A combination of the Bernstein decomposition of $L^2(Z)$ in \cite{DKKS}
with soft techniques from microlocal analysis \cite{HW}  yields the implication
"$\Leftarrow$" in \eqref{d-conj}, see \cite[Th. 12.1]{DKKS}.  Developing the techniques in \cite{HW} a
bit further would
yield the more general implication "$\Leftarrow$"  in \eqref{td-conj}.  Let us point out that
we consider the  implication "$\Rightarrow$"  in \eqref{td-conj}  as one
of the most interesting current problems in this area.

Representations of the discrete series feature interesting additional structures.
For instance, for a reductive group Schmid realized the discrete spectrum in $L^2$-Dolbeault cohomology \cite{Schmid}. This was the first of series of realizations of the discrete series representations for reductive Lie groups. Vogan established that the
representations of the discrete series on a symmetric space are cohomologically induced \cite{Vogan}.
It would be interesting to know  for non-symmetric spaces
to which extent $\hat G_{H, {\rm d}}$ consists
of cohomologically induced representations.

\subsection{Methods}
We first describe the idea of proof for Theorem \ref{thm intro} in the case $Z=G$ is a semisimple group.
Let $\pi \in \hat G_{\rm d}$ be a discrete series.
Let $\sigma\in \hat M$ and $\lambda\in \af_\C^*$ be such that there is a quotient
$$
\pi_{\lambda, \sigma}= \Ind_P^G ( \lambda\otimes \sigma) \twoheadrightarrow \pi
$$
of the principal series representation $\Ind_P^G ( \lambda\otimes \sigma)$.
Here induction is normalized and from the left. Such a quotient exists for every irreducible
representation $\pi$ by the  subrepresentation theorem of Casselman.

Let now $v \in \pi_{\lambda, \sigma}^\infty $ be a smooth vector and
let $\oline v$ be its image in $\pi^\infty$.
Further let $\oline \eta$ be any smooth vector in $(\pi^{\vee})^\infty$ where $\pi^{\vee}$ is the dual
representation of $\pi$.  We view $\oline \eta$ as an element of
$ (\pi_{\lambda, \sigma}^{\vee})^\infty= \pi_{-\lambda, \sigma^{\vee}}^\infty$,
denote it then by $\eta$, and   record the relation
$$
m_{\oline v, \oline \eta}(g)
:= \la\oline \eta, \pi(g)^{-1}\oline v\ra  = \la \eta, \pi_{\lambda,\sigma}(g^{-1})v\ra=:m_{v,\eta}(g)
\qquad (g\in G) \, .
$$
We now use the non-compact model for
$\pi_{\lambda, \sigma}$, i.e. $\sigma$-valued functions on $\oline N$ (the opposite of $N$), and let
$v$ be a $\sigma$-valued a test function on $\oline N$.
Let $g=a\in A$. As $v$ is compactly supported on $\oline N$, the functions $\oline n \mapsto
a^{-2\rho} v(a\oline n a^{-1}) $ form a Dirac sequence on $\oline N$ for $a\in A^-$
tending to infinity along a regular ray, and a partial Dirac sequence in case of a semi-regular ray.
Here $A^-=\exp(\af^-)$ with $\af^-\subseteq \af$ the closure of the negative Weyl chamber determined by $N$.
Dirac approximation and appropriate choices of
$\oline v$ and $\oline\eta$ then give a constant
$c=c(\oline v, \oline \eta)\neq 0$ and the asymptotic behavior:
\begin{equation} \label{asymp intro}
m_{\oline v  ,  \oline\eta}(a)\sim c\cdot  a^{-\lambda + \rho}
\qquad (a\in A^-, a\to \infty)\, .
\end{equation}
Strictly speaking, the constant $c$ above also depends on the ray
along which we go to infinity, in case it is not regular.
The asymptotics (\ref{asymp intro}) are motivated by a lemma of Langlands
\cite[Lemma 3.12]{Langlands} which is at the core of the Langlands classification.
This lemma asserts for $K$-finite vectors $v$ and $\eta$,
and for $\lambda$ in the range
of absolute convergence of the long intertwining operator, say $I$, that
$$
c(\oline v, \oline \eta)
=   \la I(v)(e), \eta(e)\ra_\sigma\,.
$$
As our $v$ is compactly supported on
$\oline N$ the integral defining $I(v)$ is in fact
absolutely convergent for every parameter $\lambda$.

As $\pi$ belongs to the discrete series, $m_{\oline v, \oline \eta}$ is square integrable on $G$.
One then
derives from (\ref{asymp intro}) and
the integral formula for the Cartan decomposition $G=KA^-K$ that the parameter $ \lambda$ has to satisfy the strict inequality
\begin{equation}\label{parameter intro one}
\re \lambda\big|_{\af^-\bs \{0\}}>0\, .
\end{equation}

There exists a number $N(G)\in\N$ such that every rank one standard intertwiner
$$
I_\alpha:  \Ind_P^G ( s_\alpha \lambda\otimes s_\alpha\sigma)\to \Ind_P^G ( \lambda\otimes \sigma)
$$
is an isomorphism for $\lambda(\alpha^\vee)\not\in\frac{1}{N(G)}\Z$
(see Proposition \ref {Prop intertwining operators} below).
Suppose  that $\lambda(\alpha^\vee)\not\in \frac{1}{N(G)}\Z$ for some simple root
$\alpha\in \Sigma(\nf, \af)$. Then we obtain an additional quotient morphism
$ \pi_{s_\alpha\lambda, s_\alpha\sigma}\twoheadrightarrow \pi$.
As above this implies
\begin{equation} \label{parameter intro}
\re s_\alpha \lambda\big|_{\af^-\bs \{0\}}>0\,.
\end{equation}

Motivated by (\ref{parameter intro}) we define an equivalence relation on $\af_{\C}^{*}$ in Section \ref{subsection equivalence relations} as follows: $\lambda\sim \mu$ provided $\mu$ is obtained from $\lambda$ by a sequence $\lambda=\mu_{0},\mu_{1},\dots,\mu_{l}=\mu$
such that
\begin{enumerate}[(a)]
\item $\mu_{i+1}=s_{i}(\mu_{i})$ for $s_{i}=s_{\alpha_{i}}$ a simple reflection,
\item $\mu_{i}(\alpha_{i}^{\vee})\notin\frac{1}{N(G)}\Z$.
\end{enumerate}
The equivalence class of $\lambda\in\af_{\C}^{*}$ is denoted $[\lambda]$ and (by slight abuse of terminology introduced in
Section \ref{subsection integral-negative}) we say that $\lambda$ is strictly integral-negative provided all elements of $[\lambda]$ satisfy (\ref{parameter intro}).
In particular we see that any parameter $\lambda$, for which there exists a discrete series representation $(\pi,V)$ and a quotient $\pi_{\lambda,\sigma}\twoheadrightarrow V$, is strictly integral-negative.

Using the geometry of the Euclidean apartment of the Weyl group we show in Section \ref{section negative integral} (Corollary \ref{Cor strictly int-neg parameters}) that there exists an $N=N(\gf)\in \N$ such that for strictly integral-negative parameters $\lambda\in\af_{\C}^{*}$ one has
$$
\lambda(\alpha^{\vee})\in\frac{1}{N}\Z
\qquad(\alpha\in\Sigma)\,.
$$
In particular strictly integral-negative parameters are real and discrete.

For a general real spherical space
$Z=G/H$ we start with a twisted discrete series representation $\pi$ and consider it as a quotient
$ \pi_{\lambda, \sigma}= \Ind_P^G ( \lambda\otimes \sigma) \twoheadrightarrow \pi$ of a principal series representation.
The role of $\oline \eta\in (\pi^{\vee})^\infty$ above is now played by an element
$\oline \eta\in (\pi^{-\infty})^H$ where $\pi^{-\infty}$ refers to the dual of $\pi^\infty$.
We let $\eta$ be the lift of $\oline \eta$ to an element of $(\pi_{\lambda, \sigma}^{-\infty})^H$.

The function
$$
m_{\oline v,\oline \eta}(g):=\oline \eta(\pi(g^{-1}) \oline v) = \eta(\pi_{\lambda, \sigma} (g^{-1}) v)=:m_{v,\eta}(g)
$$
descends to a smooth function on $Z=G/H$ and is referred to as a generalized matrix coefficient.

Now  $\eta$ is supported on various $H$-orbits on $P\bs G$  and we pick one with maximal dimension, say $PxH$ for some $x\in G$.
Here one meets the first serious technical obstruction: Unlike in the symmetric case
(Matsuki \cite{Mat}, Rossmann \cite{Ros}), there is no
explicit description of the $P\times H$ double cosets,
but merely the information that
the number of double cosets is finite \cite{KS}.
However, for computational purposes related to asymptotic analysis it turns out that one can replace
the unknown isotropy algebra $\hf_x:=\Ad(x)\hf$ by its deformation
$$
\hf_{x,X}:= \lim_{t\to \infty} e^{t\ad X} \hf_x  \qquad (X\in \af^- \ \text{regular})\, .
$$
There are only finitely many of those for regular $X$ and they are all $\af$-stable, i.e. nicely lined up
for arguments related to Dirac-compression.  One is then interested in the asymptotics of
$ t\mapsto m_{v,\eta}(\exp(tX)x)$ for appropriately compactly supported $v$.  The main technical result
of this paper is a generalization of (\ref{asymp intro}) in terms of natural  geometric data
related to $\hf_{x,X}$, see Theorem \ref{leading behavior} and Corollary \ref{cor asymp}.
As above it leads to a variant of (\ref{parameter intro one}) in Corollary \ref{Cor negative and integral}
and the final conclusion is derived via our Weyl group techniques from Section \ref{section negative integral}.

\medbreak

{\it Acknowledgement: } We would like to thank Patrick Delorme  for  posing the question about the
spectral gap for twisted discrete series representations, and for explaining to us how to adapt
the work of Sakellaridis and Venkatesh for $p$-adic spherical spaces to real spherical spaces.

\section{Notions and Generalities}\label{Notions}
We write $\N=\{1, 2,  3\ldots\}$ for the set of natural numbers and put $\N_0:=\N\cup\{0\}$.
Throughout this paper we use upper case Latin letters $A,B, C\ldots$ to denote Lie groups and write
$\af, {\mathfrak b}, \cf,\ldots$ for their corresponding Lie algebras.
If $A, B\subseteq G$ are Lie groups, then we write  $N_A(B):=\{ a\in A\mid aBa^{-1}=B\}$ for the normalizer of $B$
in $A$ and likewise we denote by $Z_A(B)$ the centralizer of $B$ in $A$. Correspondingly
if $\af, \mathfrak b\subseteq\gf$ are subalgebras, then we  write $N_\af(\mathfrak b)$ for the normalizer of $\mathfrak b$ in $\af$.

For a real vector space $V$ we write $V_{\C}$ for the complexification $V\otimes \C$ of $V$.

If $L$ is a real reductive Lie group, then we denote by $L_{\rm n}$ the normal subgroup generated
by all unipotent elements of $L$, or, phrased equivalently, $L_{\rm n}$ is the connected subgroup with Lie algebra equal to the direct sum of all non-compact simple ideals of $\lf$.

Let $G$ be an open subgroup of the real points $\mathbf{G}(\R)$ of a reductive algebraic group $\mathbf{G}$ defined over $\R$. Let $\mathbf{H}$ be an algebraic subgroup of $\mathbf{G}$ defined over $\R$ and let $H$ be an open subgroup of $\mathbf{H}(\R)\cap G$.
Define the homogeneous space $Z:=G/H$.
We assume that $Z$ is unimodular, i.e., carries a $G$-invariant positive Radon measure.
Let $z_0:=e\cdot H\in Z$ be the standard base point.

Let $P\subseteq G$ be a minimal parabolic subgroup. We assume that $Z$ is {\it real spherical}, that is,
the action of $P$ on $Z$ admits an open orbit.  After replacing $P$ by a
conjugate we will assume that $P\cdot z_0$ is open in $Z$.
The local structure theorem (see \cite{KKS}) asserts the existence  of a parabolic subgroup
$Q\supseteq P$  with Levi-decomposition $Q=L \ltimes U$ such that:
\begin{align}\label{adapted parabolic}
\nonumber P\cdot z_0 &= Q \cdot z_0\,,\\
Q\cap H &= L \cap H,\\
\nonumber L_{\rm n} &\subseteq L \cap H\,.
\end{align}
We emphasize that the choice of $L$ has to be taken in accordance with the local structure theorem, see
\cite[Remark 2.2]{DKKS}.

Let now $L=K_L A  N_L$ be any Iwasawa-decomposition of $L$ and set $A_H:=A \cap H$ and $A_Z:=A/A_H$.
We note that $A_H$ is connected.  The number $\rank_\R Z:=\dim A_Z$ is an invariant of $Z$ and referred to as the
{\it real rank of $Z$}.

We inflate $K_L$ to a maximal compact subgroup $K\subseteq G$ and set $M:=Z_K(\af)$.
We denote by $\theta$ the Cartan involution on $\gf$ defined by $K$ and set $\oline {\uf}:= \theta(\uf)$. We may and will assume that $A\subseteq P$. Let $P=MAN$ be the corresponding Langlands decomposition of $P$ and define $\oline\nf:=\theta(\nf)$.

\subsection{Spherical roots and the compression cone}
Let $\Sigma=\Sigma(\gf,\af)$ be the restricted root system for the pair $(\gf,\af)$ and
$$
\gf = \af \oplus \mf \oplus \bigoplus_{\alpha\in \Sigma} \gf^\alpha
$$
be the attached root space decomposition.
Write $(\lf \cap \hf)^{\perp_\lf} \subseteq \lf$ for the orthogonal complement of $\lf \cap \hf$ in $\lf$
with respect to a non-degenerate $\Ad(G)$-invariant bilinear form  on $\gf$ restricted to $\lf$.
From $\gf=\qf +\hf=\uf \oplus (\lf\cap\hf)^{\perp_\lf}\oplus \hf$ and $\gf=\qf\oplus \oline{\uf}$ we infer the existence of a linear
map $ T:\oline{\uf}\to \uf \oplus (\lf\cap\hf)^{\perp_\lf}$ such that
$\hf=\lf\cap \hf \oplus \G(T)$ with $\G(T) \subseteq \oline{\uf} \oplus \uf \oplus (\lf\cap\hf)^{\perp_\lf}$ the graph of $T$.

Set $\Sigma_\uf:=\Sigma(\uf,\af)\subseteq \Sigma$.
For $\alpha\in\Sigma_{\uf}$ and $\beta\in\Sigma_{\uf}\cup\{0\}$ we denote by $T_{\alpha,\beta}:\gf^{-\alpha}\to\gf^{\beta}$ the map obtained by restriction of $T$ to $\gf^{-\alpha}$ and projection to $\gf^{\beta}$.
Then
$$
T\big|_{\gf^{-\alpha}} = \sum_{\beta\in\Sigma_{\uf}\cup\{0\}} T_{\alpha,\beta}\,.
$$
Let $\M\subseteq \af^*\bs\{0\}$ be the additive semi-group generated by
$$
\{ \alpha+\beta\mid \alpha\in \Sigma_\uf, \beta\in\Sigma_{\uf}\cup\{0\} \text{ such that } T_{\alpha,\beta}\neq 0\}\,.
$$
We recall from \cite{KK},  Cor. 12.5 and Cor. 10.9,  that the cone generated by $\M$ is simplicial. We fix a set of generators $S$
of this cone with the property $\M\subseteq \N_0[S]$ and refer to $S$
as a set of  {\it (real) spherical roots}.
Note that all elements of $\M$ vanish on $\af_H$ so that we can view $\M$ and $S$ as subsets of
$\af_Z^*$.

We define the {\it compression cone}  by
$$
\af_Z^-:=\{ X\in\af_Z\mid (\forall \alpha\in S) \alpha(X)\leq 0\}
$$
and write $\af_{Z,E}:=\af_Z^- \cap (-\af_Z^-)$ for its edge.   We note that
$$
\# S = \dim \af_Z/\af_{Z,E}\, .
$$

For an $\af$-fixed subspace $\sf$ of $\gf$, we define
$$
\rho(\sf)(X):=\frac{1}{2}\tr(\ad(X)\big|_{\sf})\qquad(X\in \af)\,.
$$
We write $\rho_{P}$ for $\rho(\pf)$ and $\rho_{Q}$ for $\rho(\qf)$.
Recall that the unimodularity of $Z$ implies that $\rho_Q|_{\af_H}=0$, see \cite[Lemma 4.2]{KKSS2}.

Let $\Pi\subseteq \Sigma^+$ be the set of simple roots.
We let $\af^\pm:=\{ X \in \af \mid (\forall \alpha\in \Pi) \  \pm\alpha(X)\geq 0\}$ and write
$\af^{--}$ for the interior (Weyl chamber) of $\af^{-}$.

We write $p: \af\to \af_Z$ for the projection and set $\af_E:=p^{-1}(\af_{Z,E})$ and $A_{E}=\exp(\af_{E})$.
Set $\hat H=HA_{E}$ and note that $\hat H$ normalizes $H$.
Obviously $\hat H$ is real spherical as well.
Finally, we define $\hat Z:=G/\hat H$.

\subsection{The normalizer of a real spherical subalgebra}

\begin{lemma}\label{normalizer lemma}
Let $\hf\subseteq \gf$ be a real spherical subalgebra. Then the following assertions hold:
\begin{enumerate}[(i)]
\item\label{normalizer lemma part 2} $N_\gf(\hf)= \hat \hf +\hat \mf$ with $\hat\mf\subseteq \mf$, the sum not necessarily being direct.
\item\label{normalizer lemma part 3} $\hat{\hat\hf}=\hat \hf$.
\item\label{normalizer lemma part 4} $[N_\gf(\hf)]_{\rm n}=\hf_{\rm n}$, i.e. every $\ad_\gf$-nilpotent element in $N_\gf(\hf)$ is contained in $\hf$.
\end{enumerate}
\end{lemma}

\begin{proof}
For (\ref{normalizer lemma part 2}) see \cite[(5.10)]{KKSS}.
Lemma 4.1 in \cite{KKS} implies (\ref{normalizer lemma part 3}). Finally, (\ref{normalizer lemma part 4}) follows from (\ref{normalizer lemma part 2}).
\end{proof}

\section{Twisted discrete series as quotients of principal series}
\subsection{The spherical subrepresentation theorem}
For a Harish-Chandra module $V$, we denote by $V^{\infty}$ the unique smooth moderate growth Fr\'echet globalization and by $V^{-\infty}$ the continuous dual of $V^{\infty}$.
If $\eta\in (V^{-\infty})^{H}\setminus\{0\}$, then the pair $(V,\eta)$ is called a spherical pair.

For a Harish-Chandra module $V$ we denote by $\tilde V$ its contragredient or dual Harish-Chandra module,
that is, $\tilde V$ consist of the $K$-finite vectors in the algebraic dual $V^*$ of $V$.  Further we denote by
$\oline V$ the conjugate Harish-Chandra module, that is, $\oline V=V$ as $\R$-vector space but with the conjugate complex
multiplication.   We recall that $\tilde V = \oline V$ in case $V$ is unitarizable. In particular if $(V,\eta)$ is a
spherical pair with $V$ unitarizable, then so is $(\oline V, \oline \eta)$ with $\oline \eta(v):= \oline{\eta(v)}$.

Associated to $\eta \in (V^{-\infty})^H$  and $v\in V^\infty$ we find the generalized matrix coefficient on $Z$
$$
m_{v,\eta}(z):= \eta (g^{-1}v) \qquad (z=gH\in Z)\,,
$$
which defines a smooth function on $Z$.  If $v\in V$ then $m_{v,\eta}$ admits a convergent
power series expansion (cf. \cite{KKS2}, Sect. 6):
$$
m_{v,\eta}(ma\cdot z_0) =  \sum_{\mu\in \E} \sum_{\alpha \in \N_0[S]}  c_{\mu, v}^\alpha(m;\log a) a^{\mu + \alpha}
\qquad (a\in A_Z^-, m \in M)\, .
$$
Here $\E\subseteq \af_{Z,\C}^*$ is a finite set of leading exponents only depending on $(V,\eta)$; the term "leading" refers
to the following relation:  for all $\mu, \mu'\in \E$, $\mu\neq \mu'$ one has $\mu \not \in \mu' + \N_0[S]$.
Further, for each $\mu\in\E$, $\alpha\in \N_0[S]$ and $v\in V$, the assignment
$$
c_{\mu, v}^\alpha:  M \times \af_Z \to \C, \ \ (m, X) \to c_{\mu,v}^\alpha(m; X)
$$
is polynomial in $X$ and $M$-finite.
Moreover, for each $\mu\in\E$ there exists a $v\in V$ such that $c_{\mu,v}^{0}\neq 0$.
The $M$-types which can occur are those obtained
from branching the $K$-module $\Span_{\C}\{K\cdot v\}$ to $M$.  The degrees of the polynomials
are uniformly bounded and we set $d_\mu:=\max_{v\in V} \deg c_{\mu, v}^0\in \N_0$.

Let us set $A_L:=Z(L)\cap A$.   Then $L =M_L A_L $ for a complementary reductive subgroup $M_L \subseteq L$.
For a unitary representation $(\sigma, V_\sigma)$ of $M_L$ and $\lambda \in \af_{L,\C}^*$ we denote by $\Ind_{\oline Q}^G (\lambda \otimes \sigma)$ the normalized left induced representation.
Note that the elements $v\in\Ind_{\oline Q}^G (\lambda \otimes \sigma)$  are $K$-finite functions $v: G \to V_\sigma$ which satisfy
$$
v(\oline u m a g)  =  a^{\lambda-\rho_Q}  \sigma(m) v(g)
$$
for all $g\in G$, $a\in A_{L}$, $u\in \oline U$ and $m\in M_L$.

Note that $A_{L}A_{H}=A$ and that therefore there exists a natural inclusion $\af_{Z}^{*}\hookrightarrow \af_{L}^{*}$.
The representations $\Ind_{\oline Q}^G (\lambda \otimes \sigma)$ are related to spherical representation theory as follows.

\begin{lemma}  Let $(V,\eta)$ be a spherical pair with $V$ irreducible and
$\mu \in \af_{Z}^{*}\subseteq\af_{L}^{*}$ a leading exponent.  Then there exist
an irreducible finite dimensional representation $\sigma$ of $M_L$ with a $(M_L\cap H)$-fixed vector, and
an embedding of Harish-Chandra modules:
\begin{equation} \label{disc-embedQ}
 V \hookrightarrow \Ind_{\oline Q}^G ((-\mu + \rho_Q)\otimes \sigma)\, .
\end{equation}
\end{lemma}

\begin{proof}  This is implicitly contained in \cite{KS}, Section 4. We confine ourselves with a sketch of the argument.

Recall $d_\mu$ and fix a basis $X_1, \ldots, X_n $ of $\af_Z$. For ${\bf m} \in \N_0^n$,
$X=\sum_{j=1}^n x_j X_j\in \af_Z$ we set $X^{\bf m}:= x_1^{m_1}\cdot \ldots\cdot x_n^{m_n}$.  Then
$$
c_{\mu, v}^0(m;X)=\sum_{|{\bf m }|\leq d_\mu}   c_{\mu, v}^{\bf m}(m) X^{\bf m}
\qquad(m\in M)
$$
where $c_{\mu, v}^{\bf m}$ is an $M$-finite function. Fix now $\sigma\in \hat M$ and ${\bf m}\in \N_0^n$ with
$|{\bf m}|=d_\mu$ such that the $\sigma$-isotypical part of  $c_{\mu, v}^{\bf m}(m)$ is non-zero. This
gives rise to a non-trivial $M$-equivariant map
$$
V \to V_\sigma,  \ \ v\mapsto   c_{\mu, v}^{\bf m}[\sigma]\, .
$$
It is easy to see that $(\lf\cap \hf  +\oline \uf)V$ is in the kernel of this map.
Note that $M\cap M_{L,n}$ is a normal subgroup of $M$ that is contained in $M\cap H$. From
the fact that $\sigma$ admits a non-zero $M\cap H$-fixed vector it follows that
$\sigma\big|_{M\cap M_{L,n}}$ is trivial. We may thus extend $\sigma$ to a representation of
$M_L\simeq M\underset{M\cap M_{L,n}}{\rtimes}M_{L,n}$
by setting $\sigma\big|_{M_{L,n}}=1$.
The assertion now follows from Frobenius reciprocity.
\end{proof}

\subsection{Discrete series and twisted discrete series}
For $\chi\in(\hat\hf/\hf)^{*}_{\C}\simeq\af_{Z,E,\C}^*$ we define the space of functions
$$
C_{c}(\hat Z;\chi)
:=\big\{\phi\in C_{c}(G):\phi(\dotvar ha)=a^{-\chi}\phi \text{ for all }a\in A_{E}, h\in H\big\}\,.
    $$
We call $\chi\in(\hat\hf/\hf)^{*}_{\C}$ {\it normalized unitary} if
$$
\re\chi\big|_{\af_{E}}
=-\rho_{Q}\big|_{\af_{E}}\,.
$$
Let $\Delta_{\hat Z}$ be the modular function of $\hat Z$. By \cite[Lemma 8.4]{KKS2} we have
\begin{equation}\label{Expression for Delta_hatZ}
\Delta_{\hat Z}(ha)=a^{-2\rho_{Q}}\qquad(h\in H, a\in A_{E})\,.
\end{equation}
For $g\in G$, let $l_{g}$ denote left multiplication by $g$. Let $\Omega\in\bigwedge^{\dim Z}(\gf/\hf)^{*}\setminus\{0\}$. If $\chi\in(\hat\hf/\hf)^{*}_{\C}$ is normalized unitary, then it follows that for all $\phi,\psi\in C_{c}(\hat Z;\chi)$ the density
$$
|\Omega|_{\phi,\psi}:G\ni g\mapsto \phi(g)\oline{\psi(g)}\big(T_{g}l_{g^{-1}}\big)^{*}|\Omega|
$$
factors to a smooth density on $\hat Z$, and the bilinear form
$$
C_{c}(\hat Z;\chi)\times C_{c}(\hat Z;\chi)\to\C;
\qquad (\phi,\psi)\mapsto \int_{\hat Z} |\Omega|_{\phi,\psi}
$$
is an inner product.
We write $L^{2}(\hat Z;\chi)$ for the Hilbert completion of $C_{c}(\hat Z;\chi)$ with respect to this inner product. Note that the inner product is invariant under the left regular action of $G$ and thus $L^{2}(\hat Z;\chi)$ equipped with the left-regular representation is a unitary representation of $G$.

\begin{definition}
If $\chi\in(\hat\hf/\hf)^{*}_{\C}$ is normalized unitary, then we say that the spherical pair $(V,\eta)$ belongs to the {\it$\chi$-twisted discrete series for $Z$} provided that $V$ is irreducible, $\pi^{\vee}(Y)\eta=-\chi(Y)\eta$ for all $Y\in\hat\hf$, and $m_{v,\eta}\in L^2(\hat Z; \chi)$ for all $v\in V^\infty$. Furthermore, we say that $(V,\eta)$ belongs to the {\it twisted discrete series} for $Z$ if $(V,\eta)$ belongs to the $\chi$-twisted discrete series for some normalized unitary $\chi$. Finally we say that $(V,\eta)$ belongs to the {\it discrete series} for $Z$ provided that $V$ is irreducible and $m_{v,\eta}\in L^2(Z)$ for all $v\in V^\infty$.
\end{definition}

\begin{lemma}\label{DS implies H=hat H}
If there exits a spherical pair $(V,\eta)$ belonging to the discrete series for $Z$, then $H=\hat H= HA_{E}$. Hence $\hat\hf/\hf=0$ and therefore the discrete series for $Z$ coincide with the $0$-twisted discrete series for $Z$.
\end{lemma}

\begin{proof}
Let $(V,\eta)$ be a spherical pair belonging to the discrete series for $Z$.
The right-action of $A_{E}$ commutes with the left-action of $G$ on $L^{2}(Z)$, and thus induces a natural action of $A_{E}$ on $(V^{-\infty})^{H}$. By \cite{KoOs} and \cite{KS2} the space $(V^{-\infty})^{H}$ is finite dimensional.
We may therefore assume that $\eta$ is a joint-eigenvector for the right-action of $A_{E}$, i.e., the generalized matrix coefficients of $V$ satisfy
$$
m_{v,\eta}(gha)=a^{-\chi}m_{v,\eta}(g)\qquad(g\in G, h\in H, a\in A_{E})
$$
for some normalized unitary $\chi\in(\hat\hf/\hf)^{*}_{\C}$. Let $A_{0}$ be a subgroup of $A$ such that $A_{0}\times A_{E}\simeq A$.
If $g\in G$ and $m_{v,\eta}(g\cdot z_{0})\neq 0$, then, if the Haar measures are properly normalized,
\begin{align*}
\int_{Z}|m_{v,\eta}(z)|^{2}\,dz
&\geq \int_{gQ\cdot z_{0}}|m_{v,\eta}(z)|^{2}\,dz\\
&=\int_{U}\int_{M}\int_{A_{0}}\int_{A_{E}/(A\cap H)}(a_{0}a_{E})^{-2\rho_{Q}}|m_{v,\eta}(gnma_{0}a_{E})|^{2}\,da_{E}\,da_{0}\,dm\,dn\\
&= \int_{U}\int_{M}\int_{A_{0}}\int_{A_{E}/(A\cap H)}a_{0}^{-2\rho_{Q}}|m_{v,\eta}(gnma_{0})|^{2}\,da_{E}\,da_{0}\,dm\,dn\,.
\end{align*}
Clearly the last repeated integral can only be absolutely convergent if $A_{E}/(A\cap H)$ has finite volume, or equivalently if $A_{E}=A\cap H$.
\end{proof}

We recall from Section 8 in \cite{KKS2} that  $(V,\eta)$ belongs to the twisted
discrete series for $Z$ only if the conditions
\begin{align}
\label{DS1} &(\re \mu - \rho_Q)|_{\af_Z^- \bs \af_{Z,E}}<0\,,\\
\label{DS2} &(\re \mu-\rho_Q)|_{\af_{Z,E}}=0
\end{align}
hold for all leading exponents $\mu$. Moreover,
\begin{equation}\label{DS3}
 \mu|_{\af_{Z,E}}=-\chi
\end{equation}
when $(V,\eta)$ belongs to the $\chi$-twisted discrete series.
Note that (\ref{DS1}) implies (\ref{DS2}) unless $\af_{Z,E}=\af_Z$.

\subsection{Quotient morphisms}
It is technically easier to work with representations induced from the minimal
parabolic $P$.  Set $\rho_P^Q=\rho_P -\rho_Q$ and observe that there is a natural
inclusion
$$
\Ind_{\oline Q}^G (\lambda \otimes \sigma)\to \Ind_{\oline P}^G( (\lambda +\rho_P^Q)\otimes \sigma|_M)\, .
$$
In particular (\ref{disc-embedQ}) yields
\begin{equation} \label{disc-embedP}
V \hookrightarrow \Ind_{\oline P}^G ((-\mu + \rho_P)\otimes \sigma)\,,
\end{equation}
where we allowed ourselves to write $\sigma$ for $\sigma|_M$.

For general $\Ind_{\oline P}^G (\lambda\otimes \sigma)$ we record that
its dual representation is given by $\Ind_{\oline P}^G (-\lambda\otimes \sigma^{\vee})$.
The natural pairing between these
two representations is given as follows in the non-compact picture:
$$
v^{\vee}(v)=\int_N v^{\vee}(n) \big(v(n)\big)\, dn
$$
for $v^{\vee} \in \Ind_{\oline P}^G (-\lambda\otimes \sigma^{\vee})$ and $v\in \Ind_{\oline P}^G (\lambda\otimes \sigma)$.

Let now $(V,\eta)$ be an irreducible spherical pair belonging to the twisted discrete series.  Then $\oline \mu$ is a
leading exponent for the dual pair $(\oline V, \oline\eta)$.
By applying (\ref{disc-embedP}) to $\oline V$ we embed
$$
\oline V \hookrightarrow   \Ind_{\oline P}^G ((-\oline \mu +\rho_P)\otimes \sigma^{\vee})\,.
$$
Dualizing this inclusion we obtain the quotient morphism
\begin{equation}\label{quotient morphism}
\Ind_{\oline P}^G ((\oline \mu  -\rho_P)\otimes \sigma) \twoheadrightarrow V\,.
\end{equation}

In view of  the $P\times H$-geometry of $G$ it is a bit inconvenient to work with representations
induced from the left by the opposite parabolic $\oline P$.  We can correct this by employing the long Weyl
group element $w_0\in W=W(\gf,\af)$, which maps $\oline P$ to $P$.
This gives us for every $\lambda\in \af^*_\C$ and $\sigma\in\hat M$ an isomorphism
\begin{equation}\label{w0 intertwining}
\Ind_P^{G} (\lambda\otimes \sigma) \to \Ind_{\oline P}^G (w_0\lambda\otimes w_0\sigma);
\ \ v\mapsto v(w_0\dotvar)\,,
\end{equation}
where $w_0\sigma:=\sigma\circ w_0\in\hat M$. With proper choices of
$\lambda$ and $\sigma$ we obtain from
(\ref{w0 intertwining}) and (\ref{quotient morphism})
a quotient morphism of $\Ind_P^G (\lambda\otimes\sigma)$ onto $V$.

We write now $\pi_{\lambda,\sigma}$ for $\Ind_P^G(\lambda\otimes\sigma)$ and record that functions
$v\in \pi_{\lambda,\sigma}$ feature the transformation property
\begin{equation} \label{P-trans}
v (man g)  = a^{\lambda+\rho_P} \sigma(m) v(g)\, .
\end{equation}

To summarize our discussion so far:

\begin{lemma}\label{disc cover}  Let $(V,\eta)$ be a twisted discrete series representation for $Z$ and $\mu\in\af_\C^*$ a leading exponent. Then there exists a $\sigma\in \hat M$ and a surjective quotient morphism $\pi_{\lambda,\sigma} \twoheadrightarrow V$
with $\lambda= w_0 \oline \mu +\rho_P$.
\end{lemma}

We write $\pi_{\lambda,\sigma}^{\infty}$ for the smooth Fr\'echet globalization of moderate growth.
In the sequel we will model $\pi_{\lambda, \sigma}^{\infty}$ on all smooth functions which satisfy (\ref{P-trans}).

\section{Generalized volume growth}
\subsection{Limiting subalgebras}
Define {\it order-regular} elements in $\af^{--}$ by
$$
\af_{\rm o-reg}^{--}:=\{ X\in \af^{--}\mid \alpha(X)\neq \beta(X),  \alpha,\beta\in \Sigma,\alpha\neq \beta\}\, .
$$

In this and the next section we will make heavy use of certain limits of subspaces of $\gf$ in the Grassmannian. In the following lemma we collect the important properties of such limits.

\begin{lemma}\label{Lemma Limits of subspaces}
Let $E$ be a subspace of $\gf$ and let $X\in\af$. Then the limit
$$
E_{X}:=\lim_{t\to\infty}\Ad\big(\exp(tX)\big)E\,,
$$
exists in the Grassmannian. If $\lambda_{1}<\lambda_{2}<\dots<\lambda_{n}$ are the eigenvalues and $p_{1},\dots,p_{n}$ the corresponding projections onto the eigenspaces $V_{i}$ of $\ad(X)$, then $E_{X}$ is given by
\begin{equation}\label{eq formula for E_X}
E_{X}
=\bigoplus_{i=1}^{n}p_{i}\big(E\cap \bigoplus_{j=1}^{i}V_{j}\big).
\end{equation}
The following hold.
\begin{enumerate}[(i)]
\item\label{Lemma Limits of subspaces - item 1} If $E$ is a Lie subalgebra of $\gf$, then $E_{X}$ is a Lie subalgebra of $\gf$.
\item\label{Lemma Limits of subspaces - item 2} If $X\in\af^{--}$, then $\big(\Ad(man)E\big)_{X}=\Ad(ma)\big(E_{X}\big)$ for all $m\in M$, $a\in A$ and $n\in N$. Moreover, if $X$ is order-regular, then $E_{X}$ is $A$-stable.
\item\label{Lemma Limits of subspaces - item 3} Let $\Cc$ be a connected component of $\af_{\rm o-reg}^{--}$. Then $\big(E_{X}\big)_{Y}=E_{Y}$ for all $X\in\overline{\Cc}$ and $Y\in\Cc$. In particular, if $X,Y\in\Cc$, then $E_{X}=E_{Y}$.
\item\label{Lemma Limits of subspaces - item 4}  If $X,X'\in\af^{--}$, then $\af\cap E_{X}=\af\cap E_{X'}$.
\end{enumerate}
\end{lemma}

\begin{proof}
Let $k=\dim(E)$ and let $\iota:\Gr(\gf,k)\hookrightarrow \P(\bigwedge^{k}\gf)$ be the Pl{\"u}cker embedding, i.e., $\iota$ is the map given by
\begin{equation}\label{eq Plucker embedding}
\iota\big(\mathrm{span}(v_{1},\dots,v_{k})\big)
=\R (v_{1}\wedge\cdots \wedge v_{k}).
\end{equation}
The map $\iota$ is a diffeomorphism onto a compact submanifold of $\P(\bigwedge^{k}\gf)$.
The map $\ad(X)$ acts diagonalizably on $\bigwedge^{k}\gf$, say with eigenvalues $\mu_{1}<\mu_{2}<\dots<\mu_{m}$.
Let $\xi\in\bigwedge^{k}\gf\setminus\{0\}$ be so that $\iota(E)=\R\xi$. We decompose $\xi$ into eigenvectors for $\ad(X)$ as
$$
\xi
=\sum_{i=1}^{m}\xi_{i}\,,
$$
where $\xi_{i}$ is an eigenvector of $\ad(X)$ with eigenvalue $\mu_{i}$. Now
$$
\Ad\big(\exp(tX)\big)(\R\xi)
=\R\big(\sum_{i=1}^{m}e^{t\mu_{i}}\xi_{i}\big)\,.
$$
Let $1\leq k\leq m$ be the largest number so that $\xi_{k}\neq 0$. Then $\Ad\big(\exp(tX)\big)(\R\xi)$ converges for $t\to\infty$ to $\R\xi_{k}$. Let $E_{X}=\iota^{-1}(\R\xi_{k})$. Since $\iota$ is a diffeomorphism, $\Ad\big(\exp(tX)\big)E$ converges to $E_{X}$ for $t\to\infty$.

We move on to prove (\ref{eq formula for E_X}).
For $1\leq i\leq n$ we define
$$
E_{i}
:=E\cap \bigoplus_{j=1}^{i}V_{j}\,.
$$
We will prove with induction that for every $1\leq i\leq n$
\begin{equation}\label{eq formula for E_X - induction hypothesis}
\big(E_{i}\big)_{X}
=\bigoplus_{j=1}^{i}p_{j}\big(E_{j}\big)\,.
\end{equation}
Clearly $E_{1}=E\cap V_{1}$ is stable under the adjoint action of $X$, and hence $(E_{1})_{X}=E_{1}$. This proves (\ref{eq formula for E_X - induction hypothesis}) for $i=1$. Assume that $(\ref{eq formula for E_X - induction hypothesis})$ holds for some $i$. We claim that
\begin{equation}\label{eq sum p_jE_j subseteq (E_i+1)_X}
\bigoplus_{j=1}^{i+1}p_{j}\big(E_{j}\big)
\subseteq \big(E_{i+1}\big)_{X}\,.
\end{equation}
In view of the induction hypothesis it suffices to prove that $p_{i+1}(Y)\in \big(E_{i+1}\big)_{X}$ for every $Y\in E_{i+1}\setminus E_{i}$.
We decompose $Y$ as
$$
Y
=\sum_{j=1}^{i+1}p_{j}(Y)\,.
$$
Then $p_{i+1}(Y)\neq 0$ and thus
$$
(\R Y)_{X}
=\lim_{t\to\infty}\R\big(\sum_{j=1}^{i+1}e^{t\lambda_{i}}p_{j}(Y)\big)
=\R p_{i+1}(Y)\,.
$$
This shows that $p_{i+1}(Y)\in (E_{i+1})_{X}$. Therefore, the inclusion (\ref{eq sum p_jE_j subseteq (E_i+1)_X}) holds.
In fact, equality holds because the dimensions agree.
This proves (\ref{eq formula for E_X}).

Observe that $[E,E]\subseteq E$ is a closed condition in the Grassmannian. Therefore, the set of Lie subalgebras in the Grassmannian is a closed set. It follows that $E_{X}$ is a Lie subalgebra if $E$ is a Lie subalgebra. This proves (\ref{Lemma Limits of subspaces - item 1}).

Assume that $X\in\af^{--}$.
If $n\in N$, then $\exp(tX)n\exp(tX)^{-1}$ converges to $e$ for $t\to\infty$. Now
\begin{align*}
\big(\Ad(man)E\big)_{X}
&=\lim_{t\to\infty}\Ad(ma)\Ad\big(\exp(tX)n\exp(tX)^{-1}\big)\Ad\big(\exp(tX)\big)E\\
&=\Ad(ma)\big(E_{X}\big).
\end{align*}
If $X\in\af_{\rm{o-reg}}^{--}$, then the eigenvalues $\{\alpha(X):\alpha\in\Sigma\cup\{0\}\}$ of $\ad(X)$ are in bijection with $\Sigma\cup\{0\}$. Therefore, all projections $p_{i}$ in (\ref{eq formula for E_X}) are projections onto $\af$-eigenspaces, namely the root spaces and $\mf\oplus\af$. This implies that $E_{X}$ is $A$-stable. This proves (\ref{Lemma Limits of subspaces - item 2}).

We move on to prove (\ref{Lemma Limits of subspaces - item 3}). It follows from (\ref{eq formula for E_X}) that for every $X\in\af^{-}$ the limit $E_{X}$ is spanned by the limits $L_{X}$ of the lines $L$ in $E$. Hence we may assume that $E$ is $1$-dimensional. Let $X\in\overline{\Cc}$ and $Y\in\Cc$. For $\alpha\in\Sigma\cup\{0\}$ we define $p_{\alpha}$ to be the projection $\gf\to\gf_{\alpha}$ along the root space decomposition.
Let $\alpha_{0}\in\Sigma\cup\{0\}$ be so that $\alpha_{0}(Y)$ is maximal among the numbers $\alpha(Y)$ with $\alpha\in\Sigma\cup\{0\}$ for which $p_{\alpha}(E)\neq\{0\}$. By (\ref{eq formula for E_X}) we have $E_{Y}=p_{\alpha_{0}}(E)$.
Since $Y\in\Cc$ and $X\in\overline{\Cc}$ we have $\alpha(X)\geq\beta(X)$ if $\alpha(Y)>\beta(Y)$. In particular the largest eigenvalue of $\ad(X)$ that appears in $E$ is equal to $\alpha_{0}(X)$. The projection onto the eigenspace of $\ad(X)$ with eigenvalue $\alpha_{0}(X)$ is given by
$$
\sum_{\substack{\alpha\in\Sigma\cup\{0\}\\\alpha(X)=\alpha_{0}(X)}}p_{\alpha}\,.
$$
Therefore,
$$
E_{X}
=\Big(\sum_{\substack{\alpha\in\Sigma\cup\{0\}\\\alpha(X)=\alpha_{0}(X)}}p_{\alpha}\Big)(E)\,,
$$
and hence
$$
(E_{X})_{Y}
=p_{\alpha_{0}}\bigg(\Big(\sum_{\substack{\alpha\in\Sigma\cup\{0\}\\\alpha(X)=\alpha_{0}(X)}}p_{\alpha}\Big)(E)\bigg)
=p_{\alpha_{0}}(E)
=E_{Y}\,.
$$
If $X,Y\in\Cc$, then by (\ref{Lemma Limits of subspaces - item 2}) the space $E_{X}$ is $\af$-stable and therefore $(E_{X})_{Y}=E_{X}$.
This proves (\ref{Lemma Limits of subspaces - item 3}).

Finally we prove (\ref{Lemma Limits of subspaces - item 4}). Let $X\in\af^{--}$. Let $p_{\mf}$, $p_{\af}$ be the projections $\gf\to\mf$ and $\gf\to\af$, respectively, along the Bruhat decomposition. Since $X$ is regular, it follows from (\ref{eq formula for E_X}) that
$$
(\mf\oplus\af)\cap E_{X}
=(p_{\mf}+p_{\af})\big((\mf\oplus\af\oplus\nf)\cap E\big)\,.
$$
Clearly $p_{\af}\big((\af\oplus\nf)\cap E\big)\subseteq \af\cap E_{X}$. Moreover, if $Y\in \af\cap E_{X}$ and $Y'\in (\mf\oplus\af\oplus\nf)\cap E$ is so that $(p_{\mf}+p_{\af})(Y')=Y$, then $p_{\mf}(Y')=0$. Hence $Y\in p_{\af}\big((\af\oplus\nf)\cap E\big)$. It follows that
$$
p_{\af}\big((\af\oplus\nf)\cap E\big)
=\af\cap E_{X}\,.
$$
The left-hand side is independent of $X$.
\end{proof}

Let $\Cc$ be a connected component of $\af_{\rm o-reg}^{--}$.
If $X\in\Cc$, then in view of (\ref{Lemma Limits of subspaces - item 3}) in Lemma \ref{Lemma Limits of subspaces} the space $E_{X}$ does not depend on the specific choice of $X$. Therefore for every subspace $E$ of $\gf$ we may define
$$
E_{\Cc}
:=E_{X}
\qquad(X\in\Cc)\,.
$$
Let $x\in G$.
We define the following spaces.
First set
$$
\hf_{\Cc,x}
:=\big(\Ad(x)\hf\big)_{\Cc}\,.
$$
Observe that by (\ref{Lemma Limits of subspaces - item 2}) in Lemma \ref{Lemma Limits of subspaces}
\begin{equation}\label{hf_C,x invariance}
\hf_{\Cc,manxh}=\Ad(m)\hf_{\Cc,x}\qquad\big(m\in M, a\in A, n\in N, h\in H)\,.
\end{equation}
We define
$$
\af_{x}:=\hf_{\Cc,x}\cap\af\,.
$$
In view of Lemma \ref{Lemma Limits of subspaces}(\ref{Lemma Limits of subspaces - item 4}) this space does not depend on $\Cc$.
Note that (\ref{hf_C,x invariance}) implies that $\af_{x}$ only depends on the double coset $PxH\in P\bs G/H$, not on the representative $x\in G$ for that coset.
We further define the $\af$-stable subalgebras
$$
\oline\nf_{\Cc,x}:=\hf_{\Cc,x}\cap \oline\nf,
\qquad \uf_{\Cc,x}:=\hf_{\Cc,x}\cap \nf\,.
$$
Since $\hf_{\Cc,x}$ is $\af$-stable, it follows that
\begin{equation}\label{eq hf_C,x decomposition}
\hf_{\Cc,x}=\oline\nf_{\Cc,x}\oplus \big((\mf\oplus\af)\cap\hf_{\Cc,x}\big)\oplus\uf_{\Cc,x}\,.
\end{equation}
Finally we choose $\oline\nf_{\Cc}^{x}$ and $\uf_{\Cc}^{x}$ to be $\af$-stable complementary subspaces to $\oline\nf_{\Cc,x}$ in $\oline\nf$ and $\uf_{\Cc,x}$ in $\nf$, respectively, so that
\begin{equation}\label{eq oline nf and nf decomp}
\oline\nf=\oline\nf_{\Cc,x}\oplus\oline\nf_{\Cc}^{x},\qquad
\nf=\uf_{\Cc,x}\oplus\uf_{\Cc}^{x}\,.
\end{equation}

\begin{lemma}\label{Lemma Ad(x)hf+pf+nf^x=gf}
For every $x\in G$
$$
\gf
=(\Ad(x)\hf+\pf)\oplus\oline\nf_{\Cc}^{x}\,.
$$
\end{lemma}

\begin{proof}
Let $X\in\Cc$.
In view of (\ref{eq oline nf and nf decomp}) and (\ref{eq hf_C,x decomposition}) we have
\begin{equation}\label{eq gf=hf_C+pf+nf^x}
\gf
=\oline\nf_{\Cc,x}\oplus \pf\oplus\oline\nf_{\Cc}^{x}
=(\hf_{\Cc,x}+\pf)\oplus\oline\nf_{\Cc}^{x}\,.
\end{equation}
If $\gf\neq(\Ad(x)\hf+\pf)+\oline\nf_{\Cc}^{x}$, then also
$$
\gf
\neq\Ad(a)\big(\Ad(x)\hf+\pf+\oline\nf_{\Cc}^{x}\big)
=(\Ad(ax)\hf+\pf)+\oline\nf_{\Cc}^{x}
$$
for every $a\in A$. This would imply that the limit of $(\Ad(\exp(tX)x)\hf+\pf)+\oline\nf_{\Cc}^{x}$ for $t\to\infty$ is a proper subspace of $\gf$. This in turn would contradict (\ref{eq gf=hf_C+pf+nf^x}).
Therefore, $\gf=(\Ad(x)\hf+\pf)+\oline\nf_{\Cc}^{x}$.
Moreover, it follows from (\ref{eq formula for E_X}) that $\pf\cap\hf_{\Cc,x}=\pf\cap\Ad(x)\hf$, and hence
$$
\dim(\hf_{\Cc,x}+\pf)
=\dim(\Ad(x)\hf+\pf)\,.
$$
Therefore, by comparing with (\ref{eq gf=hf_C+pf+nf^x}) we see that the sum $(\Ad(x)\hf+\pf)+\oline\nf_{\Cc}^{x}$ is direct.
\end{proof}

\subsection{Volume-weights}
We recall the volume-weight function on $Z$
$$
\v(z):=\vol_Z (Bz)\qquad (z\in Z)\,,
$$
where $B$ is some compact neighborhood of $e$ in $G$.
We refer to Appendix A for the properties of volume-weights.
The volume weight naturally shows up in the treatment of twisted discrete series representations.

The following proposition is a direct corollary of the invariant Sobolev lemma in Appendix A.

\begin{prop}
Let $(V,\eta)$ be a spherical pair corresponding to a twisted discrete series representation.
Then
\begin{equation} \label{DS-sup}
\sup_{z\in Z} |m_{v,\eta}(z)|  \v(z)^{\frac{1}{2}}
<\infty\,.
\end{equation}
Moreover, if $(z_n)_{n\in \N}$ is a sequence in $Z$ such that its image in $\hat Z$ tends to infinity, then
\begin{equation} \label{limit bound}
\lim_{n\to \infty}  |m_{v,\eta}(z_n)|  \v(z_n)^{\frac{1}{2}} =0\,.
\end{equation}
\end{prop}

The basic asymptotic behavior of $\v$ on the compression cone is
\begin{equation} \label{reg growth}
\v(a\cdot z_0)\asymp a^{-2\rho_Q} \quad (a \in A_Z^-)\, .
\end{equation}
See \cite[Proposition 4.3]{KKSS2}.
We investigate now the growth of $\v$ with the base point $z_0$ shifted by an element $x\in G$, i.e., we investigate how  $\v(ax\cdot z_{0})$  grows for $a\in A^-$.

Recall the parabolic subgroup $Q=LU$ from (\ref{adapted parabolic}). For $x=e$, we have $\hf_{\Cc,e}=(\lf\cap \hf)\oplus\oline\uf$, and thus
$$
\rho(\hf_{\Cc,e})=-\rho_{Q}\,.
$$
Hence the following proposition is a partial generalization of the lower bound in (\ref{reg growth}) for shifted base points.

\begin{prop}\label{vol growth}  Let $x\in G$ and $X\in \af^{-}$. Let $\Cc$ be a connected component of $\af_{\rm o-reg}^{--}$ such that $X\in\oline\Cc$.  Then there exists a $C>0$ such that
$$
\v(\exp(tX) x\cdot z_0)\geq C  e^{2t\rho(\hf_{\Cc,x})(X)}\qquad (t\geq 0)\,.
$$
\end{prop}

\begin{proof}
Set $\oline N^x =\exp(\oline \nf_{\Cc}^x)$. Since $\exp:\oline\nf\to\oline N$ is a polynomial isomorphism the group $\oline N^x$ is an affine subvariety of $\oline N$.
Define an affine subvariety of $N$ by $U^x:=\exp(\uf_{\Cc}^x)$. Let $\af^{x}$ be the orthogonal complement of
$\af_{x}$ in $\af$ and set $A^{x}:=\exp(\af^{x})$.
Further let $X_1, \ldots, X_k$ be a basis of a subspace in $\mf$ which is complementary to $p_{\mf}(\hf_{\Cc,x})$ in $\mf$, where $p_{\mf}$ is the projection $\gf\to\mf$ along the Bruhat decomposition. We may assume in addition that the $X_{j}$ are so that $M_j:=\exp(\R X_j)\simeq \R/\Z$. Now
$$
\af\oplus\mf
= \big((\mf\oplus\af)\cap \hf_{\Cc,x}\big)\oplus\af^{x} \oplus \bigoplus_{j=1}^k \R X_j\,.
$$
Further, we define the affine variety $\Mcal:=M_1 \times \ldots \times M_k $. For $m=(m_1, \ldots, m_k)\in \Mcal$ we set $\phi(m):=m_1\cdot\ldots\cdot m_k\in M$.

For $t\in\R$ define $a_{t}:=\exp(tX)$ and consider the algebraic map
$$
\Phi_t:\,\, U^x \times \oline N^x\times A^x  \times \Mcal\times H \to G;\qquad
(u,\oline n, a, m,h)\mapsto    u\oline n a\phi(m) a_txh\, .
$$
We have
$$
\gf
=\uf_{\Cc}^x\oplus\oline \nf_{\Cc}^x \oplus \af^x \oplus \bigoplus_{j=1}^k\R X_j \oplus \hf_{\Cc,x}\,.
$$
Note that if $\Ad(ax)\hf$ would not be transversal to $\uf_{\Cc}^x\oplus \oline \nf_{\Cc}^x \oplus\af^x \oplus \bigoplus_{j=1}^k\R X_j$ for some $a\in A$, then it would not be transversal for any $a\in A$ since the space $\uf_{\Cc}^x\oplus \oline \nf_{\Cc}^x \oplus\af^x \oplus \bigoplus_{j=1}^k\R X_j$ is $A$-invariant. This would contradict the fact that $\hf_{\Cc,x}$ is transversal to $\uf_{\Cc}^x\oplus \oline \nf_{\Cc}^x \oplus\af^x \oplus \bigoplus_{j=1}^k\R X_j$.
We thus conclude that for every $a\in A$
$$
\gf = \uf_{\Cc}^x\oplus \oline \nf_{\Cc}^x \oplus\af^x \oplus \bigoplus_{j=1}^k\R X_j \oplus \Ad(ax)\hf\,.
$$
In particular this holds for $a=a_{t}$.
This implies for generic $t$, and hence in particular for $t\gg0$, that the map $\Phi_t$ is dominant and as such has generically finite fibers, with a fiber bound independent of $t$. See \cite[Prop. 15.5.1(i)]{Groth}.

Let $U^x_{B}$, $\oline N^x_{B}$, $A^x_{B}$, $\Mcal_{B}$ and $H_{B}$ be relatively compact,
open neighborhoods of $e$ in $U^{x}$, $\oline N^{x}$, $A^{x}$, $\phi(\Mcal)$ and $H$ respectively. We choose these sets small enough so that $U^x_{B}\oline N^x_{B}A^x_{B}\Mcal_{B}\subseteq B$.
Then
\begin{equation}\label{eq v ineq}
\v(a_tx\cdot z_0)
\geq \int_{Z}\1_{U^x_{B}\oline N^x_{B}A^x_{B}\Mcal_{B}a_{t}x\cdot z_{0}}(z)\,dz\,.
\end{equation}
For $y\in G$, let $F_{y}$ be the projection onto $H$ of $\Phi_{t}^{-1}(\{y\})$. If $yh\in U^{x}_{B}\oline{N}^{x}_{B}A^{x}_{B}\Mcal_{B}a_{t}xH_{B}$ then $y\in U^{x}_{B}\oline{N}^{x}_{B}A^{x}_{B}\Mcal_{B}a_{t}xH_{B}h^{-1}$. Hence $H_{B}h^{-1}$ contains an element from $F_{y}$ and $h$ belongs to $(F_{y})^{-1}H_{B}$. Therefore,
\begin{align*}
\int_{H}\1_{U^x_{B}\oline N^x_{B}A^x_{B}\Mcal_{B}a_{t}xH_{B}}(yh)\,dh
&\leq\int_{H} \1_{(F_{y})^{-1} H_{B}} (h) dh\,\1_{U^x_{B}\oline N^x_{B}A^x_{B}\Mcal_{B}a_{t}x\cdot z_{0}}(y\cdot z_{0})\\
&\leq\#\Phi_{t}^{-1}(\{y\})\vol_{H}(H_{B})\1_{U^x_{B}\oline N^x_{B}A^x_{B}\Mcal_{B}a_{t}x\cdot z_{0}}(y\cdot z_{0})\,.
    \end{align*}
Let $c=\big(n\vol_{H}(H_{B})\big)^{-1}$, where $n$ is the generic fiber bound. Then for generic $y\in G$ we have
$$
\1_{U^x_{B}\oline N^x_{B}A^x_{B}\Mcal_{B}a_{t}x\cdot z_{0}}(y\cdot z_{0})
\geq c\int_{H}\1_{U^x_{B}\oline N^x_{B}A^x_{B}\Mcal_{B}a_{t}xH_{B}}(yh)\,dh.
$$
By inserting this inequality into (\ref{eq v ineq}) we obtain
\begin{align*}
\v(a_tx\cdot z_0)
&\geq\int_{Z}c\int_{H}\1_{U^x_{B}\oline N^x_{B}A^x_{B}\Mcal_{B}a_{t}xH_{B}}(yh)\,dh\,dyH\\
&=c\int_{G}\1_{U^x_{B}\oline N^x_{B}A^x_{B}\Mcal_{B}a_{t}xH_{B}}(y)\,dy\\
&=c\int_{G}\1_{U^x_{B}\oline N^x_{B}A^x_{B}\Mcal_{B}a_{t}xH_{B}x^{-1}a_{-t}}(y)\,dy\,.
\end{align*}
For the last equality we used the invariance of the Haar measure on $G$.

We define $\Xi:=U^{x}_{B}\oline N^{x}_{B}A^{x}_{B}\Mcal_{B}$ and set
$$
\Psi_t:\,\,\Xi\times xH_{B}x^{-1} \to G;
\qquad(\xi, y)\mapsto \xi a_{t} ya_{-t}\, .
$$
The fibers of $\Psi_{t}$ are bounded by the fibers of $\Phi_{t}$, and hence are generically finite with fiber bound independent of $t$ for $t\gg0$.
Let $\omega_{G}$ be the section of $\bigwedge^{\dim G}T^{*}G$ corresponding to the Haar measure on $G$.
Then
$$
\v(a_tx\cdot z_0)
\geq\frac{c}{k} \int_{\Xi}\int_{xH_{B}x^{-1}}\Psi_{t}^{*}\omega_{G}\,,
$$
where $k$ is the fiber bound of $\Psi_{t}$.

We finish the proof by estimating $\Psi_{t}^{*}\omega_{G}$.
For $g\in G$, let $l_{g}:G\to G$ and $r_{g}:G\to G$ be left and right-multiplication by $g$, respectively. Let
$$
\xi\in\Xi,\quad
y\in xH_{B}x^{-1},\quad
Y_{1}\in T_{\xi}\Xi \quad\text{and}\quad
Y_{2}\in T_{y}(xHx^{-1})\,.
$$
Let $\gamma:\R\to\xi^{-1}\Xi$ and $\delta:\R\to xH_{B}x^{-1}y^{-1}$ be smooth paths so that
$$
\gamma(0)=\delta(0)=e,
\quad\gamma'(0)=(T_{e}l_{\xi})^{-1}Y_{1}
\quad\text{and}\quad\delta'(0)=T_{y}r_{y^{-1}}Y_{2}\,.
$$
Then
$$
\frac{d}{ds}\gamma(s)a_{t}\delta(s)a_{-t}\big|_{s=0}
=\gamma'(0)+\Ad(a_{t})\delta'(0)
=(T_{e}l_{\xi})^{-1}Y_{1}+\Ad(a_{t})\Big(T_{y}r_{y^{-1}}Y_{2}\Big).
$$
Now $\xi\gamma$ is a smooth path in $\Xi$ with $(\xi\gamma)(0)=\xi$ and $(\xi\gamma)'(0)=Y_{1}$. Likewise, $\delta y$ is a smooth path in $x H_{B}x^{-1}$ satisfying $(\delta y)(0)=y$ and $(\delta y)'(0)=Y_{2}$.

The tangent map of $\Psi_{t}$ is determined by the following identity of elements in $T_{\xi}G$
\begin{align}
\nonumber&T_{(\xi,y)}\big(r_{a_{t}y^{-1}a_{-t}}\circ\Psi_{t}\big)\big(Y_{1},Y_{2}\big)
=\frac{d}{ds}\,\Psi_{t}\Big(\xi\gamma(s),\delta(s)y\Big)a_{t}y^{-1}a_{-t}\,\big|_{s=0}\\
\nonumber&\qquad=\frac{d}{ds}\,\xi\gamma(s)a_{t}\delta(s)a_{-t}\,\big|_{s=0}
=T_{e}l_{\xi}\Big(\frac{d}{ds}\gamma(s)a_{t}\delta(s)a_{-t}\,\big|_{s=0}\Big)\\
\label{eq Tangent map of transl of Psi}&\qquad=Y_{1}+T_{e}l_{\xi}\Ad(a_{t})\Big(T_{y}r_{y^{-1}}Y_{2}\Big)\,.
\end{align}

We write $\hf_{X,x}$ for the limit for $t\to\infty$ of $\Ad(a_{t})\Ad(x)\hf$ in the Grassmannian.
Let $Y$ be a non-zero eigenvector of $\ad(X)$ in $\hf_{X,x}$ and let $\alpha\in\Sigma\cup\{0\}$ be such that $\alpha(X)$ is the eigenvalue.
It follows from (\ref{eq formula for E_X}) with $E=\Ad(x)\hf$, that there exists an element
$$
Y'\in\big(Y+\sum_{\substack{\beta\in\Sigma\cup \{0\}\\\beta(X)<\alpha(X)}}\gf^{\beta}\big)\cap\Ad(x)\hf.
$$
Let $\tilde Y$ be a right-invariant vector field on $xHx^{-1}$ such that $\tilde Y(e)=Y'$.
Then
\begin{align*}
\lim_{t\to\infty}e^{-t\alpha(X)}T_{(\xi,y)}\big(r_{a_{t}y^{-1}a_{-t}}\circ\Psi_{t}\big)\big(0,\tilde Y(y)\big)
&=\lim_{t\to\infty}e^{-t\alpha(X)}\Big(T_{e}l_{\xi}\circ \Ad(a_{t})\Big)\big(T_{y}r_{y^{-1}}\tilde Y(y)\big)\\
&=\lim_{t\to\infty}e^{-t\alpha(X)}\Big(T_{e}l_{\xi}\circ \Ad(a_{t})\Big)(Y').
\end{align*}
For $\beta\in\Sigma\cup\{0\}$ with $\beta(X)<\alpha(X)$, let $Y'_{\beta}\in\gf^{\beta}$ be so that
$$
Y'=Y+\sum_{\substack{\beta\in\Sigma\cup\{0\}\\\beta(X)<\alpha(X)}}Y'_{\beta}.
$$
Then
$$
e^{-t\alpha(X)}\Ad(a_{t})Y'
=Y+\sum_{\substack{\beta\in\Sigma\cup\{0\}\\\beta(X)<\alpha(X)}}a_{t}^{\beta-\alpha}Y'_{\beta}.
$$
Therefore,
\begin{align*}
&\lim_{t\to\infty}e^{-t\alpha(X)}T_{(\xi,y)}\big(r_{a_{t}y^{-1}a_{-t}}\circ\Psi_{t}\big)\big(0,\tilde Y(y)\big)\\
&\qquad=T_{e}l_{\xi}Y+\lim_{t\to\infty}\sum_{\substack{\beta\in\Sigma\cup\{0\}\\\beta(X)<\alpha(X)}}a_{t}^{\beta-\alpha}T_{e}l_{\xi}Y'_{\beta}
=T_{e}l_{\xi}Y\,.
\end{align*}
The convergence is uniform in $y$ and uniform on compact sets in $\xi$.
Combining this with (\ref{eq Tangent map of transl of Psi}) yields that for every $Y_{1}\in T_{\xi}\Xi$ and $Y$ as before we have
$$
\lim_{t\to\infty}e^{-t\alpha(X)}T_{(\xi,y)}\big(r_{a_{t}y^{-1}a_{-t}}\circ\Psi_{t}\big)\big(Y_{1},\tilde Y(y)\big)
=Y_{1}+T_{e}l_{\xi}Y\,,
$$
where again the convergence is uniform in $y$ and uniform on compact sets in $\xi$.
Define $\rho_{X,x}:=\frac{1}{2}\tr\big(\ad(X)\big|_{\hf_{X,x}}\big)$.
It follows that
$$
e^{-2t\rho_{X,x}}\Psi_{t}^{*}\omega_{G}
=e^{-2t\rho_{X,x}}\big(r_{a_{t}y^{-1}a_{-t}}\circ\Psi_{t}\big)^{*}\omega_{G}
$$
converges for $t\to\infty$ to a nowhere vanishing continuous section of $\bigwedge^{\dim G}T^{*}\big(\Xi\times xH_{B}x^{-1}\big)$.
The proposition now follows from the facts that $\Xi$ and $xH_{B}x^{-1}$ are relatively compact and that $\rho(\hf_{\Cc,x})(X)=\rho_{X,x}$.
\end{proof}

\subsection{Escaping to infinity on $\hat Z$}

Recall $\hat \hf=\hf+\af_{E}$. For a connected component $\Cc$ of $\af_{\rm o-reg}^{--}$ and $Y\in\Cc$, define $\hat\hf_{\Cc,x}=\lim_{t\to\infty}\Ad(\exp(tY)x)\hat\hf$. Obviously we have
$\hf_{\Cc,x} \triangleleft \hat \hf_{\Cc,x}$ and that $\hat \hf_{\Cc,x}$ is $\af$-invariant. We define
$$
\af_{x}^{E}:= \hat \hf_{\Cc,x}\cap \af \supseteq \af_{x}\, .
$$
It follows from Lemma \ref{Lemma Limits of subspaces}(\ref{Lemma Limits of subspaces - item 4}) that the space $\af_{x}^{E}$ does not depend on the connected component $\Cc$ of $\af_{\rm o-reg}^{--}$.
Furthermore, it is independent of the representative $x\in G$ of the double coset $PxH\in P\bs G/H$, cf.~(\ref{hf_C,x invariance}).
Note that $\af_{e}^{E}=\af_{E}$.

\begin{prop} \label{escape to infinity}Let $X\in\af^{-}\setminus\af_{x}^{E}$. Then
$\{ \exp(t X)x \hat H\mid t\geq 0\} $ is unbounded in $\hat Z$.
\end{prop}

\begin{proof}
Set $a_t:=\exp(tX)$.  We argue by contradiction and assume that
$\{ a_t x \hat H\mid t\geq 0\} $ is relatively compact in $\hat Z$.
Then there exists a compact set $C\subseteq G$ such that
\begin{equation}\label{C-incl}
a_t x \in Cx \hat H \qquad (t\geq 0)\, .
\end{equation}
Let
$$
\hat \hf^1:= \big(\Ad(x) \hat \hf\big)_{X}\, .
$$
With $\hat d:=\dim \hat \hf$ we notice that the natural map
$$
\hat Z \to \Gr_{\hat d}(\gf), \ \ g\hat H \mapsto   \Ad(g)\hat \hf
$$
is continuous and thus (\ref{C-incl}) implies that there exists a $c\in C$
such that $\hat \hf^1= \Ad(cx) \hat \hf$.  Since $\Ad(a_t) \hat \hf^1=\hat\hf^1$ for all $t\in \R$ we thus obtain that
$\Ad(c^{-1} a_t cx)\hat \hf =\Ad(x)\hat \hf$ and in particular $\Ad(c)^{-1} X \in N_\gf(\Ad(x)\hat \hf)=\Ad(x)N_\gf(\hat \hf)$.
Recall from Lemma \ref{normalizer lemma} (\ref{normalizer lemma part 2}, \ref{normalizer lemma part 3}) that $N_\gf(\hat \hf)=\hat \hf +\hat \mf$ for some subalgebra $\hat \mf \subseteq\mf$.
Hence it follows that
\begin{equation}\label{X in hat hf^1 + Ad(cx) hat mf}
X\in \hat \hf^1 + \Ad(cx)\hat \mf
=:\tilde\hf^{1}\,.
\end{equation}
We claim that $X\in\hat\hf^{1}$. To see this, assume that $X\notin\hat\hf^{1}$. Since $X$ is hyperbolic and the elements in $\Ad(cx)\hat \mf$ are elliptic, $X\notin\Ad(cx)\hat \mf$. Let $X_{\mf}\in\Ad(cx)\hat\mf$ be so that $X\in\hat\hf^{1}+X_{\mf}$. Let $\tilde H^{1}$ and $\hat H^{1}$ be the connected algebraic subgroups with Lie algebra equal to $\tilde \hf^{1}$ and $\hat \hf^{1}$, respectively. The map $\R X\to\R X_{\mf}$; $tX\mapsto tX_{\mf}$ induces a non-trivial algebraic homomorphism from $\R^{\times}$ to the compact group $\tilde H^{1}/\hat H^{1}$. This leads to a contradiction as such algebraic homomorphisms do not exist. This proves the claim.

Let $\Cc$ be a connected component of $\af^{--}_{\mathrm o-reg}$ so that $X\in\oline{\Cc}$ and let $Y\in\Cc$. Then by Lemma \ref{Lemma Limits of subspaces} (\ref{Lemma Limits of subspaces - item 3})
$$
(\hat\hf^{1})_{Y}
=(\Ad(x)\hat\hf)_{Y}
=\hat\hf_{\Cc,x}.
$$
Therefore,
$$
X
\in\hat\hf^{1}\cap\af
=\big(\hat\hf^{1}\cap\af\big)_{Y}
\subseteq(\hat\hf^{1})_{Y}\cap\af
=\hat\hf_{\Cc,x}\cap\af
=\af_{x}^{E}\,,
$$
which is the desired contradiction.
\end{proof}

\section{Principal asymptotics}

In this section we analyze the asymptotic behavior of generalized matrix coefficients $m_{v,\eta}$ where $\eta\in(\pi_{\lambda,\sigma}^{-\infty})^{H}$.
Before we state the main theorem, we introduce some notation.

Let $\sigma\in\hat M$ and $\lambda\in\af_{\C}^{*}$.
We identify $\pi_{\lambda,\sigma}^{\infty}$ with the space of smooth sections of the vector-bundle $V_{\sigma}\otimes\C_{\lambda}\times_{P}G\to P\bs G$. The support of a section or a functional is defined in the usual way as a closed subset of $P\bs G$. For an open subset $U$ of $P\bs G$ we define $\pi_{\lambda,\sigma}^{\infty}(U)$ to be the space of all $v\in \pi_{\lambda,\sigma}^{\infty}$ with compact support contained in $U$.
We write $\pi_{\lambda,\sigma}^{-\infty}(U)$ for the continuous dual of $\pi_{\lambda,\sigma}^{\infty}(U)$.

For $x\in G$ we define $[x]\in P\bs G$ to be the coset $Px$.

It follows from Lemma \ref{normalizer lemma}(\ref{normalizer lemma part 4}) that $\hat\hf\cap\Ad(x^{-1})\nf\subseteq\hf$. Moreover, for every $Y\in\af_{x}^{E}$ there exists a $Y_{\nf}\in\nf$ such that $Y+Y_{\nf}\in\Ad(x)\hat\hf$.
(See equation (\ref{eq formula for E_X}) in Lemma \ref{Lemma Limits of subspaces}.)
Therefore for $\chi\in(\hat\hf/\hf)^{*}_{\C}$ and $x\in G$ we may define $\chi_{x}\in(\af_{x}^{E})^{*}_{\C}$ to be given by the singleton
\begin{equation}\label{def chi_x}
\{\chi_{x}(Y)\}=\chi\Big([\Ad(x^{-1})(Y+\nf)]\cap\hat\hf\Big)\qquad(Y\in\af_{x}^{E})\,.
\end{equation}
Note that $\chi_x\big|_{\af_x}=0$ and that $\chi_{x}$ only depends on the $H$-orbit $P\bs PxH$, not on the representative $x\in G$ of the orbit.

\begin{theorem}\label{leading behavior}
Let $\eta\in(\pi_{\lambda,\sigma}^{-\infty})^{H}$ and let $x\in G$. Assume that there exists an open neighborhood $\Upsilon$ of $[x]$ in $P\bs G$ such that
\begin{equation}\label{support condition}
\supp\eta\cap \Upsilon= P\bs Px H\cap \Upsilon\,.
\end{equation}
Let $\Cc$ be a connected component of $\af_{\rm o-reg}^{--}$.
For every $X\in\oline\Cc$ there exists a neighborhood $\Omega$ of $[e]$ in $\Upsilon x^{-1}$ and
a unique pair of a constant $r_{X}\geq0$ and a non-zero functional $\eta_{X,x}\in\pi_{\lambda,\sigma}^{-\infty}(\Omega)$, satisfying
\begin{equation}\label{limit of functional}
\lim_{t\to\infty}e^{t \big(\lambda(X)+\rho_P(X)+2\rho(\oline\nf_{\Cc,x})(X)-r_{X}\big)}\pi_{\lambda,\sigma}^{\vee}\big(\exp(tX)x\big)\eta
=\eta_{X,x}\,.
\end{equation}
Here the limit is with respect to weak-$*$ topology on $\pi_{\lambda,\sigma}^{-\infty}(\Omega)$.

For $X\in\Cc$ outside of a finite set of hyperplanes $\mathcal{H}_{\Cc}$, there exists a $\omega\in-\N_{0}[\Pi]$, so that $\omega(X)=r_{X}$, and so that $\eta_{X,x}$ satisfies
\begin{align}
&\label{eta_(X,x) hf_C-invariance}
    \pi_{\lambda,\sigma}^{\vee}(\hf_{\Cc,x})\eta_{X,x}=\{0\}\,,\\
&\label{eta_(X,x) af-equivariance}
    \pi_{\lambda,\sigma}^{\vee}(Y)\eta_{X,x}=\big(-\lambda-\rho_P-2\rho(\oline\nf_{\Cc,x})+\omega\big)(Y)\eta_{X,x}
    \qquad(Y\in\af)\,.
\end{align}
Moreover, if $\chi\in(\hat\hf/\hf)^{*}_{\C}$ and $\eta$ satisfy
\begin{equation}\label{eta hat hf-equivariance}
\pi_{\lambda,\sigma}^{\vee}(Y)\eta
=-\chi(Y)\eta
\qquad(Y\in\hat\hf)\,,
\end{equation}
then
\begin{equation}\label{eta_(X,x)hat hf_C-equivariance}
    \pi_{\lambda,\sigma}^{\vee}(Y)\eta_{X,x}
    =-\chi_{x}(Y)\eta_{X,x}
    \qquad(Y\in\af_{x}^{E})\,.
\end{equation}
\end{theorem}

\begin{rmk}\label{rmk support condition}
For every non-zero $H$-invariant functional $\eta\in\pi_{\lambda,\sigma}^{-\infty}$ there exist an $x\in G$ and an open neighborhood $\Upsilon$ of $[x]$ in $P\bs G$ such that (\ref{support condition}) holds.
Indeed, let $\Oc_{0}$ be an $H$-orbit in $\supp(\eta)$ of maximal dimension and let $x\in\Oc_{0}$.
The action of $H$ on $P\bs G$ admits finitely many orbits.
(See \cite{Bien} and \cite{KS}.)
Since $H$ is a real algebraic group, and $P\bs G$ is a real algebraic variety, and the action of $H$ on $P\bs G$ is real algebraic, the closure of any $H$-orbit $\Oc$ in $P\bs G$ consists of $\Oc$ and $H$-orbits of strictly smaller dimension.
See \cite[Proposition 8.3]{Hum}. Therefore,
$$
\Upsilon
:=\Oc_{0}\cup\bigcup_{\substack{\Oc\in P\bs G/H\\ \dim(\Oc)>\dim(\Oc_{0})}}\Oc\,.
$$
is an open neighborhood of $[x]$ and $\supp(\eta)\cap\Upsilon=\Oc_{0}$.
\end{rmk}

Before we prove the theorem we list some direct implications, which will be crucial in the following sections.

\begin{cor}\label{cor asymp}
Let $\eta\in\big(\pi_{\lambda,\sigma}^{-\infty}\big)^{H}$ and let $x\in G$. Assume that there exists an open neighborhood $\Upsilon$ of $[x]$ in $P\bs G$ such that (\ref{support condition}) holds.
Let $\mathcal{C}$ be a connected component of $\af_{\rm o-reg}^{--}$.
\begin{enumerate}[(i)]
\item\label{Cor asymp item 1}
For every $X\in\oline\Cc$
there exists a $r_{X}\geq 0$ and a $v\in\pi_{\lambda,\sigma}^{\infty}$ such that
$$
m_{v,\eta}\big(\exp(tX)x\cdot z_0\big)\sim  e^{t\big(-\lambda(X) -\rho_P(X) -2\rho(\oline\nf_{\Cc,x})(X)+r_{X}\big)}
\qquad (t\to \infty)\,.
$$
\item\label{Cor asymp item 2}
There exists a $\omega\in-\N_{0}[\Pi]$ such that
$$
\lambda\big|_{\af_{x}}
=\big(-\rho_P -2\rho(\oline\nf_{\Cc,x})+\omega\big)\big|_{\af_{x}}\,.
$$
\item\label{Cor asymp item 3}
Let $\chi\in(\hat\hf/\hf)^{*}_{\C}$ and assume that (\ref{eta hat hf-equivariance}) is satisfied. Then there exists a $\omega\in-\N_{0}[\Pi]$
such that
$$
\lambda\big|_{\af_{x}^{E}}
=\big(-\rho_P -2\rho(\oline\nf_{\Cc,x})+\omega\big)\big|_{\af_{x}^{E}}+\chi_{x}\,.
$$
Here $\chi_x$ is given by (\ref{def chi_x}).
\end{enumerate}
\end{cor}

\begin{proof}
{\em Ad (\ref{Cor asymp item 1}):} The functional $\eta_{X,x}$ is non-zero, hence there exists a $v\in\pi_{\lambda,\sigma}^{\infty}(\Omega)$ for which $\eta_{X,x}(v)=1$. The claim now follows from (\ref{limit of functional}).\\
{\em Ad (\ref{Cor asymp item 3}):}
Let $X\in\Cc\setminus\mathcal{H}_{\Cc}$.
Since $\af_{x}^{E}=\hat\hf_{\Cc,x}\cap\af$, the identity follows from (\ref{eta_(X,x) af-equivariance}) and (\ref{eta_(X,x)hat hf_C-equivariance}).\\
{\em Ad (\ref{Cor asymp item 2}):}
The identity  follows from (\ref{Cor asymp item 3}) since $\chi_{x}\big|_{\af_{x}}=0$.
\end{proof}

In the remainder of this section we give the proof of Theorem \ref{leading behavior}.
\medbreak

We fix an element $x\in G$ and a connected component $\Cc$ of $\af_{\rm o-reg}^{--}$.
Recall that $\oline \nf_{\Cc}^{x}\subseteq\oline \nf$ is an $\af$-invariant vector complement of $\oline \nf_{\Cc,x}$, so that $\oline\nf =\oline\nf_{\Cc,x} \oplus\oline\nf_{\Cc}^{x}$. By Lemma \ref{Lemma Ad(x)hf+pf+nf^x=gf} we have
$$
\gf=(\Ad(x)\hf+\pf)\oplus\oline\nf_{\Cc}^{x}\,.
$$
Choose a subspace $\pf'$ of $\pf$ so that $\gf=\Ad(x)\hf\oplus\oline\nf_{\Cc}^{x}\oplus\pf'$.
Let
$$
\psi:\oline\nf_{\Cc,x}\to\oline\nf_{\Cc}^{x}+\pf
$$
be minus the restriction of the projection $\gf\to\oline\nf_{\Cc}^{x}\oplus\pf'$ along this decomposition.
Then
$$
Y+\psi(Y)\in\Ad(x)\hf\qquad(Y\in\oline\nf_{\Cc,x})\,.
$$
For every $Y\in\oline\nf_{\Cc,x}$
$$
Y
=(1+\psi)(Y)-\psi(Y)\in\mathrm{Im}(1+\psi)+\oline\nf_{\Cc}^{x}+\pf.
$$
Combining this with a dimension count yields
\begin{equation}\label{eq decomp with psi}
\gf
=\mathrm{Im}(1+\psi)\oplus\oline\nf_{\Cc}^{x}\oplus\pf.
\end{equation}

For the proof of Theorem \ref{leading behavior} we need the following lemma.

\begin{lemma}\label{Lemma convergence of a psi ainv}
Let $X\in\oline\Cc$ and let $\psi:\oline\nf_{\Cc,x}\to\oline\nf_{\Cc}^{x}+\pf$ as above. The limit
$$
\psi_{X}:=\lim_{t\to\infty}\Ad\big(\exp(tX)\big)\circ\psi\circ\Ad\big(\exp(-tX)\big)
$$
exists in the space of  linear maps $\oline\nf_{\Cc,x}\to\oline\nf_{\Cc}^{x}+\pf$.
Moreover, if $X\in\Cc$, then $\psi_{X}=0$.
\end{lemma}

\begin{proof}
Let $X_{0}\in\Cc$.
If $E$ is a line in the set $\Ad(x)\hf\setminus(\Ad(x)\hf\cap \pf)$, then in view of (\ref{eq formula for E_X}) in Lemma \ref{Lemma Limits of subspaces}, the limit $E_{X_{0}}$ is a line in $\oline\nf$. Since this limit is also contained in $\hf_{\Cc,x}$, it is in fact contained in $\hf_{\Cc,x}\cap\oline{\nf}=\oline\nf_{\Cc,x}$.
In particular, if $Y\in\oline\nf_{\Cc,x}\setminus\{0\}$, then $Y+\psi(Y)\in \Ad(x)\hf\setminus(\Ad(x)\hf\cap \pf)$ by (\ref{eq decomp with psi}), and hence the limit of $\Ad\big(\exp(tX_{0})\big)\R\big(Y+\psi(Y)\big)$ is a line in $\oline\nf_{\Cc,x}$.
Since $\oline\nf_{\Cc}^{x}\oplus\pf$ is stable under the adjoint action of $A$,
the eigenvalues of $\ad(X_{0})$ occurring in the decomposition of $\psi(Y)$ into eigenvectors must be smaller than the largest eigenvalue occurring in the decomposition of $Y$ into eigenvectors. Therefore, it follows that
\begin{equation}\label{limit of quotient}
\lim_{t\to\infty}\frac{\|\Ad\big(\exp(tX_{0})\big)\psi(Y)\|}{\|\Ad\big(\exp(tX_{0})\big)Y\|}=0\qquad(Y\in\oline\nf_{\Cc,x}\setminus\{0\})\,.
\end{equation}

For $\alpha\in\Sigma\cup\{0\}$ let $p_{\alpha}$ be the projection onto $\gf^{\alpha}$ with respect to the root space decomposition. Here  $\gf^{0}=\mf\oplus\af$. Let $\alpha,\beta\in\Sigma\cup\{0\}$. It follows from (\ref{limit of quotient}) that $p_{\beta}\circ\psi\circ p_{\alpha}\neq0$ implies that $\alpha(X_{0})-\beta(X_{0})>0$. Since this holds for every $X_{0}\in\mathcal{C}$, it follows that
$$
\psi=\sum_{\substack{\alpha,\beta\in\Sigma\cup\{0\}\\ (\alpha-\beta)|_{\mathcal{C}}>0}}p_{\beta}\circ\psi\circ p_{\alpha}\,.
$$
Now
$$
\Ad\big(\exp(tX)\big)\circ\psi\circ\Ad\big(\exp(-tX)\big)
=\sum_{\substack{\alpha,\beta\in\Sigma\cup\{0\}\\ (\alpha-\beta)|_{\mathcal{C}}>0}}
    e^{t\big(\beta(X)-\alpha(X)\big)}p_{\beta}\circ\psi\circ p_{\alpha}
$$
If $X\in\oline{\mathcal C}$, then $(\alpha-\beta)\big|_{\mathcal{C}}>0$ implies that $\alpha(X)\geq\beta(X)$. Therefore,
$$
\lim_{t\to\infty}\Ad\big(\exp(tX)\big)\circ\psi\circ\Ad\big(\exp(-tX)\big)
=\sum_{\substack{\alpha,\beta\in\Sigma\cup\{0\}\\ (\alpha-\beta)|_{\mathcal{C}}>0\\ \alpha(X)=\beta(X)}}
    p_{\beta}\circ\psi\circ p_{\alpha}\,.
$$
The first claim in the lemma now follows with
\begin{equation}\label{psi_X}
\psi_{X}
=\sum_{\substack{\alpha,\beta\in\Sigma\cup\{0\}\\ (\alpha-\beta)|_{\mathcal{C}}>0\\ \alpha(X)=\beta(X)}}
    p_{\beta}\circ\psi\circ p_{\alpha}\,.
\end{equation}
If $X\in\Cc$ then the sum in (\ref{psi_X}) is over the empty set, and hence $\psi_{X}=0$.
This proves the second assertion in the lemma.
\end{proof}

It follows from (\ref{eq decomp with psi}) and  the inverse function theorem that for sufficiently small neighborhoods $V_{1}$ of $0$ in $\oline\nf_{\Cc,x}$ and $V_{2}$ of $0$ in $\oline\nf_{\Cc}^{x}$, the map
$$
\Phi: V_{1}\times V_{2}\to P\bs G;\qquad (Y_{1},Y_{2})\mapsto P\exp(Y_{2})\exp\big(Y_{1}+\psi(Y_{1})\big)x
$$
is a diffeomorphism onto an open neighborhood of $[x]$. Moreover,
$$
V_{1}\ni Y\mapsto \Phi(Y,0)
$$
is a diffeomorphism onto a submanifold of $P\bs G$ contained in $P\bs PxH$. Because the dimension of the image equals the dimension of $P\bs PxH$, it in fact covers an open neighborhood of $[x]$ in $P\bs PxH$.

We view $\pi_{-\lambda,\sigma^{\vee}}^{\infty}$ and $C^{\infty}(V_{1}\times V_{2},V_{\sigma})$ as spaces of smooth sections of vector-bundles and write $\Phi^{*}$ for the pull-back along $\Phi$, i.e.,
$\Phi^{*}$ is the map $\pi_{-\lambda,\sigma^{\vee}}^{\infty}\to C^{\infty}(V_{1}\times V_{2},V_{\sigma}^{*})$ given by $\Phi^{*}v=v\circ\Phi$. This map has a continuous extension to a map $\Phi^{*}:\pi_{\lambda,\sigma}^{-\infty}\to\D'(V_{1}\times V_{2})\otimes V_{\sigma}^{*}$. Similarly we have a pull-back map $\pi_{\lambda,\sigma}^{\infty}\to C^{\infty}(V_{1}\times V_{2},V_{\sigma})$ which we also denote by $\Phi^{*}$. We note that there exists a strictly positive smooth function $J$ on $V_{1}\times V_{2}$ such that
\begin{equation}\label{pull-back chi}
\varphi(\phi)=\Phi^{*}\varphi(J\Phi^{*}\phi)
\end{equation}
for every $\varphi\in\pi_{\lambda,\sigma}^{-\infty}$ and $\phi\in \pi_{\lambda,\sigma}^{\infty}$ with $\supp\phi\subseteq\Phi(V_{1}\times V_{2})$.

Let $n=\dim(V_{2})$ and let $e_{1},\dots, e_{n}$ a basis of $\oline\nf_{\Cc}^{x}$ of joint eigenvectors
for the action of $\ad(\af)$. We write $\partial_{i}$ for the partial derivative in the direction $e_{i}$,
and whenever $\mu$ is an $n$-dimensional multi-index we write
$\partial^{\mu}$ for $\partial_{1}^{\mu_{1}}\dots\partial _{n}^{\mu_{n}}$.

Now $\Phi^{*}\eta$ is a $V_{\sigma}^{*}$-valued distribution on $V_{1}\times V_{2}$. From the condition (\ref{support condition}) on the support of $\eta$ it follows that the support of $\Phi^{*}\eta$ is contained in $V_{1}\times\{0\}$.
It follows from \cite[p. 102]{Schwartz} that there exist a minimal $k\in \N$ and for every multi-index $\mu$ with $|\mu|\leq k$ a $V_{\sigma}^{*}$-valued distribution $\eta_{\mu}$ on $V_{1}$ such that
\begin{equation}\label{Phi*eta decomp}
\Phi^{*}\eta
=\sum_{|\mu|\leq k}\eta_{\mu}\otimes \partial^{\mu}\delta\,.
\end{equation}
Here $\delta$ is the Dirac delta distribution at $0$ on $\oline\nf_{\Cc}^{x}$. Note that this decomposition of $\Phi^{*}\eta$ is unique.

\begin{lemma}\label{analytic densities}
\quad\begin{enumerate}[(i)]
\item\label{analytic densities item 1} For each multi-index $\mu$, the distribution $\eta_{\mu}$ is given by a real analytic function $f_{\mu}:V_{1}\to V_{\sigma}^{*}$, i.e. $\eta_{\mu}=f_{\mu}\,dY_{1}$ where $dY_{1}$ is the Lebesgue measure on $V_{1}$.
\item\label{analytic densities item 2}
For each $Y_{1}\in V_{1}$ there exists a $\mu$ of length $|\mu|=k$ so that $f_{\mu}(Y_{1})\neq 0$.
\end{enumerate}
\end{lemma}

\begin{proof}
In the first part of the proof we follow the analysis of Bruhat as it is described in \cite[Section 5.2.3]{Warner}.
For $h\in H$ we write $U_{h}=\Phi^{-1}\big(\Phi(V_{1}\times V_{2})h^{-1}\big)$
and define the real analytic map
$$
\rho_{h}:U_{h}\to V_{1}\times V_{2};\qquad v\mapsto\Phi^{-1}\big(\Phi(v)h\big)\,.
$$
Note that $\rho_{h}$ maps $U_{h}\cap(V_{1}\times\{0\})$ to $V_{1}\times\{0\}$.
We further write
$$
U_{h,1}:=\{v\in V_{1}:(v,0)\in U_{h}\}
$$
and we define the map $\xi_{h}: U_{h,1}\to V_{1}$ to be given by $\rho_{h}(v,0)=(\xi_{h}(v),0)$ for $v\in V_{1}$.

For all multi-indices $\mu$ and $\nu$ with
$|\mu|,|\nu|\leq k$ there exists a real analytic function
$$
\lambda_{\mu,\nu}: \{(h,v)\in H\times V_{1}:v\in U_{h,1}\}\to\R
$$
such that
$$
\rho_{h}^{*}\big(\1_{V_{1}}\otimes\partial^{\mu}\delta\big)
=\sum_{|\nu|\leq k}\lambda_{\nu,\mu}(h,\dotvar) \otimes \partial^{\nu}\delta
\qquad(h\in H)\,.
$$
(The domain of definition of $\lambda_{\mu,\nu}$ is equal to the inverse image of $\Phi(V_{1},V_{2})$ under the smooth map $V_{1}\times H\to P\bs G$; $(v,h)\mapsto\Phi(v,0)h^{-1}$, and hence it is open.)
Note that pulling back along $\rho_{h}$ does not increase the order of the transversal derivatives, hence $\lambda_{\nu,\mu}=0$ whenever $|\nu|>|\mu|$.
We apply this identity to (\ref{Phi*eta decomp}) and obtain
\begin{align*}
\rho_{h}^{*}(\Phi^{*}\eta)
&=\sum_{|\mu|\leq k}\sum_{|\nu|\leq|\mu|}\lambda_{\nu,\mu}(h,\dotvar)\xi_{h}^{*}\eta_{\mu}\otimes\partial^{\nu}\delta\\
&=\sum_{|\mu|\leq k}\Big(\sum_{k\geq|\nu|\geq |\mu|}\lambda_{\mu,\nu}(h,\dotvar)\xi_{h}^{*}\eta_{\nu}\Big)\otimes\partial^{\mu}\delta\,.
\end{align*}
Since $\eta$ is an $H$-invariant functional we have $\rho_{h}^{*}(\Phi^{*}\eta)=\Phi^{*}\eta$ on $U_{h}$. Together with the uniqueness of the decomposition (\ref{Phi*eta decomp}) this implies for each $\mu$ that
$$
\eta_{\mu}\big|_{U_{h,1}}=\sum_{|\nu|\geq |\mu|}\lambda_{\mu,\nu}(h,\dotvar)\xi_{h}^{*}\eta_{\nu}
\qquad(h\in H)\,.
$$
We now apply the pull-back along $\xi_{h}$ to this identity with $h$ replaced by $h^{-1}$ and thus obtain
\begin{equation}\label{eta_mu equivariance}
\xi_{h}^{*}\eta_{\mu}
=\sum_{|\nu|\geq |\mu|}\lambda_{\mu,\nu}\big(h^{-1},\xi_{h}(\dotvar)\big)\eta_{\nu}\big|_{U_{h,1}}.
\end{equation}
Here we used that $\xi_{h}^{-1}(U_{h^{-1},1})=U_{h,1}$.

Let $n=\dim(\oline\nf_{\Cc}^{x})$ and let $S$ be the set of multi-indices $\mu\in\N_{0}^{n}$ with $|\mu|\leq k$.
We write $p_{\mu}$ for the projection of $(V_{\sigma}^{*})^{S}$ onto the $\mu^{\rm th}$ component and define $\zeta$ to be the $(V_{\sigma}^{*})^{S}$-valued distribution on $V_{1}$ which for a multi-index $\mu$ is given by
$$
p_{\mu}\zeta
=\eta_{\mu}\,.
$$
For $h\in H$ and $v\in U_{h,1}$, let $\Lambda(h,v)\in\End\big((V_{\sigma}^{*})^{S}\big)$ be given by
$$
p_{\mu}\circ \Lambda(h,v)\circ p_{\nu}=\lambda_{\mu,\nu}\big(h^{-1},\xi_{h}(v)\big)\,.
$$
Then
$$
\xi_{h}^{*}\zeta=\Lambda(h,\dotvar)\zeta\big|_{U_{h,1}}\,.
$$

We will finish the proof of the lemma by invoking the elliptic regularity theorem to show that $\zeta$ is locally given by a real analytic $(V_{\sigma}^{*})^{S}$-valued function. To this end, let $D$ be a real analytic elliptic differential operator of order $d>0$ on the trivial vector bundle $V_{1}\times (V_{\sigma}^{*})^{S}\to V_{1}$. (Such differential operators exist, e.g. $\Delta\otimes \1$ where $\Delta$ is the Laplacian on $V_{1}$ and $\1$ the identity operator on $(V_{\sigma}^{*})^{S}$.) Let $u_{1},\dots, u_{l}$ be a basis of $\U_{d}(\hf)$. Since $H$ acts transitively on $P\bs PxH$, there exist real analytic functions $c_{j}:V_{1}\to \End\big((V_{\sigma}^{*})^{S}\big)$ such that for $\phi\in C^{\infty}\big(V_{1},(V_{\sigma}^{*})^{S}\big)$
$$
D\phi(v)
=\sum_{j=1}^{l}c_{j}(v)u_{j}(\xi_{h}^{*}\phi)(v)\big|_{h=e}
\qquad(v\in V_{1})\,.
$$
Let $v_{0}\in V_{1}$.
Since $D$ is elliptic of order $d>0$, there exists a neighborhood $U$ of $v_{0}$ such that the operator $D'$, which for $\phi\in C^{\infty}\big(U,(V_{\sigma}^{*})^{S}\big)$ is given by
$$
D'\phi(v)
=\sum_{j=1}^{l}c_{j}(v_{0})u_{j}\big(\xi_{h}^{*}\phi-\Lambda(h,\dotvar)\phi\big)(v)\big|_{h=e}\qquad(v\in U)\,,
$$
is a real analytic elliptic differential operator on the vector bundle $U\times (V_{\sigma}^{*})^{S}\to U$.
Note that $D'\zeta=0$ on $U$. By the elliptic regularity theorem, there exists a real analytic function $f:U\to (V_{\sigma}^{*})^{s}$ such that $\zeta=f\, dY_{1}$ on $U$.
(See for example \cite[Theorem IV.4.9]{Wells} for the smoothness of the solutions and \cite[p. 144]{John} for the analyticity.)
Since $v_{0}$ was chosen arbitrarily, it follows that $f$ extends to an analytic function on $V_1$ and that $\zeta=f\, dY_{1}$ on $V_{1}$.
Let $f_{\mu}=p_{\mu}f$. Then $f_{\mu}$ is real analytic and $\eta_{\mu}=f_{\mu}dY_{1}$.
This proves (\ref{analytic densities item 1}).

By (\ref{eta_mu equivariance}) we have for every $\mu$ of length $|\mu|=k$
$$
f_{\mu}\big(\xi_{h}(Y_{1})\big)
=\sum_{|\nu|=k}\lambda_{\mu,\nu}\big(h^{-1},\xi_{h}(Y_{1})\big)f_{\nu}(Y_{1})\qquad (h\in H, Y_{1}\in U_{h,1}).
$$
Let $Y_{1}\in V_{1}$ be such that $f_{\nu}(Y_{1})=0$ for all $\nu$ of length $|\nu|=k$, then the right-hand side vanishes at the point $Y_{1}$ for all $h\in H$ such that $Y_{1}\in U_{h,1}$. This implies that the left-hand side vanishes on an open neighborhood of $Y_{1}$. Since the $f_{\mu}$ are analytic, it follows that all $f_{\mu}$ for $\mu$ of length $|\mu|=k$ vanish on $V_1$.
Assertion (\ref{analytic densities item 2}) now follows from the definition of $k$.
\end{proof}

\begin{proof}[Proof of Theorem \ref{leading behavior}]
Let $\Phi$ be as before.
Recall that $\pi_{\lambda,\sigma}^{\infty}\big(\im(\Phi)\big)$ is the space of all $v\in \pi_{\lambda,\sigma}^{\infty}$ with compact support contained in the image $\im(\Phi)$ of $\Phi$.
Let $v\in\pi_{\lambda,\sigma}^{\infty}\big(\im(\Phi)\big)$.
It follows from (\ref{pull-back chi}), (\ref{Phi*eta decomp}) and Lemma \ref{analytic densities}(\ref{analytic densities item 1}) that
\begin{align*}
&\eta(v)
=\Phi^{*}\eta\big(J\Phi^{*}(v)\big)\\
&=\sum_{|\mu|\leq k}(-1)^{|\mu|}
    \int_{V_{1}}\partial^{\mu}_{Y_{2}} \left[J(Y_{1},Y_{2})
    f_{\mu}(Y_{1})\Big(v\Big(\exp(Y_{2})\exp\big(Y_{1}+\psi(Y_{1})\big)x\Big)\Big)\right]\Big|_{Y_{2}=0}
    \,dY_{1}\,.
\end{align*}
By the Leibniz rule the integrand on the right-hand side is equal to
$$
\sum_{\nu\leq\mu}\binom{\mu}{\nu}
    \big[\partial^{\mu-\nu}_{Y_{2}}J(Y_{1},Y_{2})\big]\Big|_{Y_{2}=0}f_{\mu}(Y_{1})
        \Big[\partial^{\nu}_{Y_{2}}v\Big(\exp(Y_{2})\exp\big(Y_{1}+\psi(Y_{1})\big)x\Big)\Big]\Big|_{Y_{2}=0}\,.
$$

Note that the Jacobian $J$ is a real analytic function. By Lemma \ref{analytic densities}(\ref{analytic densities item 1}) also the functions $f_{\mu}$ are real analytic. Let $\epsilon_{1},\dots, \epsilon_{m}$ be a basis of $\oline\nf_{\Cc,x}$ consisting of joint eigenvectors for the action of $\ad(\af)$ on $\oline\nf_{\Cc,x}$.
For a multi-index $\kappa$ and $Y\in\oline\nf_{\Cc,x}$ define $Y^{\kappa}\in\R$ in the usual manner with respect to the basis $\epsilon_{1},\dots,\epsilon_{m}$.
By shrinking $V_{1}$ and $V_{2}$ we may assume that the Taylor series of $J$ and the $f_{\mu}$ are absolutely convergent on $V_{1}\times V_{2}$ and $V_{1}$, respectively.
Let
\begin{equation}\label{eq Taylor series}
(-1)^{|\mu|}\binom{\mu}{\nu}\big[\partial_{Y_{2}}^{\mu-\nu}J(Y_{1},Y_{2})\big]\big|_{Y_{2}=0}f_{\mu}(Y_{1})
=\sum_{\kappa}Y_{1}^{\kappa}c_{\mu,\nu}^{\kappa}
\end{equation}
be the Taylor expansion of the function on the left-hand side. Here for every multi-index $\kappa$ the coefficient $c_{\mu,\nu}^{\kappa}$ is an element of $V_{\sigma}^{*}$.
Since the series on the right-hand side of (\ref{eq Taylor series}) is absolutely convergent on $V_{1}$ and since $v$ has compact support in $\im(\Phi)$, we can
apply Lebesgue's dominated convergence theorem to interchange the integral and the sums, and obtain
\begin{equation}\label{eq formula for eta}
\eta(v)
=\sum_{|\nu|\leq k}\sum_{\kappa}
    \int_{V_{1}}Y_{1}^{\kappa}C_{\nu}^{\kappa}
        \Big[\partial^{\nu}_{Y_{2}}v\Big(\exp(Y_{2})\exp\big(Y_{1}+\psi(Y_{1})\big)x\Big)\Big]\Big|_{Y_{2}=0}\,dY_{1}\,,
\end{equation}
where
$$
C_{\nu}^{\kappa}
:=\sum_{|\mu|\leq k, \mu\geq \nu}c_{\mu,\nu}^{\kappa}\in V_{\sigma}^{*}\,.
$$
Recall that $e_{1},\dots, e_{n}$ is a basis of $\oline\nf_{\Cc}^{x}$ consisting of joint eigenvectors for the action of $\ad(\af)$ on $\oline\nf_{\Cc}^{x}$.
For a multi-index $\nu$, let $\omega_{2,\nu}\in-\N_{0}[\Pi]$ be the $\af$-weight of $e_{1}^{\nu_{1}}\cdots e_{n}^{\nu_{n}}\in\U(\oline\nf)$, where $\U(\oline\nf)$ denotes the universal enveloping algebra of $\oline\nf$. Further, for a multi-index $\kappa$ we define $\omega_{1,\kappa}\in-\N_{0}[\Pi]$ to be the $\af$-weight of $\epsilon_{1}^{\kappa_{1}}\cdots \epsilon_{m}^{\kappa_{n}}\in\U(\oline\nf)$.
Define
$$
\Xi
:=\{(\nu,\kappa): C_{\nu}^{\kappa}\neq0\}\,.
$$
Let $X\in\oline{\Cc}$ be fixed.
The set $\{\omega_{2,\nu}(X)-\omega_{1,\kappa}(X):(\nu,\kappa)\in\Xi\}$ is discrete. Moreover, it is bounded from above as there exists only finitely many multi-indices $\nu$ of length at most $k$ and $\omega_{1,\kappa}(X)\geq 0$ for every $\kappa$.
Define
\begin{equation}\label{r_X}
r_{X}
:=\max\{\omega_{2,\nu}(X)-\omega_{1,\kappa}(X):(\nu,\kappa)\in\Xi\}
\end{equation}
and
$$
\Xi_{X}
:=\{(\nu,\kappa)\in \Xi:\omega_{2,\nu}(X)-\omega_{1,\kappa}(X)= r_{X}\}\,.
$$
By Lemma \ref{analytic densities}(\ref{analytic densities item 2}) there
exists a multi-index $\mu_{0}$ of length $k$ such that $f_{\mu_{0}}(0)\neq 0$.
If we take $\mu=\nu=\mu_{0}$ then the left-hand side of (\ref{eq Taylor series}) is non-zero in $Y_{1}=0$. Therefore, the coefficient $C_{\mu_{0}}^{0}=c_{\mu_{0},\mu_{0}}^{0}\neq 0$, and hence $(\mu_{0},0)\in\Xi$. Since $\omega_{1,0}=0$, we have $r_{X}\geq \omega_{2,\mu_{0}}(X)\geq 0$.

We will now specify the domain $\Omega$ that appears in the theorem.
For this we first introduce a family of diffeomorphisms. For $t\in\R$, let $a_{t}:=\exp(tX)$.
We define
$$
\Psi_{t}:\Ad(a_{t})V_{1}\times \Ad(a_{t})V_{2}\to P\bs G;
\quad (Y_{1},Y_{2})\mapsto \Phi\Big(\Ad(a_{t}^{-1})Y_{1},\Ad(a_{t}^{-1})Y_{2}\Big)\,x^{-1}a_{t}^{-1}\,.
$$
Observe that $\Psi_{t}$ is a diffeomorphism onto its image for every $t\in\R$. For every $(Y_{1},Y_{2})\in\Ad(a_{t})V_{1}\times\Ad(a_{t})V_{2}$ we have
\begin{align*}
\Psi_{t}(Y_{1},Y_{2})
&=P\exp\Big(\Ad(a_{t}^{-1})Y_{2}\Big)\exp\Big(\Ad(a_{t}^{-1})Y_{1}+\psi\big(\Ad(a_{t}^{-1})Y_{1}\big)\Big)a_{t}^{-1}\\
&=P\exp(Y_{2})\exp\Big(Y_{1}+\Ad(a_{t})\psi\big(\Ad(a_{t}^{-1})Y_{1}\big)\Big)\,.
\end{align*}

Let $\mathcal{G}_{X}$ be the graph of $\psi_{X}$. Then $\gf=\pf\oplus\mathcal{G}_{X}\oplus\oline\nf_{\Cc}^{x}$, and thus there exist open neighborhoods $W_{1}$ and $W_{2}$ of $0$ in $\oline\nf_{\Cc,x}$ and $\oline\nf_{\Cc}^{x}$ respectively such that the map
$$
\Psi_{\infty}:W_{1}\times W_{2}\to P\bs G\,,
$$
given by
$$
\Psi_{\infty}(Y_{1},Y_{2})=P\exp(Y_{2})\exp\big(Y_{1}+\psi_{X}(Y_{1})\big)\,,
$$
is a diffeomorphism onto an open neighborhood of $[e]$ in $P\bs G$.
The map $\Psi_{\infty}$ is a limit of the maps $\Psi_{t}$ in the following sense. Since $\Ad(a_{t})$ acts with eigenvalues larger or equal than $1$ on $\oline\nf_{\Cc,x}$ and $\oline\nf_{\Cc}^{x}$, there exist bounded open neighborhoods $U_{1}$ and $U_{2}$ of $0$ in $\oline \nf_{\Cc,x}$ and $\oline \nf_{\Cc}^{x}$, respectively, satisfying
$$
\overline{U_{1}}\subseteq W_{1}\cap\bigcap_{t\geq0}\Ad(a_{t})V_{1}
\quad\text{and}\quad
\overline{U_{2}}\subseteq W_{2}\cap\bigcap_{t\geq0}\Ad(a_{t})V_{2}\,.
$$
It follows from Lemma \ref{Lemma convergence of a psi ainv} that
\begin{equation}\label{eq limit Psi_t to Psi_infty}
\lim_{t\to\infty}\Psi_{t}(Y_{1},Y_{2})
=\Psi_{\infty}(Y_{1},Y_{2})
\qquad\big((Y_{1},Y_{2})\in U_{1}\times U_{2}\big)\,,
\end{equation}
where the limit takes place in the space of smooth maps $U_{1}\times U_{2}\to P\bs G$.
We claim that for sufficiently large $R>0$ there exists an open neighborhood $\Omega$ of $[e]$ in $P\bs G$ such that
\begin{equation}\label{eq Omega contained in images of Psi_t}
\Omega
\subseteq\Psi_{\infty}(U_{1}\times U_{2})\cap\bigcap_{t>R}\Psi_{t}(U_{1}\times U_{2})\,.
\end{equation}
Indeed, the constructive proof of the inverse function theorem (see for example Lemma 1.3 in \cite{Lang}) gives a lower bound on the size of the open neighborhood of $[e]\in P\bs G$  that is contained in $\Psi_{t}(U_{1}\times U_{2})$ in terms of the tangent map of $\Psi_{t}$ at $(0,0)$. The claim therefore follows immediately from (\ref{eq limit Psi_t to Psi_infty}).

For $(\nu,\kappa)\in\Xi$, let $\eta_{X}^{\nu,\kappa}\in\pi_{\lambda,\sigma}^{-\infty}(\Omega)$ be the functional which for $v\in\pi_{\lambda,\sigma}^{\infty}(\Omega)$ is  given by
\begin{equation}\label{def eta_X^(nu,kappa)}
\eta_{X}^{\nu,\kappa}(v)
:=\int_{U_{1}}Y_{1}^{\kappa}C_{\nu}^{\kappa}
        \Big[\partial_{Y_{2}}^{\nu}v\Big(\exp(Y_{2})\exp\big(Y_{1}+\psi_{X}(Y_{1})\big)\Big)\Big]\Big|_{Y_{2}=0}\,dY_{1}\,.
\end{equation}
We claim that (\ref{limit of functional}) holds with $r_{X}$ given by (\ref{r_X}) and $\eta_{X,x}$ by the sum
\begin{equation}\label{eta_X=sum eta_X^(nu,kappa)}
\eta_{X,x}
:=\sum_{(\nu,\kappa)\in\Xi_{X}}\eta_{X}^{\nu,\kappa}\,,
\end{equation}
where the sum is convergent in $\pi_{\lambda,\sigma}^{-\infty}(\Omega)$ with respect to the weak-$*$-topology.

To prove the claim, let $t>R$ and consider $v\in\pi_{\lambda,\sigma}^{\infty}(\Omega)$.
For every $Y_{2}\in \oline\nf_{\Cc}^{x}$ and $Y\in\gf$.
$$
[\pi_{\lambda,\sigma}(x^{-1}a_{t}^{-1})v]\big(\exp(Y_{2})\exp(Y) x\big)
=a_{t}^{-\lambda-\rho_{P}}v\Big(\exp\big(\Ad(a_{t})Y_{2}\big)\exp\big(\Ad(a_{t})Y\big)\Big)\,.
$$
From (\ref{eq formula for eta}) it follows that
\begin{align}\label{eq formula for eta - 2}
&a_{t}^{\lambda+\rho_{P}}\eta\big(\pi_{\lambda,\sigma}(x^{-1}a_{t}^{-1})v\big)\\
\nonumber&=\sum_{|\nu|\leq k}\sum_{\kappa}
    \int_{U_{1}}Y_{1}^{\kappa}C_{\nu}^{\kappa}
        \Big[\partial_{Y_{2}}^{\nu}v\Big(\exp\big(\Ad(a_{t})Y_{2}\big)\exp\Big(\Ad(a_{t})\big(Y_{1}+\psi(Y_{1})\big)\Big)\Big)\Big]\Big|_{Y_{2}=0}\,dY_{1}\,.
\end{align}
If $1\leq i\leq n$ and $\alpha$ is the root so that $e_{i}\in\gf^{\alpha}$, then $\Ad(a_{t})e_{i}=a_{t}^{\alpha}e_{i}$, and hence
$$
\frac{d}{ds}v\Big(\exp\big(\Ad(a_{t})(s e_{i})\big)\exp\big(\Ad(a_{t})Y\big)\Big)
=a_{t}^{\alpha}\frac{d}{ds}v\Big(\exp(se_{i})\exp\big(\Ad(a_{t})Y\big)\Big)\,.
$$
Applying the previous identity repeatedly yields
\begin{align*}
&\partial^{\nu}_{Y_{2}}v\Big(\exp\big(\Ad(a_{t})Y_{2}\big)\exp\big(\Ad(a_{t})Y\big)\Big)\big|_{Y_{2}=0}\\
&\qquad\qquad=a_{t}^{\omega_{2,\nu}}\partial_{Y_{2}}^{\nu}v\Big(\exp(Y_{2})\exp\big(\Ad(a_{t})Y\big)\Big)\big|_{Y_{2}=0}\,.
\end{align*}
Combining this identity with (\ref{eq formula for eta - 2}), we obtain
\begin{align*}
&a_{t}^{\lambda+\rho_{P}}\eta\big(\pi_{\lambda,\sigma}(x^{-1}a_{t}^{-1})v\big)\\
&=\sum_{|\nu|\leq k}\sum_{\kappa}
    a_{t}^{\omega_{2,\nu}}
    \int_{U_{1}}Y_{1}^{\kappa}C_{\nu}^{\kappa}
        \Big[\partial_{Y_{2}}^{\nu}v\Big(\exp(Y_{2})\exp\Big(\Ad(a_{t})\big(Y_{1}+\psi(Y_{1})\big)\Big)\Big)\Big]\Big|_{Y_{2}=0}\,dY_{1}\,.
\end{align*}
By definition of $\omega_{1,\kappa}$
$$
\big(\Ad(a_{t}^{-1})Y_{1}\big)^{\kappa}
=a_{t}^{-\omega_{1,\kappa}}Y_{1}^{\kappa}\qquad(Y_{1}\in\oline\nf_{\Cc,x})\,.
$$
We now perform a substitution of variables and obtain that
\begin{align*}
&\int_{U_{1}}Y_{1}^{\kappa}C_{\nu}^{\kappa}
        \Big[\partial_{Y_{2}}^{\nu}v\Big(\exp(Y_{2})\exp\Big(\Ad(a_{t})\big(Y_{1}+\psi(Y_{1})\big)\Big)\Big)\Big]\Big|_{Y_{2}=0}\,dY_{1}\\
&\qquad\qquad=a_{t}^{-2\rho(\oline\nf_{\Cc,x})-\omega_{1,\kappa}}
    \int_{\Ad(a_{t})U_{1}}Y_{1}^{\kappa}C_{\nu}^{\kappa}
        \big[v_{\nu,t}(Y_{1})\big]\,dY_{1}\,,
\end{align*}
where
$$
v_{\nu,t}(Y_{1})
:=
\partial_{Y_{2}}^{\nu}v\Big(\exp(Y_{2})\exp\Big(Y_{1}+\Ad(a_{t})\psi\big(\Ad(a_{t}^{-1})Y_{1}\big)\Big)\Big)\Big|_{Y_{2}=0}
$$
for $Y_{1}\in \Ad(a_{t})U_{1}$. It follows from (\ref{eq Omega contained in images of Psi_t}) and the fact that $v$ is supported in $\Omega$, that
$$
\supp(v_{\nu,t})\subseteq U_{1}.
$$
Now
\begin{align}\label{eq series expansion of functional}
\nonumber&e^{t\big(\lambda(X)+\rho_{P}(X)+2\rho(\oline\nf_{\Cc,x})(X)-r_{X}\big)}\eta\big(\pi_{\lambda,\sigma}(x^{-1}a_{t}^{-1})v\big)\\
&\qquad=\sum_{|\nu|\leq k}\sum_{\kappa}e^{t\big(\omega_{2,\nu}(X)-\omega_{1,\kappa}(X)-r_{X}\big)}
    \int_{U_{1}}Y_{1}^{\kappa}C_{\nu}^{\kappa}[v_{\nu,t}(Y_{1})]\,dY_{1}\,.
\end{align}

Since $U_{1}$ is bounded, the support of the functions $v_{\nu,t}$ is bounded uniformly in $t>0$.
Therefore, $v_{\nu,t}$ converges for $t\to\infty$ in the space $C_{c}^{\infty}(U_{1},V_{\sigma})$ to the function
$$
Y_{1}\mapsto \partial_{Y_{2}}^{\nu}v\Big(\exp(Y_{2})\exp\big(Y_{1}+\psi_{X}(Y_{1})\big)\Big)\Big|_{Y_{2}=0}\,,
$$
and thus we obtain,
\begin{align*}
&\lim_{t\to\infty}\int_{U_{1}}Y_{1}^{\kappa}C_{\nu}^{\kappa}[v_{\nu,t}(Y_{1})]\,dY_{1}\\
&\qquad=\int_{U_{1}}Y_{1}^{\kappa}C_{\nu}^{\kappa}
        \Big[\partial_{Y_{2}}^{\nu}v\Big(\exp(Y_{2})\exp\big(Y_{1}+\psi_{X}(Y_{1})\big)\Big)\Big]\Big|_{Y_{2}=0}\,dY_{1}
=\eta_{X}^{\nu,\kappa}(v)\,.
\end{align*}
For the last equality we used (\ref{def eta_X^(nu,kappa)}).

Let $r=\sup_{Y_{1}\in U_{1}}\|Y_{1}\|$. Since $U_{1}$ is bounded, we have $r<\infty$. Moreover, since $\overline{U_{1}}\subseteq V_{1}$, we also have that $r$ is strictly smaller than the convergency radius of the Taylor series in (\ref{eq Taylor series}), and hence
\begin{equation}\label{eq bounded sum}
\sum_{\kappa}r^{|\kappa|}\|C_{\nu}^{\kappa}\|<\infty\,.
\end{equation}

As $v_{\nu,t}$ is bounded uniformly in $t>0$ and $\nu$, and $e^{t\big(\omega_{2,\nu}(X)-\omega_{1,\kappa}(X)-r_{X}\big)}\leq 1$ for all $t>0$ and $(\nu,\kappa)\in\Xi$, it follows from (\ref{eq bounded sum}) that the series in (\ref{eq series expansion of functional}) is absolutely convergent uniformly in $t>0$.
Therefore,
\begin{align*}
&\lim_{t\to\infty}e^{t\big(\lambda(X)+\rho_{P}(X)+2\rho(\oline\nf_{\Cc,x})(X)-r_{X}\big)}\eta\big(\pi_{\lambda,\sigma}(x^{-1}a_{t}^{-1})v\big)\\
&\qquad=\sum_{|\nu|\leq k}\sum_{\kappa}\lim_{t\to\infty}\Big(e^{t\big(\omega_{2,\nu}(X)-\omega_{1,\kappa}(X)-r_{X}\big)}
    \int_{U_{1}}Y_{1}^{\kappa}C_{\nu}^{\kappa}[v_{\nu,t}(Y_{1})]\,dY_{1}\Big)\\
&\qquad=\sum_{(\nu,\kappa)\in\Xi_{X}}\eta_{X}^{\nu,\kappa}(v)\,.
\end{align*}
This proves the claim that (\ref{limit of functional}) holds with $r_{X}$ given by (\ref{r_X}) and $\eta_{X,x}$ by the convergent sum
(\ref{eta_X=sum eta_X^(nu,kappa)}).

\medbreak

We claim that $\eta_{X,x}\neq0$.
Let $v\in\pi_{\lambda,\sigma}^{\infty}(\Omega)$. Since $v$ is compactly supported, it follows from (\ref{eq bounded sum}) and Lebesgue's dominated convergence theorem, that we may interchange the sum and the integral, so that
\begin{align}
\nonumber
    \eta_{X,x}(v)
    &=\sum_{(\nu,\kappa)\in\Xi_{X}}\int_{U_{1}}Y_{1}^{\kappa}C_{\nu}^{\kappa}
            \Big[\partial_{Y_{2}}^{\nu}v\Big(\exp(Y_{2})\exp\big(Y_{1}+\psi_{X}(Y_{1})\big)\Big)\Big]\Big|_{Y_{2}=0}\,dY_{1}\\
\label{eq formula for eta_X,x}
    &=\sum_{|\nu|\leq k}\int_{U_{1}}F_{\nu,X}(Y_{1})
          \Big[\partial_{Y_{2}}^{\nu}v\Big(\exp(Y_{2})\exp\big(Y_{1}+\psi_{X}(Y_{1})\big)\Big)\Big]\Big|_{Y_{2}=0}\,dY_{1}\,,
\end{align}
where $F_{\nu,X}:U_{1}\to V_{\sigma}^{*}$ is given by the absolutely convergent series
\begin{equation}\label{eq def F_nu,X}
F_{\nu,X}(Y_{1})
:=\sum_{\{\kappa:(\nu,\kappa)\in\Xi_{X}\}}Y_{1}^{\kappa}C_{\nu}^{\kappa}\,.
\end{equation}
If $\{\kappa:(\nu,\kappa)\in\Xi_{X}\}\neq\emptyset$, then $F_{\nu,X}$ is not identically equal to $0$ since it is given by an absolutely convergent power series with at least one non-zero coefficient. Since $\Xi_{X}\neq\emptyset$ there exists at least one multi-index $\nu_{0}$ so that $F_{\nu_{0},X}$ is not identically equal to $0$.

Let $v_{\sigma}\in V_{\sigma}$ and let $\phi_{1}\in C_{c}^{\infty}(U_{1})$ and $\phi_{2}\in C_{c}^{\infty}(U_{2})$.
We now take $v$ to be the element of $\pi_{\lambda,\sigma}^{\infty}(\Omega)$ that is determined by
$$
v\Big(\exp(Y_{2})\exp\big(Y_{1}+\psi_{X}(Y_{1})\big)\Big)
=\phi_{1}(Y_{1})\phi_{2}(Y_{2})v_{\sigma}
\qquad(Y_{1}\in U_{1},Y_{2}\in U_{2})\,.
$$
(Recall that $\Psi_{\infty}$ is a diffeomorphism, and hence $v$ is well defined.)
Then
$$
\eta_{X,x}(v)
=\sum_{|\nu|<k}\partial^{\nu}\phi_{2}(0)
    \int_{U_{1}}\big(F_{\nu,X}(Y_{1})(v_{\sigma})\big)\phi_{1}(Y_{1})\,dY_{1}\,.
$$
We assume that $v_{\sigma}$, $\phi_{1}$ and $\phi_{2}$ satisfy
\begin{enumerate}[(a)]
  \item $\displaystyle\partial^{\nu_{0}}\phi_{2}(0)=1$,
  \item If $\nu\neq\nu_{0}$, then $\displaystyle\partial^{\nu}\phi_{2}(0)=0$,
  \item $\displaystyle Y_{1}\mapsto F_{\nu_{0},X}(Y_{1})(v_{\sigma})$ is not identically equal to $0$,
  \item $\displaystyle\int_{U_{1}}\big(F_{\nu,X}(Y_{1})(v_{\sigma})\big)\phi_{1}(Y_{1})\,dY_{1}=1$.
\end{enumerate}
Under these assumptions we have $\eta_{X,x}(v)=1$, and hence $\eta_{X,x}\neq0$.

\medbreak

We move on to show (\ref{eta_(X,x) hf_C-invariance}) for $X\in\Cc$.
Let $\alpha\in\Sigma\cup\{0\}$ and let $Y\in(\hf_{\Cc,x}\cap \gf^{\alpha})\setminus\{0\}$.
For $Y'\in\gf$, we write
$$
\Ad(x)Y'
=\sum_{\beta\in\Sigma\cup\{0\}}Y'_{x,\beta}\,,
$$
with $Y'_{x,\beta}\in\gf^{\beta}$ for every $\beta\in\Sigma\cup\{0\}$.
In view of (\ref{eq formula for E_X}) in Lemma \ref{Lemma Limits of subspaces}
there exists an element $Y'\in\hf$ such that $Y'_{x,\alpha}=Y$ and $\alpha$ is the unique element of $\Sigma\cup\{0\}$ satisfying
$$
\alpha(X)
=\max\{\beta(X):\beta\in\Sigma\cup\{0\}, Y'_{x,\beta}\neq0\}\,.
$$
For every $v\in\pi^{\infty}_{\lambda,\sigma}$ we have
\begin{align*}
&e^{-t\alpha(X)}e^{t (\lambda+\rho_P+2\rho(\oline\nf_{\Cc,x})-\omega_{X})(X)}
    \pi_{\lambda,\sigma}^{\vee}\big(\exp(tX)x\big)\pi_{\lambda,\sigma}^{\vee}(Y')\eta\\
&\qquad=\sum_{\beta\in\Sigma\cup\{0\}}e^{t(\beta-\alpha)(X)}\pi_{\lambda,\sigma}^{\vee}(Y'_{x,\beta})
    \big[e^{t(\lambda+\rho_P+2\rho(\oline\nf_{\Cc,x})-\omega_{X})(X)}
            \pi_{\lambda,\sigma}^{\vee}\big(\exp(tX)x\big)\eta\big]\\
&\qquad\longrightarrow\pi_{\lambda,\sigma}^{\vee}(Y'_{x,\alpha})\eta_{X,x}
=\pi_{\lambda,\sigma}^{\vee}(Y)\eta_{X,x}
 \qquad(t\to\infty)\,.
\end{align*}
Here the limit is taken with respect to the weak-$*$ topology.
Since $Y'\in\hf$, we have $\pi_{\lambda,\sigma}^{\vee}(Y')\eta=0$, hence $\pi_{\lambda,\sigma}^{\vee}(Y)\eta_{X,x}=0$.
This proves (\ref{eta_(X,x) hf_C-invariance}).

For $X\in\Cc$, we have $\psi_{X}=0$ by Lemma \ref{Lemma convergence of a psi ainv}. Let $\oline N_{\Cc,x}$ be the connected subgroup of $G$ with Lie algebra $\oline\nf_{\Cc,x}$ and let $d\oline n$ denote the Haar measure on $\oline N_{\Cc,x}$.
Then the expression (\ref{eq formula for eta_X,x}) for $\eta_{X,x}$ simplifies to
\begin{align*}
\eta_{X,x}(v)
&=\sum_{|\nu|\leq k}\int_{\oline\nf_{\Cc,x}}F_{\nu,X}(Y_{1})
          \Big[\partial_{Y_{2}}^{\nu}v\Big(\exp(Y_{2})\exp(Y_{1})\Big)\Big]\Big|_{Y_{2}=0}\,dY_{1}\\
&=\sum_{|\nu|\leq k}\int_{\oline N_{\Cc,x}}F_{\nu,X}\big(\log(\oline{n})\big)
          \Big[\partial_{Y_{2}}^{\nu}v\Big(\exp(Y_{2})\oline n\Big)\Big]\Big|_{Y_{2}=0}\,d\oline n
          \qquad\big(v\in\pi_{\lambda,\sigma}^{\infty}(\Omega)\big)\,.
\end{align*}
Since $\oline\nf_{\Cc,x}\subseteq\hf_{\Cc,x}$, it follows from (\ref{eta_(X,x) hf_C-invariance}) that
$\pi_{\lambda,\sigma}^{\vee}(\oline\nf_{\Cc,x})\eta_{X,x}=\{0\}$.
Because of the invariance of the Haar measure on $\oline N_{\Cc,x}$, this implies that $F_{\nu,X}$ is constant for every $\nu$.
Therefore, only terms with $\kappa=0$ can contribute to $F_{\nu,X}$ in the series in (\ref{eq def F_nu,X}).
In particular it follows that $(\nu,\kappa)\in\Xi_{X}$ implies that $\kappa=0$. Moreover, $r_{X}$ in (\ref{r_X})
is equal to $\omega_{2,\mu_{0}}(X)$ for some multi-index $\mu_{0}$ with the property that $f_{\mu_{0}}(0)\neq 0$
and $f_{\mu}(0)=0$ for every $\mu>\mu_{0}$. Let $\omega:=\omega_{2,\mu_{0}}\in-\N_{0}[\Pi]$.
Then $\Xi_{X}$ consists of pairs $(\nu,0)$ with $\omega_{2,\nu}(X)=\omega(X)$.
The formula for $\eta_{X,x}$ simplifies further to
\begin{equation}\label{eta_(X,x)-final formula}
\eta_{X,x}(v)
=\sum_{\substack{|\mu|\leq k\\ \omega_{2,\mu}(X)=\omega(X)}}
    \int_{\oline N_{\Cc,x}}c_{\mu}
        \Big[\partial_{Y_{2}}^{\mu}v\big(\exp(Y_{2})\oline n\big)\Big]\Big|_{Y_{2}=0}\,d\oline n,
\qquad\big(v\in\pi_{\lambda,\sigma}^{\infty}(\Omega)\big)\,,
\end{equation}
with $c_{\mu}:=(-1)^{|\mu|}J(0,0)f_{\mu}(0)\in V_{\sigma}^{*}\setminus\{0\}$.

If we further impose on $X\in\Cc$ the condition that $\chi(X)\neq\chi'(X)$ whenever
$\chi,\chi'\in-\N_{0}[\Pi]$ are two different elements, each of which being a sum of at most
$k$ roots in $-\Sigma$, then $\omega_{2,\mu}(X)=\omega(X)$ if and only if $\omega_{2,\mu}=\omega$.
Equation (\ref{eta_(X,x) af-equivariance}) then follows directly from (\ref{eta_(X,x)-final formula}).

It remains to prove that (\ref{eta hat hf-equivariance}) implies (\ref{eta_(X,x)hat hf_C-equivariance}).
Let $Y\in\af_{x}^{E}$. Then $\Ad(x^{-1})(Y+\nf)\cap\hat\hf$ is non-empty, see (\ref{eq formula for E_X}) in Lemma \ref{Lemma Limits of subspaces}.
Let $Y'\in\Ad(x^{-1})(Y+\nf)\cap\hat\hf$. Then for every $v\in\pi^{\infty}_{\lambda,\sigma}$
\begin{align*}
&e^{t (\lambda+\rho_P+2\rho(\oline\nf_{\Cc,x})-\omega_{X})(X)}
   \pi_{\lambda,\sigma}^{\vee}\big(\exp(tX)x\big)\pi_{\lambda,\sigma}^{\vee}(Y')\eta\\
&\qquad=e^{t (\lambda+\rho_P+2\rho(\oline\nf_{\Cc,x})-\omega_{X})(X)}
   \pi_{\lambda,\sigma}^{\vee}\big(\Ad(\exp(tX)x)Y'\big)\pi_{\lambda,\sigma}^{\vee}\big(\exp(tX)x\big)\eta
\end{align*}
converges to $\pi_{\lambda,\sigma}^{\vee}(Y)\eta_{X,x}$ for $t\to\infty$.
Moreover, it follows from (\ref{eta hat hf-equivariance}) that it also converges to $-\chi(Y')\eta_{X,x}$ for $t\to\infty$. Here again the limits are taken with respect to the weak-$*$ topology. It thus follows that
$$
\pi_{\lambda,\sigma}^{\vee}(Y)\eta_{X,x}
=-\chi(Y')\eta_{X,x}\,.
$$
Now (\ref{eta_(X,x)hat hf_C-equivariance}) follows as $\chi(Y')=\chi_{x}(Y)$.
\end{proof}

\section{Integrality and negativity conditions}
Let us denote by $(\cdot, \cdot)$ the Euclidean structure on $\af^*$.  For $\alpha\in \af^*\bs\{0\}$ we define
$\alpha^\vee \in \af$ by  $\alpha^\vee:= 2\frac{(\cdot , \alpha)}{(\alpha, \alpha) }\in (\af^*)^*=\af$. Recall that if $\alpha\in \Sigma$
then $\alpha^\vee$ is called the co-root of $\alpha$ and $\Sigma^\vee:=\{ \alpha^\vee \mid \alpha\in\Sigma\}$
is a root system on $\af$,  called the dual root system.

For a connected component $\Cc$ of $\af_{\rm o-reg}^{--}$ and $x\in G$, define
$\lf_{\Cc, x}:= \hf_{\Cc, x} \cap \theta\hf_{\Cc, x}$.
Note that $\lf_{\Cc, x}$ is a reductive
$\af$-stable subalgebra of $\hf_{\Cc, x}$ and $\lf_{\Cc,x}\cap \af=\af_{x}$.
Moreover, it follows from (\ref{hf_C,x invariance}) that
$\lf_{\Cc,manxh}=\Ad(m)\lf_{\Cc,x}$ for $m\in M$, $a\in A$, $n\in N$ and $h\in H$.
For $\lambda\in \af_\C^*$ we set
$$
\Sigma(\lambda):=\{ \alpha\in \Sigma| \lambda(\alpha^\vee)\in \Z\}\, .
$$

\begin{lemma}\label{Lemma negative and integral}
Let $(V,\eta)$ be a spherical pair  belonging to the twisted discrete series and assume that there is a quotient
$\pi_{\lambda,\sigma}\twoheadrightarrow V$. Consider $\eta$ as an $H$-fixed element of $\pi_{\lambda,\sigma}^{-\infty}$ and let $x\in G$ satisfy the support condition (\ref{support condition}). (See Remark \ref{rmk support condition}.)
Then the following assertions hold.
\begin{enumerate}[(i)]
\item\label{Lemma negative and integral item 1} $\lambda\big|_{\af_x}\in (-\rho_P + \Z[\Pi])\big|_{\af_x}$.
\item\label{Lemma negative and integral item 2}
Let $\chi\in(\hat\hf/\hf)^{*}_{\C}$ be normalized unitary.
If $(V,\eta)$ belongs to the $\chi$-twisted discrete series, then
$$
\lambda\big|_{\af_{x}^{E}}
\in\frac{1}{2}\Z[\Pi]\big|_{\af_{x}^{E}}+i\im\chi_{x}\,.
$$
\end{enumerate}
Let $\Cc$ be a connected component of $\af_{\rm o-reg}^{--}$. Then the following hold.
\begin{enumerate}[(i)]
\setcounter{enumi}{2}
\item\label{Lemma negative and integral item 3} $\Sigma(\af, \lf_{\Cc, x})\subseteq\Sigma(\lambda)$.
\item \label{Lemma negative and integral item 4} $\re \lambda(X) \leq 2\rho(\lf_{\Cc, x}\cap \nf)(X) $  for all $X\in -\oline{\Cc}\subseteq \af^+$. The inequality is strict for $X\in -\oline{\Cc}\setminus\af_{x}^{E}$.
\end{enumerate}
\end{lemma}

\begin{proof}
Assertion (\ref{Lemma negative and integral item 1}) is immediate from Corollary \ref{cor asymp}(\ref{Cor asymp item 2})
for any choice of $\Cc$.
We move on to (\ref{Lemma negative and integral item 2}).
By (\ref{Expression for Delta_hatZ}) we have
$$
|\det\Ad(ha)\big|_{\gf/\hf}|
=a^{2\rho_{Q}}\qquad(h\in H, a\in A_{Z,E})\,.
$$
We thus see that $\re \chi(Y')=-\frac{1}{2}\tr\ad(Y')\big|_{\gf/\hf}$ for every $Y'\in\hat\hf$. Let $Y\in\af_{x}^{E}$.
It follows from (\ref{eq formula for E_X}) in Lemma \ref{Lemma Limits of subspaces} that there exists an element $Y'$ in $\Ad(x^{-1})(Y+\nf)\cap\hat\hf$.
Now
$$
\re\chi_{x}(Y)
=\re\chi(Y')
=-\frac{1}{2}\tr \ad(Y')\big|_{\gf/\hf}
\in\frac{1}{2}\Z\mathrm{spec}\big(\ad(Y')\big)\,.
$$
The eigenvalues of $\ad(Y')$ are equal to the eigenvalues of $\ad(Y)$. Therefore,
$$
\re\chi_{x}\in\frac{1}{2}\Z[\Pi]\big|_{\af_{x}^{E}}\,.
$$
Since $(V,\eta)$ is $\chi$-twisted, assertion (\ref{Lemma negative and integral item 2}) now follows from Corollary \ref{cor asymp}(\ref{Cor asymp item 3}) for any choice of $\Cc$.

Assertion (\ref{Lemma negative and integral item 3}) is a consequence of (\ref{Lemma negative and integral item 1}) since $\lf_{\Cc,x}\cap \af=\af_{x}$ and hence $\alpha^{\vee}\in\af_{x}$ for all $\alpha\in\Sigma(\af,\lf_{\Cc,x})$.

Moving on to (\ref{Lemma negative and integral item 4}) we first observe that if $(V,\eta)$ is a spherical pair of the twisted discrete series and $\pi_{\lambda,\sigma}\twoheadrightarrow V$, then Corollary \ref{cor asymp} (\ref{Cor asymp item 1}) combined with the bound (\ref{DS-sup}) and Proposition \ref{vol growth} results for
$X\in -\oline{\Cc} \subseteq \af^{+}$ in the inequality
$$
\big( - \re\lambda -\rho_P -2\rho(\oline\nf_{\Cc,x})\big) (-X)+r_{-X} \leq  -\rho(\hf_{\Cc,x})(-X)\,.
$$
Hence
\begin{equation} \label{decay}
(-\re \lambda)(X) \geq   \big(\rho_P+ 2\rho(\oline\nf_{\Cc,x}) - \rho(\hf_{\Cc,x})\big)(X)
\end{equation}
for all $X\in -\oline \Cc$.
If $X\in -\oline\Cc\setminus\af_{x}^{E}$, then instead of (\ref{DS-sup}) we may use (\ref{limit bound}) in conjunction with Proposition \ref{escape to infinity} and conclude that in that case the inequality is strict.

Let $\vf_{\Cc,x}$ be an $\af$-stable complement of $\lf_{\Cc,x}$ in $\hf_{\Cc, x}$. Note that
$$
2\rho(\oline\nf_{\Cc,x}) - \rho(\hf_{\Cc,x})
=-2\rho(\lf_{\Cc,x}\cap\nf)+\rho(\vf_{\Cc,x}\cap\oline\nf)-\rho(\vf_{\Cc,x}\cap\nf)\,.
$$
Since $\vf_{\Cc,x}\cap\theta(\vf_{\Cc,x})=0$, it follows that
$$
\rho_P+ 2\rho(\oline\nf_{\Cc,x}) - \rho(\hf_{\Cc,x}) \in -2\rho(\lf_{\Cc,x}\cap \nf)  + \frac{1}{2} \N_{0}[\Sigma^{+}]\, .
$$
Now (\ref{Lemma negative and integral item 4}) follows from (\ref{decay}).
\end{proof}

\begin{cor}\label{Cor negative and integral}
\quad
\begin{enumerate}[(i)]
\item\label{Cor negative and integral item 1}
Let $\chi\in(\hat\hf/\hf)^{*}_{\C}$ be normalized unitary. There exists a finite set $S_{\chi}$ of pairs $(\bfrak,\nu)$, where $\bfrak$ is a subspace of $\af$ and $\nu\in\bfrak^{*}$, with the following property.  If $(V,\eta)$ is a spherical pair belonging to the $\chi$-twisted discrete series of representations, and there is a quotient
$\pi_{\lambda,\sigma}\twoheadrightarrow V$, then there exists an $\omega\in\Span_{\R}\big(\Sigma(\lambda)\big)$ and a pair $(\bfrak,\nu)\in S_{\chi}$ such that \begin{align*}
&\lambda\big|_{\bfrak}\in\frac{1}{2}\Z[\Pi]\big|_{\bfrak}+i\nu\,,\\
& \re\lambda(X) \leq \omega(X) \qquad (X\in \af^+)\,,\\
& \re\lambda(X) < \omega(X) \qquad (X\in \af^+\setminus\bfrak)\,.
\end{align*}
\item\label{Cor negative and integral item 2}
If $(V,\eta)$ is a spherical pair belonging to the discrete series of representations, and there is a quotient
$\pi_{\lambda,\sigma}\twoheadrightarrow V$, then there exists an $\omega\in\Span_{\R}\big(\Sigma(\lambda)\big)$ and a subspace $\bfrak$ of $\af$ such that
\begin{align*}
&\lambda\big|_{\bfrak}\in(-\rho_{P}+\Z[\Pi])\big|_{\bfrak}\subseteq\frac{1}{2}\Z[\Pi]\big|_{\bfrak}\,,\\
& \re\lambda(X) \leq \omega(X) \qquad (X\in \af^+)\,,\\
& \re\lambda(X) < \omega(X) \qquad (X\in \af^+\setminus\bfrak)\,.
\end{align*}
\end{enumerate}
\end{cor}

\begin{proof}
{\em Ad (i):} Let $S_{\chi}$ be the set of pairs $(\af_{x}^{E},\chi_{x})$ where $x$ runs over a set of representatives in $G$ of $H$-orbits in $P\bs G$.
Consider $\eta$ as an $H$-fixed element of $\pi_{\lambda,\sigma}^{-\infty}$.
Then there exists an $H$-orbit in $P\bs G$ so that the support condition (\ref{support condition}) is satisfied. See Remark \ref{rmk support condition}.
Let $x\in G$ be the representative of the orbit.
The assertions now follow from (\ref{Lemma negative and integral item 2}),
(\ref{Lemma negative and integral item 3}) and (\ref{Lemma negative and integral item 4})
in Lemma \ref{Lemma negative and integral} with $\omega=\sum_{\Cc} 2\rho(\lf_{\Cc, x}\cap \nf)$,
$\bfrak=\af_{x}^{E}$ and $\nu=\im\chi_{x}$.
\\
{\em Ad (ii):} If $V$ belongs to a discrete series representation, then $\hat\hf=\hf$ by Lemma \ref{DS implies H=hat H},
and therefore $\af_x^E=\af_x$. We set $\bfrak=\af_x$ and use (\ref{Lemma negative and integral item 1})
in Lemma \ref{Lemma negative and integral} instead of (\ref{Lemma negative and integral item 2}).
\end{proof}

\section{Negativity versus integrality in root systems}\label{section negative integral}
In this section we develop some general theory which is independent of the results in previous sections.

\subsection{Equivalence relations}\label{subsection equivalence relations}

Let $\Sigma$ be a (possibly non-reduced) root system  spanning the Euclidean space $\af^*$.  We denote by $W$ the corresponding Weyl group. Let $\Pi\subseteq \Sigma$ be a basis,
$\Sigma^+$ the corresponding positive system
and $C\subseteq \af=(\af^*)^*$ be the closure of the corresponding positive Weyl chamber, i.e.
$$
C=\{ x\in\af\mid (\forall\alpha\in \Pi)  \ \alpha(x) \geq 0\}\, .
$$
Further we use the notation $C^\times =C\bs\{0\}$.

We define an equivalence relation on $\af_\C^*$ by $\lambda\sim \mu$ provided that
$\mu$ is obtained from $\lambda$ via a sequence
$$
\lambda=\mu_0,\mu_1,\ldots,\mu_l=\mu\,,
$$
where for all $i$:
\begin{enumerate}[(a)]
\item $\mu_{i+1}=s_i(\mu_i)$ with $s_i=s_{\alpha_i} $ the simple reflection associated to $\alpha_i\in\Pi$,
\item $\mu_i(\alpha_i^\vee)\not\in \Z$.
\end{enumerate}
The equivalence class of $\lambda$ is denoted by $[\lambda]$.

A {\it root subsystem} $\Sigma^0$ of the root system $\Sigma$ is a
subset of $\Sigma$ that satisfies:
\begin{enumerate}[(a)]
\item  $\Sigma^0$ is
a root system in the subspace it spans,
\item  if $\alpha,\beta$
are in $\Sigma^0$, and $\gamma=\alpha+\beta\in\Sigma$, then
$\gamma \in\Sigma^0$.
\end{enumerate}
A root subsystem $\Sigma^0\subseteq \Sigma$ has a unique system of
positive roots $\Sigma^{0,+}$ contained in $\Sigma^+$.

Given now $\lambda\in\af_\C^*$ we define
\begin{align*}
\Sigma(\lambda)^\vee:=\{\alpha^\vee\in\Sigma^\vee|\lambda(\alpha^\vee)\in\Z\} \\
\Sigma(\lambda):=\{\alpha\in\Sigma|\lambda(\alpha^\vee)\in\Z\}\, .
\end{align*}
Clearly $\Sigma(\lambda)^\vee $  is a root subsystem of $\Sigma^\vee$,  but observe that
$\Sigma(\lambda)$ might not be a root subsystem of $\Sigma$.
We call an element $\mu\in\af^{*}$ a weight of $\Sigma(\lambda)$ if $\mu(\alpha^{\vee})\in\Z$ for every $\alpha\in\Sigma(\lambda)$. The set of weights of $\Sigma(\lambda)$ forms a lattice in $\af^{*}$ which contains $\re(\lambda)$.

Next we define an equivalence relation on $W$ by  $u\sim_{\lambda} v$ provided that  $uC$ and $vC$
are connected by a gallery of chambers $(uC=C_0,C_1,...,C_l=vC)$
such that for each $i$, $C_i$ and $C_{i+1}$ are separated by $H_{\beta_i}$
with $\beta_i\in\Sigma\backslash\Sigma(\lambda)$ an indivisible root for each $i$.

Let $\Sigma(\lambda)^{+}=\Sigma(\lambda)\cap\Sigma^{+}$.
We denote the closure of the corresponding positive chamber by $C(\lambda)\subseteq \af$.

\begin{lemma} \label{Eric1}
Let $\lambda\in \af_\C^*$. Then the following assertions hold:
\begin{enumerate}[(i)]
\item\label{Lemma Eric1 item 1} $C(\lambda)$ equals the union of the sets $w(C)$ where
$w$ runs over $[e]_{\lambda}$, the equivalence class of $e\in W$.
\item\label{Lemma Eric1 item 2} Let $\mu\in\af_{\C}^{*}$.
 Then $\lambda\sim\mu$ if and only if there exists a $w\in W$ with $w^{-1}\in [e]_{\lambda}$ such that $\mu=w\lambda$.
\end{enumerate}
\end{lemma}

\begin{proof} We start with the proof of (\ref{Lemma Eric1 item 1}).
Let $D$ be the union of the sets $w(C)$ where $w$ runs over $[e]_{\lambda}$.
By definition $C(\lambda)$ is the closure of a connected
component of the complement of the union
of the hyperplanes $H_{\alpha}$ with $\alpha \in \Sigma(\lambda)$, namely the connected component which contains $\Int(C)$.

Clearly $C(\lambda)$ is the closure of the union of the open chambers it contains. These are of the form $w(\Int(C))$, where $w$ varies over a subset of $[e]_{\lambda}$; indeed, the latter follows since the hyperplanes
intersecting $\Int\big(C(\lambda)\big)$ are hyperplanes of roots which are not
in $\Sigma(\lambda)$. Hence $C(\lambda)\subseteq D$, since $D$ is closed.
But clearly we can not extend any further beyond $C(\lambda)$ while staying
in $D$, since all the walls of $C(\lambda)$ are hyperplanes of roots in
$\Sigma(\lambda)$. Hence the equality is clear and (\ref{Lemma Eric1 item 1}) is established.

Moving on to (\ref{Lemma Eric1 item 2}), let $\lambda=\mu_0,\mu_1,\ldots ,\mu_l=\mu$ be a
sequence connecting $\lambda$ and $\mu=w\lambda$ such that
$\mu_{i+1}=s_i(\mu_i)$, with $s_i$ a reflection in a simple root $\alpha_{i}$, and $\mu_i(\alpha_i^\vee)\not\in \Z$  for all $i$.
Let $w_{0}=e$ and $w_{i+1}=s_iw_i$. Furthermore, let $\beta_i=w_i^{-1}(\alpha_i)$, so that $w_{i+1}=w_is_{\beta_i}$. Then
$\beta_i$ is an indivisible root and
$$
\lambda(\beta^\vee_i)=w_i(\lambda)(\alpha_i^\vee)=\mu_i(\alpha_i^\vee)\not\in \Z\, ,
$$
that is, $\beta_i\in\Sigma\backslash\Sigma(\lambda)$.
We may assume that  $w_{l}=w$.
Therefore, the gallery
$$
C,w_{1}^{-1}(C),w_{2}^{-1}(C),\ldots,w^{-1}(C)
$$
yields an equivalence $w^{-1}\sim_{\lambda} e$.
The converse is also true.
If the gallery $(C_0=C,C_1,\ldots,C_l=w^{-1}(C))$
defines an equivalence $e\sim_{\lambda} w^{-1}$, then $C_{i+1}=s_{\beta_i}(C_i)$
with $ \beta_i\in\Sigma\backslash\Sigma(\lambda)$ an indivisible root for all $i$.
Let $w_{i}\in W$ so that $C_{i}=w_{i}^{-1}C$, and $\mu_i:=w_i(\lambda)$.
Since $H_{\beta_{i}}$ is a common face of $C_i$ and $C_{i+1}$ (by definition
of gallery), we have $s_{\beta_i}w_i^{-1}=w_i^{-1}s_{\alpha_i}$ for some simple root
$\alpha_i=w_{i}\beta_{i}\in \Pi$. Note that
$$
\mu_i(\alpha_i^\vee)=w_i\lambda(\alpha_i^\vee)=\lambda(\beta_i^\vee)\not\in \Z\,.
$$
This implies that $\lambda$ and $w(\lambda)$ are equivalent and finishes the proof of (\ref{Lemma Eric1 item 2}).
\end{proof}

\subsection{Integral-negative parameters}\label{subsection integral-negative}
Let us call $\lambda\in\af_\C^*$ {\it weakly integral-negative} provided that
there exists a $\omega_\lambda \in \Span_\R (\Sigma(\lambda))$ and a subspace $\af_{\lambda}\subseteq\af$ such that
\begin{align*}
&(\re \lambda -  \omega_\lambda)\big|_{C}\leq0\,,\\
&(\re \lambda -  \omega_\lambda)\big|_{C\setminus\af_{\lambda}}<0\,.
\end{align*}
Further, we call  $\lambda\in\af_\C^*$ {\it integral-negative} provided that
there exists a $\omega_\lambda \in \Span_\R (\Sigma(\lambda))$  and a subspace $\af_\lambda\subseteq \af$  such that
\begin{align*}
&\lambda\big|_{\af_\lambda}= \re \lambda\big|_{\af_\lambda}\,,\\
&(\re \lambda -  \omega_\lambda)\big|_{C}\leq0 \, ,\\
&(\re \lambda -  \omega_\lambda)\big|_{C\bs \af_\lambda} <0 \, .
\end{align*}
Finally, we call $\lambda\in\af_\C^*$ {\it strictly integral-negative} if there exists a $\omega_\lambda \in \Span_\R (\Sigma(\lambda))$ such that
$$
(\re \lambda -  \omega_\lambda)\big|_{C\setminus\{0\}} <0 \, .
$$

\begin{rmk}\label{rmk parameter negative}
These definitions are motivated by our results from the previous section.
Let $\af\subseteq\gf$ and $\Sigma(\gf,\af)$ be as introduced in Section \ref{Notions},
and let $\Sigma^+$ be the positive system determined by the minimal parabolic subgroup $P$.
Let $(V,\eta)$ be a spherical pair and assume that there exists a quotient morphism
$\pi_{\lambda, \sigma}  \twoheadrightarrow V$ for some $\lambda\in\af^{*}_{\C}$ and $\sigma\in\hat M$.
Then from Corollary \ref{Cor negative and integral} we derive the following.
\begin{enumerate}[(i)]
\item\label{parameter negative item 1}  $2\lambda$ is weakly integral-negative
if $V$ belongs to the twisted discrete series for $Z$. In fact we may take $\af_{\lambda}$ and $\omega_{\lambda}$ to be equal to $\bfrak$ and $\omega$ as in Corollary \ref{Cor negative and integral}(\ref{Cor negative and integral item 1}).
\item \label{parameter negative item 2} $\lambda$ is integral-negative if $V$ belongs to the discrete series for $Z$.
\end{enumerate}
\end{rmk}

\begin{rmk}\label{rem strictly int-neg} Sometimes more is true for parameters of the discrete series and $\lambda$ is actually strictly
integral-negative.  This for example happens in the group case $Z=G\times G/ G \simeq G$.
\end{rmk}

Let us define the {\it edge} of $\lambda$ by
$$
\ef:=\ef(\lambda):=\{ X\in \af \mid (\forall \alpha\in \Sigma(\lambda))  \ \alpha(X)=0\}\, ,
$$
i.e., $\ef$ is the intersection of all faces of $C(\lambda)$.

Notice the orthogonal decomposition
\begin{equation}\label{decomposition with e}
\af = \ef \oplus \Span_\R \Sigma(\lambda)^\vee\, .
\end{equation}

\begin{theorem} \label{fundamental lemma} Let $\lambda \in \af_\C^*$. Then the following assertions hold:
\begin{enumerate}[(i)]
\item\label{fundamental lemma weakly int-neg}
    Suppose that $[\lambda]$ consists of weakly integral-negative parameters. Then there exists a $w\in W$ with $w^{-1}\sim_{\lambda} e$ such that $\ef\subseteq w^{-1}\af_{w\lambda}$. Moreover, $\re \lambda\big|_\ef=0$. Finally, there exists an $N\in \N$ only depending on $\Sigma$ such that $\re\lambda(\alpha^\vee)\in \frac{1}{N}\Z$ for all $\alpha\in\Sigma$.
\item\label{fundamental lemma int-neg}
    If $[\lambda]$ consists of  integral-negative parameters, then $\lambda\big|_{\ef}=0$. In particular,  $\lambda=\re\lambda$.
\item\label{fundamental lemma strictly int-neg}
    If $[\lambda]$ consists of strictly integral-negative parameters, then $\ef=\{0\}$. In particular, $\Sigma(\lambda)^\vee$ has full rank.
\end{enumerate}
\end{theorem}

\begin{proof}
We start with (\ref{fundamental lemma weakly int-neg}).  Let $\mu \in [\lambda]$, that is $\mu=w \lambda$ for some $w\in W$
with $w^{-1}\sim_{\lambda} e$ by Lemma \ref{Eric1}(\ref{Lemma Eric1 item 2}).
Since $\mu$ is weakly integral-negative there exists a subspace $\af_{\mu}$ of $\af$ and an $\omega_\mu \in \Span_\R \Sigma(\mu)$ such that
$(\re \mu - \omega_\mu)\big|_{C\setminus\af_{\mu}}< 0$ and $(\re \mu -  \omega_\mu)\big|_{C}\leq0$.
The latter conditions are equivalent to $(\re \lambda - w^{-1}\omega_\mu)\big|_{w^{-1}C\setminus w^{-1}\af_{\mu}}<0$ and $(\re \lambda -  w^{-1}\omega_\mu)\big|_{w^{-1}C}\leq0$.
Now define a function $f:\af\to \R$ by
$$
f(X):=\max_{w^{-1}\sim_{\lambda} e} w^{-1}\omega_{w\lambda}(X)\qquad (X\in \af)\, .
$$
By Lemma \ref{Eric1}(\ref{Lemma Eric1 item 1}) we have $C(\lambda)= \bigcup_{w^{-1}\sim_{\lambda} e}  w^{-1}C$, and thus
\begin{align}
\label{eq re lambda-f<0 on most of C(lambda)}
    &\big(\re \lambda-f\big)\big|_{C(\lambda)\setminus\bigcup_{w^{-1}\sim_{\lambda} e}w^{-1}\af_{w\lambda}}<0\, ,\\
\label{eq re lambda-fleq0 on  C(lambda)}
    &\big(\re \lambda-f\big)\big|_{C(\lambda)}\leq0\, .
\end{align}
Recall that $\ef$ is the intersection of all faces of $C(\lambda)$.
Since $w\Sigma(\lambda)^\vee = \Sigma(w\lambda)^\vee$ for every $w\in W$, we have $w^{-1}\omega_{w\lambda} \in \Span_{\R}\big(\Sigma(\lambda)\big)$. It follows that $w^{-1}\omega_{w\lambda}\big|_{\ef}=0$ and thus $f\big|_{\ef}=0$.
Since $\ef\bs \bigcup_{w^{-1}\sim_{\lambda} e }w^{-1}\af_{w\lambda}$ is invariant under multiplication by $-1$, it follows from (\ref{eq re lambda-f<0 on most of C(lambda)}) that
$\ef\subseteq \bigcup_{w^{-1}\sim_{\lambda} e }w^{-1}\af_{w\lambda}$. Hence $\ef\subseteq w^{-1}\af_{w\lambda}$ for some $w\in W$ with $w^{-1}\sim_{\lambda} e$. It now follows from (\ref{eq re lambda-fleq0 on  C(lambda)}) that $\re \lambda\big|_{\ef}=0$.

We call a root subsystem $\Sigma'$ of $\Sigma$ parabolic if $\Sigma'$ is the intersection of $\Sigma$ with a subspace.
Let $\Sigma_P(\lambda)\subseteq\Sigma$ be the parabolic closure of
$\Sigma(\lambda)\subseteq \Sigma$, i.e., the smallest parabolic
root subsystem of $\Sigma$ containing $\Sigma(\lambda)$. Then
$\Sigma_P(\lambda)=\ef^\perp\cap\Sigma$, and $\Sigma(\lambda)^{\vee}\subseteq\Sigma_P(\lambda)^{\vee}$ is a root subsystem of maximal rank of
the corresponding dual parabolic subsystem $\Sigma_P(\lambda)^{\vee}$ of
$\Sigma^{\vee}$. By the above, $\re(\lambda)\in \ef^\perp$, and by definition of $\Sigma(\lambda)$, $\re(\lambda)$ is a weight of $\Sigma(\lambda)$.

Let $N$ be the index of the root lattice of $\Sigma_P(\lambda)$ in
the weight lattice of $\Sigma(\lambda)$ (which is a lattice containing
the weight lattice of $\Sigma_P(\lambda)$).
Then $N \re(\lambda)$ is in the root lattice of $\Sigma_P(\lambda)$ and
thus, a fortiori, in the root lattice of $\Sigma$. In particular, $N\re(\lambda)$ is integral
for $\Sigma$ (i.e., as a functional on $\Sigma^{\vee}$).

Since there are only finitely many root subsystems of maximal rank
in any given root system, and only finitely many parabolic root
subsystems, we see that we can choose the bound $N\in\mathbb{N}$
independent of $\lambda$ (only depending on $\Sigma$).
This completes the proof of (\ref{fundamental lemma weakly int-neg}).

We move on to (\ref{fundamental lemma int-neg}). From (i) it follows that there exists a $w\in W$ with $w^{-1}\sim_{\lambda} e$ such that $\ef\subseteq w^{-1}\af_{w\lambda}$.
Now $\lambda(\ef) \subseteq \lambda(w^{-1}\af_{w\lambda}) =w\lambda(\af_{w\lambda})\subseteq \R$. It follows that $\lambda\big|_{\ef}$ is
real and thus $\lambda\big|_{\ef}=0$ by (\ref{fundamental lemma weakly int-neg}).
It then follows from (\ref{decomposition with e}) that $\lambda=\re\lambda$.

Finally for (\ref{fundamental lemma strictly int-neg}) we observe that $[\lambda]$ being strictly integral-negative
implies, as above, $\re \lambda(X) < f(X) $ for all $X \in C(\lambda)\setminus\{0\}$ and therefore
$\re \lambda\big|_{\ef^\times}<0$. The latter forces $\ef^\times=\emptyset$, i.e., $\ef=\{0\}$.
\end{proof}

\subsection{Additional results}
The assertions in this subsection are of independent interest, but not needed in the remainder of this article.

Given a full rank subsystem $(\Sigma^0)^\vee$ of $\Sigma^\vee$ we note that $\Z[(\Sigma^0)^\vee]$ has finite index
in the full co-root lattice $\Z[\Sigma^\vee]$ and thus
$$
\Z[\Sigma^\vee]/ \Z[(\Sigma^0)^\vee] \simeq \bigoplus_{j=1}^r \Z/ d_j \Z
$$
for $d_j\in \N$.  Set
$N(\Sigma^0):=\operatorname{lcm}\{ d_1, \ldots, d_r\}$ and note that $N(\Sigma^0)\alpha^\vee \in \Z[(\Sigma^0)^\vee]$ for all
$\alpha\in \Sigma$.

The following corollary is particularly relevant for the group case $Z=G\times G/G$. See Remark \ref{rem strictly int-neg}.

\begin{cor}\label{Cor strictly int-neg parameters}
Let $\lambda\in \af_\C^*$ be such that $[\lambda]$ consist of strictly integral-negative parameters. Then
$$
\lambda(\alpha^\vee)\in \frac{1}{N(\Sigma(\lambda))}\Z \qquad (\alpha\in \Sigma)\, .
$$
\end{cor}

Note that
$$
N_\Sigma:=\operatorname{lcm}\{ N(\Sigma^0)\mid \hbox{ $(\Sigma^0)^\vee$ is full rank subsystem of $\Sigma^\vee$}\}\,.
$$
is finite as there are only finitely many full rank subsystems of $\Sigma^\vee$. Therefore, $N_{\Sigma}$ is an upper bound for the indices $N(\Sigma(\lambda))$ which only depends on $\Sigma$.

\begin{rmk}
Full rank subsystems can be described by repeated applications of
the "Borel-de Siebenthal" theorem. That is: The maximal such
subsystems are obtained by removing a node from the affine extended
root system (and we can repeat this procedure to obtain the non maximal
cases).

In type $A_n$, there are no proper subsystems of this type, since
the affine extension is a cycle, so removing a node will again
yield $A_n$. Hence if  $\Sigma$ is of type $A_n$, then the condition that $[\lambda]$ consists of strictly integral-negative parameters implies that $[\lambda]=\{\lambda\}$, and $\lambda$ is integral on
all coroots.
\end{rmk}

\section{Integrality properties of leading exponents of twisted discrete series}

For every $\alpha\in\Pi$ and $\lambda\in\af_\C^*$ we set $\lambda_\alpha:=s_\alpha(\lambda)$
and $\sigma_\alpha:= \sigma\circ s_\alpha$.  Further we let
$ I_\alpha(\lambda):  \pi_{\lambda_\alpha, \sigma_\alpha}^\infty   \to \pi_{\lambda, \sigma}^\infty$
be the rank one intertwining operator.  If we identify the space of smooth vectors of
$\pi_{\lambda,\sigma}$ with $C^\infty (K\times_M V_\sigma)$ then the assignment
$$
\af_\C^* \to \End(C^\infty(K\times_M V_\sigma)), \ \ \lambda\mapsto I_\alpha(\lambda)
$$
is meromorphic.
In the appendix we prove:
\begin{lemma}  \label{rank one intertwiners}
There exists a constant $N\in \N$ only depending on $G$ with the following property: If $\lambda(\alpha^\vee)\not \in \frac{1}{N} \Z$, then $I_\alpha(\lambda)$ is an isomorphism.
\end{lemma}

Combining Lemma \ref{rank one intertwiners} with Remark \ref{rmk parameter negative} we obtain:

\begin{cor}\label{intertwiner}
Let $N\in\N$ be as in Lemma \ref{rank one intertwiners}.
Let $(V,\eta)$ be a representation of the twisted discrete series
and $\pi_{\lambda, \sigma}\twoheadrightarrow V$ a quotient morphism.
Then the equivalence class $[2N\lambda]$ consists of weakly integral-negative parameters.
If moreover $(V,\eta)$ belongs to the discrete series, then $[2N\lambda]$ consists of
integral-negative parameters.
\end{cor}

\begin{proof}
If $\alpha\in \Pi$ and $\alpha \not \in \Sigma(2N\lambda)$, then $I_{\alpha}(\lambda)$ is an isomorphism
by Lemma \ref{rank one intertwiners}. Therefore the composition of $I_{\alpha}(\lambda)$ with the quotient
morphism $\pi_{\lambda,\sigma}\twoheadrightarrow V$ gives a quotient morphism
$\pi_{\lambda_{\alpha},\sigma_{\alpha}}\twoheadrightarrow V$. It then follows from
Remark \ref{rmk parameter negative}(\ref{parameter negative item 1}) that $2\lambda_{\alpha}$
and thus also $2N\lambda_{\alpha}$ is weakly integral-negative. By repeating this argument
we obtain that the equivalence class $[2N\lambda]$ consists of weakly integral-negative elements.
If $(V,\eta)$ belongs to the discrete series, then we use (\ref{parameter negative item 2}) in
Remark \ref{rmk parameter negative} instead of (\ref{parameter negative item 1}).
\end{proof}

Recall the set of spherical roots $S\subseteq \af_Z^*$ and recall that $S\subseteq \Z[\Sigma]$.
Let $\chi\in(\hat\hf/\hf)^{*}_{\C}$ be normalized unitary and let $\mu\in\af_Z^*$ be a leading exponent
of a $\chi$-twisted discrete series representation $(V,\eta)$.
Then we know from (\ref{DS1}), (\ref{DS2}), and (\ref{DS3}) that we may expand $\mu$ as
\begin{equation}\label{mu expansion}
\mu=\rho_Q+ \sum_{\alpha\in S} c_\alpha \alpha + i\nu  \qquad (c_\alpha\in \R)\, .
\end{equation}
with
\begin{enumerate}[(a)]
\item $c_\alpha > 0$ for all $\alpha\in S$,
\item $\nu\in\af_{Z}^{*}$ with $\nu\big|_{\af_{Z,E}}=\im\chi\big|_{\af_{Z,E}}$.
\end{enumerate}

\begin{theorem}\label{main thm}
Let $Z=G/H$ be a unimodular real spherical space.
There exists an $N\in\N$ and for every normalized unitary $\chi\in(\hat\hf/\hf)^{*}_{\C}$ a finite set ${\mathfrak Y}_{\chi}\subseteq \af^{*}$ with the following property.
Let $(V,\eta)$ be a spherical pair corresponding to a $\chi$-twisted discrete series representation and let $\mu$ be any leading exponent of $(V,\eta)$, which we expand as $\mu=\rho_{Q}+\sum_{\alpha\in S} c_\alpha \alpha+i\nu$ as in (\ref{mu expansion}).
Then the following hold.
\begin{enumerate}[(i)]
\item\label{Main thm item 1} $c_\alpha\in \frac{1}{N} \N$ for all $\alpha\in S$ and $\nu\in {\mathfrak Y}_{\chi}$.
\item\label{Main thm item 2} If in addition $(V,\eta)$ belongs to the discrete series, then $\nu=0$, i.e., $\mu \in \af_Z^*$.  In particular, the infinitesimal
character of $V$ is real.
\end{enumerate}
\end{theorem}

\begin{proof}
We let $\lambda:= w_0 \oline \mu +\rho_P$ and recall from Lemma \ref{disc cover}  that there exists
a $\sigma\in \hat M$ such that $\pi_{\lambda, \sigma} \twoheadrightarrow V$.
By Corollary \ref{intertwiner} there exists a constant $N(G)\in\N$, depending only on $G$, such that the equivalence class $[2N(G)\lambda]$ consists of weakly integral-negative elements.
By Theorem \ref{fundamental lemma} (\ref{fundamental lemma weakly int-neg}) there exists an
$N'\in\N$, only depending on $G$, such that
$$
\re\lambda(\alpha^{\vee})\in\frac{1}{N'}\Z\qquad(\alpha\in\Sigma)\,.
$$
This implies that $\re\lambda\in\frac{1}{N''}\Z(\Pi)$ for some $N''\in\N$ depending only on $G$.
Since the spherical roots are integral linear combinations of simple roots, it follows that there exists a $N\in \N$ (only depending on $Z$) such that $c_\alpha\in \frac{1}{N} \N$.
Moreover, it follows from Corollary \ref{Cor negative and integral}(\ref{Cor negative and integral item 1}) and Theorem \ref{fundamental lemma}(\ref{fundamental lemma weakly int-neg})
(cf. Remark \ref{rmk parameter negative})
that the imaginary part of $\lambda$ is contained in a finite subset of $\af^{*}$ depending only on $\chi$. This proves (i).
For the second assertion we use  (\ref{fundamental lemma int-neg}) in Theorem \ref{fundamental lemma} instead of (\ref{fundamental lemma weakly int-neg}).
The infinitesimal character of $V$ is equal to the infinitesimal character of $\pi_{\lambda,\sigma}$, which is real since $\lambda$ is real.
\end{proof}

Theorem \ref{main thm} (\ref{Main thm item 2}) implies the following.

\begin{cor}\label{K-type}
Fix a normalized unitary $\chi\in(\hat\hf/\hf)^{*}_{\C}$ and a $K$-type $\tau$.
There are only finitely many $\chi$-twisted discrete series representations $V$ for $Z$
such that the $\tau$-isotypical component $V[\tau]$ of $V$ is non-zero.
\end{cor}

\begin{proof} (cf. \cite[Lemma 70, p. 84]{HC})  Let $\tf\subseteq \mf$ be a Cartan subalgebra of $\mf$.   Set $\cf:=\af+i\tf$ and note that
$\cf_\C$ is a Cartan subalgebra of $\gf_\C$.  We inflate $\Sigma^+=\Sigma^+(\gf,\af)$ to a
positive system $\Sigma^+(\gf_\C, \cf)$ and write $\rho_B$ for the corresponding half sum.  Observe
that $\rho_B=\rho_P +\rho_M \in \cf^*$.
We identify $\sigma$ with its highest weight in $i\tf^*$ and write $\langle\cdot\rangle^{2}$ for the quadratic form on $\cf_{\C}$ obtained from the Cartan-Killing form.
Let $C_\gf$ be the Casimir element of $\gf$.  Note that $C_\gf$ acts on $\pi_{\lambda,\sigma}$  with $\lambda\in \af_{\C}^*$ as the scalar
$$
\langle\lambda + \sigma+\rho_M\rangle^2 - \langle\rho_B\rangle^2\, .
$$
Let $\tf_{\kf}\supseteq \tf$ be a Cartan subalgebra of $\kf$ and $\rho_\kf\in i\tf_{\kf}^*$ be the Weyl half sum with respect to a fixed positive system of $\Sigma(\kf_\C, \tf_{\kf})\subseteq i \tf_{\kf^*}$. As before we identify $\tau\in \hat K$ with its highest weight in $i\tf_{\kf}^*$. We write $\langle\cdot\rangle_{\kf}^{2}$ for the quadratic form on $\tf_{\kf,\C}$ obtained from the Cartan-Killing form.
Further we let $C_\kf$ denote the Casimir element of $\kf$.  The element $\Delta:=C_\gf +2C_\kf$ is a Laplace element and thus $\la \Delta v, v\ra \leq 0$ for all
$K$-finite vectors in a unitarizable Harish-Chandra module $V$.

Let now $V$ be a $\chi$-twisted discrete series representation and $\pi_{\lambda, \sigma}\twoheadrightarrow
V$ a quotient morphism. For $0\neq v\in V[\tau]$ we obtain
\begin{align*}
0\geq \la \Delta v, v\ra &= \la (C_\gf + 2 C_{\kf}) v, v\ra\\
&=\Big( \langle\lambda + \sigma+\rho_M\rangle^2 - \langle\rho_B\rangle^2
    - 2\big(\langle\tau +\rho_\kf\rangle_{\kf}^{2} - \langle\rho_\kf\rangle_{\kf}^{2}\big)\Big) \la v, v\ra\, .
\end{align*}
This forces
$$
\langle\lambda + \sigma+\rho_M\rangle^2 - \langle\rho_B\rangle^2
    \leq 2 \big(\langle\tau +\rho_\kf\rangle_{\kf}^{2} - \langle\rho_\kf\rangle_{\kf}^{2}\big)
$$
and in particular
$$
\langle\re\lambda\rangle^2-\langle\im\lambda\rangle^2 - \langle\rho\rangle^2
\leq 2 \big(\langle\tau +\rho_\kf\rangle_{\kf}^{2} - \langle\rho_\kf\rangle_{\kf}^{2}\big)\, .
$$
By Theorem \ref{main thm} (\ref{Main thm item 1}) $\re \lambda$ is discrete  and $\im\lambda$ is contained in a finite set that only depends on $Z$.
The assertion now follows from the fact that the map ${\mathfrak X}$ from (\ref{infinitesimal character map}) has finite fibers.
\end{proof}

\section*{Appendix A: Invariant Sobolev Lemma}
\addcontentsline{toc}{section}{Appendix: Invariant Sobolev Lemma}
\setcounter{section}{1}
\setcounter{theorem}{0}
\setcounter{equation}{0}
\renewcommand{\thesection} {\Alph{section}}

The aim of this appendix is an invariant Sobolev lemma for functions on $Z$ that transform under the right action of $A_{Z,E}$ by a unitary character.

Recall that a weight on $Z$ is a locally bounded function $w:Z\to\R_{>0}$ with the property that for every compact subset $\Omega\subseteq G$ there exists a constant $C>0$ such that
$$
w(gz)\leq Cw(z)\qquad(z\in Z,g\in\Omega)\,.
$$
Further recall that there is a natural identification between the space of smooth densities on $\hat Z$ and the space of functions
$$
C^{\infty}(G:\Delta_{\hat Z})
:=\{f\in C^{\infty}(G):f(\dotvar\hat h)=\Delta_{\hat Z}^{-1}(\hat h)f \text{ for }\hat h\in\hat H\}\,,
$$
where $\Delta_{\hat Z}$ is the modular character
$$
\Delta_{\hat Z}:\hat H\to\R_{>0};\qquad ha\mapsto a^{-2\rho_{Q}}\qquad(a\in A_{Z,E},h\in H)\,.
$$
See Sections 8.1 and 8.2 in \cite{KKS2}. Note that smooth functions $f:G\to\C$ satisfying
$$
f(\dotvar ha)=a^{\rho_{Q}+i\nu}f\qquad(h\in H, a\in A_{Z,E})
$$
for some $\nu\in\af_{Z,E}^{*}$, are in the same way identified with smooth half-densities on $\hat Z$.

Let $B$ be a ball in $G$, i.e., a compact symmetric neighborhood of $e$ in $G$. Recall that the corresponding volume-weight $\v_{B}$ is defined by
$$
\v_{B}(z):=\vol_Z (Bz)\qquad (z\in Z)\,.
$$
Note that if $B'$ is another ball in $G$, then there exists $c>0$ such that
$$
\frac{1}{c}\v_{B'}\leq \v_{B}\leq c\v_{B'}\,.
$$
In the following we drop the index and write $\v$ instead of $\v_{B}$.

The following lemma is a generalization of the invariant Sobolev lemma of Bernstein.
See the key lemma in \cite{Bernstein} on p. 686 and \cite[Lemma 4.2]{KS3}.

\begin{lemma}\label{Invariant Sobolev Lemma}
For every $k>\dim G$ there exists a constant $C>0$ with the following property.
Let $\nu \in \af_{Z,E}^*$ and let $f\in C^\infty(Z)$ be a smooth function which transforms
as $f(z\cdot a) = f(z)a^{\rho_Q + i\nu}$ for all $a\in A_{Z,E}$, and
let $\Omega_f$ be the attached  half-density  on $\hat Z=G/\hat H$.
Then
$$
|f(z)|\leq  C \v(z)^{-\frac{1}{2} } \| \Omega_f\|_{B\hat z, 2;k} \qquad (z\in Z)\,.
$$
Here $\hat z\in \hat Z$ is the image of $z\in Z$ and $\|\cdot\|_{B\hat z,2:k}$ is the $k$'th $L^{2}$-Sobolev norm on $B\hat z$.
\end{lemma}

Let $A_{0}$ be a closed subgroup of $A$ such that the multiplication map $A_{0}\times A_{E}\to A$ is a diffeomorphism. Let $A_0^-$ be the cone such that $A_0^{-}A_E/A_E=A_{\hat Z}^{-}$.
By taking inverse images of the projection $A_{Z}=A/(A\cap H)\to A_{\hat{Z}}=A/A_{E}$ we get
\begin{equation}\label{eq formula for A_Z^-}
A_{0}^{-}A_{E}/(A\cap H)=A_{Z}^{-}\,.
\end{equation}

We recall from \cite[Section 3.4]{KKS2} that there exists a finite sets $F,\W\subseteq G$ such that
\begin{equation}\label{eq WA_E subseteq A_E WH}
\W A_{Z,E}\subseteq A_{Z,E}\W H
\end{equation}
and
$$
\hat Z=FK A_{\hat Z}^{-}\W\cdot\hat z_{0}\,.
$$
For the proof of the invariant Sobolev lemma we need the following lemma.

\begin{lemma}\label{ineq for int in polar coordinates}
There exists an $a_{1}\in A$ and a constant $c>0$, depending only on the normalization of the Haar measures on $K$ and $A_{\hat{Z}}$, such that for all compactly supported measurable non-negative densities $f$ on $\hat Z$ we have
$$
\int_{\hat Z}f
\geq c\sum_{w\in \W}\int_{K}\int_{A_{\hat{Z}}^{-}}f(ka_{1}aw)a^{-2\rho_{Q}}\,da\,dk\,.
$$
\end{lemma}

\begin{proof}
Let $f$ be a compactly supported measurable non-negative density on $\hat Z$ and let $\varphi:Z\to\R_{\geq0}$ be a compactly supported continuous function such that
$$
\int_{A_{E}/A\cap H}\varphi(za_{E})a_{E}^{-2\rho_{Q}}\,da_{E}=f(\hat z) \qquad(z\in Z)\,.
$$
Here $\hat z\in\hat Z$ denotes the image of $z\in Z$.
Then by the Fubini theorem for densities (see \cite[Theorem A.8]{vdBK})
$$
\int_{\hat Z}f
=\int_{Z}\varphi(z)\,dz.
$$
We will use Lemma 3.3 (1) in \cite{KSS} to obtain a lower bound for
this integral. The estimate in that lemma involves the integration
over the conjugate of the maximal compact subgroup by some element in $A$, which we shall
denote by $a_1$. We apply the lemma to the function $z\mapsto\varphi(a_1\cdot z)$ on $Z$, and
write the estimate in terms of the original maximal compact subgroup $K$. By this we obtain
a constant $c>0$ such that
$$
\int_{Z}\varphi(a_{1}\cdot z)\,dz
\geq c\sum_{w\in \W}\int_{K}\int_{A_{Z}^{-}}\varphi(ka_{1}aw)a^{-2\rho_{Q}}\,da\,dk\,.
$$
Using that the measure on $Z$ is $G$-invariant and (\ref{eq formula for A_Z^-}), we obtain
$$
\int_{Z}\varphi(z)\,dz
\geq c\sum_{w\in \W}\int_{K}\int_{A_{0}^{-}}\int_{A_{E}/A\cap H}\varphi(ka_{1}aa_Ew)a_E^{-2\rho_{Q}}a^{-2\rho_{Q}}\,da_E\,da\,dk\,.
$$
In view of (\ref{eq WA_E subseteq A_E WH}) we now have
$$
\int_{Z}\varphi(z)\,dz
\geq c\sum_{w\in \W}\int_{K}\int_{A_{0}^{-}}f(ka_{1}aw)a^{-2\rho_{Q}}\,da\,dk
= c\sum_{w\in \W}\int_{K}\int_{A_{\hat{Z}}^{-}}f(ka_{1}aw)a^{-2\rho_{Q}}\,da\,dk\,.
$$
\end{proof}

\begin{proof}[Proof of Lemma \ref{Invariant Sobolev Lemma}]
We will prove that there exists a constant $C>0$ such that for every non-negative smooth density $\phi$ on $\hat Z$ and every $x\in G$
\begin{equation}\label{integral inequality}
\int_{B}\phi(gx)\,dg
\leq C\frac{1}{\v(x)}\int_{Bx\hat H} \phi\,.
\end{equation}
On the left-hand side $\phi$ is considered as a function on $G$ that transforms under the right-action of $\hat H$ with the modular character.
Before giving the proof of \eqref{integral inequality} we derive the lemma from it.
By the local Sobolev lemma, applied to the function
$f(\dotvar z)$ on $G$, we obtain the following bound by
the $k$-th Sobolev norm of
$f(\dotvar z)$ over the neighborhood $B$ of $e\in G$:
$$|f(z)|\leq C\|f(\dotvar z)\|_{B,2;k}.$$
The constant $C$ is independent of $f$ and $z$.
Choose $x\in G$ such that $z=xH$.
Using \eqref{integral inequality} for the square of each derivative
up to $k$ of $\Omega_f$,
we also have
$$\|f(\dotvar z)\|_{B,2;k}^2\leq C\frac1{\v(x)}\,\|\Omega_f\|_{Bx\hat H,2;k}^2.$$
The lemma follows from these inequalities.

For a measurable function $\chi:Z\to\R_{\geq0}$, let $\psi_{\chi}:G\to\R_{\geq0}$ be such that
$$
\chi=\int_{H}\psi_{\chi}(\dotvar h)\,dh\,.
$$
Then for every $a\in A_{Z,E}$ we have
$$
\int_{G}\psi_{\chi}(xa)\,dx
=\int_{Z}\int_{H}\psi_{\chi}(gha)\,dh\,dgH
=\big|\det\Ad(a)\big|_{\hf}\big|\int_{Z}\chi(z\cdot a)\,dz\,.
$$
Since $\big|\det\Ad(a)\big|_{\hf}\big|=a^{-2\rho_{Q}}$, and by the invariance of the Haar measure the left-hand side is independent of $a$, it follows that
$$
\int_{Z}\chi(z\cdot a)\,dz
=a^{2\rho_{Q}}\int_{Z}\chi(z)\,dz\,.
$$
We may apply this to $\chi=\1_{Bz}$ and obtain
$$
\v(\dotvar a)=a^{-2\rho_{Q}}\v\qquad(a\in A_{Z,E})\,.
$$
We conclude that $\frac{1}{\v}$ may be considered as a density on $Z$.

Let $B\subseteq G$ be a ball and define $\w_{B}:\hat Z\to\R_{>0}$ by
$$
\w_{B}(\hat z)
:=\int_{B \hat z}\frac{1}{\v}\qquad (\hat z\in\hat Z)\,.
$$
If $B'$ is another ball in $G$, then we may cover $B'$ by a finite number of sets of the form $g B$. Since $\v$ is a weight, it follows that
there exists a $c>0$ such that
\begin{equation}\label{comparablity w's}
\frac{1}{c}\w_{B'}\leq \w_{B}\leq c\w_{B'}\,.
\end{equation}
Let $\Omega$ be a compact subset of $G$. Let $B'=\{g^{-1}bg:g\in \Omega,b\in B\}$. Then
$$
\w_{B}(g\hat z)
=\int_{Bg\hat z}\frac{1}{\v}
\leq\int_{gB'\hat z}\frac{1}{\v}
=\w_{B'}(\hat z)
\qquad(\hat z\in \hat Z, g\in\Omega)\,.
$$
From (\ref{comparablity w's}) it follows that there exits a $c>0$ such that
$$
\w_{B}(g\hat z)\leq c\w_{B}(\hat z)\qquad(\hat z\in\hat Z,g\in\Omega)\,.
$$
We thus see that $\w_{B}$ is a weight.

We claim that there exists a $c_{1}>0$ such that for every $z\in\hat Z$
\begin{equation}\label{inf w>0}
\w_{B}(\hat z)>c_{1}\,.
\end{equation}
Since $\w_{B}$ is a weight, it suffices to show that $\inf_{a_{0}\in A_{\hat Z}^{-},w_{0}\in \W}\w_{B}(a_{0}w_{0}\cdot \hat z_{0})>0$ to prove this claim.

Let $a_{0}\in A_{\hat Z}^{-}$ and $w_{0}\in\W$.
It follows from the inequality (3.6) in \cite{KSS} and Lemma \ref{ineq for int in polar coordinates} that there exists a an element $a_{1}\in A$ and a constant $C>0$ such that,
\begin{align*}
\w_{B}(a_{0}w_{0}\cdot \hat z_{0})
&\geq C\sum_{w\in \W}\int_{K}\int_{A_{\hat Z}^{-}}\1_{Ba_{0}w_{0}\cdot z_{0}}(ka_{1}aw\cdot \hat z_{0})\,da\,dk\\
&\geq C\int_{K}\int_{A_{\hat Z}^{-}}\1_{Ba_{0}w_{0}\cdot z_{0}}(ka_{1}aw_{0}\cdot \hat z_{0})\,da\,dk\\
&\geq C\int_{K}\int_{a_{1}A_{\hat Z}^{-}}\1_{Ba_{0}w_{0}\cdot z_{0}}(kaw_{0}\cdot \hat z_{0})\,da\,dk\,.
\end{align*}
For the last equality we used the invariance of the measure on $A_{\hat Z}$.
Let $A_{c}$ be a compact subset of $A$ with non-empty interior and $A_{c}A_{Z,E}/A_{Z,E} \subseteq a_{1}A_{\hat Z}^{-}$.
By enlarging $B$, we may assume that $B$ is invariant under left translations by elements from $K$ on the left and $A_{c}\subseteq B$. Since $\int_{K}\,dk=1$, we have
$$
\int_{K}\int_{a_{1}A_{\hat Z}^{-}}\1_{Ba_{0}w_{0}\cdot z_{0}}(kaw_{0}\cdot \hat z_{0})\,da\,dk
=\int_{a_{1}A_{\hat Z}^{-}}\1_{Ba_{0}w_{0}\cdot z_{0}}(aw_{0}\cdot \hat z_{0})\,da\,.
$$
If $a\in a_{0}A_{c}A_{Z,E}/A_{Z,E}$, then $aw_{0}\cdot \hat z_{0}\in A_{c}a_{0}w_{0}\cdot \hat z_{0}\subseteq Ba_{0}w_{0}\cdot \hat z_{0}$.
Therefore,
$$
\int_{a_{1}A_{\hat Z}^{-}}\1_{Ba_{0}w_{0}\cdot z_{0}}(aw_{0}\cdot \hat z_{0})\,da
\geq \int_{a_{1}A_{\hat Z}^{-}}\1_{a_{0}A_{c}A_{Z,E}/A_{Z,E}}(a)\,da\,.
$$
Since $a_{0}\in A_{\hat Z}^{-}$ and $A_{\hat Z}^{-}A_{\hat Z}^{-}\subseteq A_{\hat Z}^{-}$, the set $a_{0}A_{c}A_{Z,E}/A_{Z,E}$ is contained in $a_{1}A_{\hat Z}^{-}$ and thus
$$
\int_{a_{1}A_{\hat Z}^{-}}\1_{a_{0}A_{c}A_{Z,E}/A_{Z,E}}(a)\,da
=\int_{A_{\hat Z}}\1_{a_{0}A_{c}A_{Z,E}/A_{Z,E}}(a)\,da
=\int_{A_{\hat Z}}\1_{A_{c}A_{Z,E}/A_{Z,E}}(a)\,da\,,
$$
and hence
$$
\w_{B}(a_{0}w_{0}\cdot \hat z_{0})
\geq C\int_{A_{\hat Z}}\1_{A_{c}A_{Z,E}/A_{Z,E}}(a)\,da\,.
$$
The claim (\ref{inf w>0}) now follows as the right-hand side is independent of $a_{0}$ and strictly positive.

Let $\phi$ be a non-negative smooth density on $\hat Z$ and let $x\in G$. To prove (\ref{integral inequality}) we may assume that $\supp\phi\subseteq Bx\hat H$ and that $B^{-1}=B$.
Since $\v$ is a weight, there exists a constant $c_{2}>0$ such that $\v(x)\leq c_{2}\v(y)$ for every $y\in Bx$. If $y=bx$ with $b\in B$, then
$$
\v(x)\int_{B}\phi(gx)\,dg
\leq c_{2}\v(y)\int_{B}\phi(gb^{-1}y)\,dg
\leq c_{2}\v(y)\int_{B^{2}}\phi(gy)\,dg\,.
$$
Note that $\v\phi$ is right $\hat H$-invariant, hence
$$
\v(x)\int_{B}\phi(gx)\,dg
\leq c_{2}\v(y)\int_{B^{2}}\phi(gy)\,dg
\qquad(y\in Bx\hat H)\,.
$$
Therefore,
$$
\int_{B}\phi(gx)\,dg\int_{y\in Bx\hat{H}}\frac{1}{\v(y)}
\leq \frac{c_{2}}{\v(x)}\int_{y\in Bx\hat{H}}\left[\int_{B^{2}}\phi(gy)\,dg\right].
$$

Let $c_{1}>0$ be as in (\ref{inf w>0}). Then
$$
\int_{B}\phi(gx)\,dg
\leq \frac{1}{c_{1}}\int_{B}\phi(gx)\,dg\,\int_{Bx\hat H}\frac{1}{\v}
\leq \frac{c_{2}}{c_{1}\v(x)}\int_{y\in Bx\hat H}\Big[\int_{B^{2}}\phi(gy)\,dg\Big]\,.
$$
Now we use Fubini's theorem to change the order of integration. We thus get
\begin{align*}
\int_{B}\phi(gx)\,dg
&\leq\frac{c_{2}}{c_{1}\v(x)}\int_{B^{2}}\Big[\int_{y\in Bx\hat H}\phi(gy)\Big]\,dg\\
&\leq\frac{c_{2}}{c_{1}\v(x)}\int_{B^{2}}\Big[\int_{y\in\hat Z}\phi(gy)\Big]\,dg\\
&=\frac{c_{2}\vol(B^{2})}{c_{1}\v(x)}\int_{\hat Z}\phi\,.
\end{align*}
This implies (\ref{integral inequality}) as by assumption $\supp\phi\subseteq Bx\hat H$.
\end{proof}

\section*{Appendix B: Intertwining operators}
\addcontentsline{toc}{section}{Appendix: Intertwining operators}
\setcounter{section}{2}
\setcounter{theorem}{0}
\setcounter{equation}{0}

The main result of this appendix is the following proposition.

\begin{prop}\label{Prop intertwining operators}
There exists a $N\in\N$ such that for every $\alpha\in\Pi$, $\sigma\in\hat{M}$ and $\lambda\in\af_{\C}^{*}$ with $\lambda(\alpha^{\vee})\notin\frac{1}{N}\Z$, the standard intertwining operator
$I_{\alpha}(\lambda,\sigma):\pi_{s_{\alpha}\lambda,s_{\alpha}\sigma}\to\pi_{\lambda,\sigma}$ is defined and an isomorphism.
\end{prop}

Before we prove the proposition, we first prove a lemma.

\begin{lemma}\label{Lemma irreducibility princ series}
Assume that the split rank of $G$ is equal to $1$ and let $\alpha$ be the simple root of $(\gf,\af)$.
There exists a $N\in\N$ such that for every $\sigma\in\hat{M}$ and $\nu\in\af_{\C}^{*}$ with $\nu(\alpha^{\vee})\notin\frac{1}{N}\Z$, the representation $\pi_{\nu,\sigma}$ is irreducible.
\end{lemma}

\begin{proof}
Let $\tf$ be a maximal torus in $\mf$. Let $\hf=\af\oplus i\tf$. Then $\hf_{\C}$ is a Cartan subalgebra of $\gf_{\C}$.
We define $\Sigma(\hf)\subseteq\hf^{*}$ to be the set of roots of $(\gf_{\C},\hf_{\C})$, choose a positive system $\Sigma^{+}(\hf)$ and define
$$
\rho_{M}:=\frac{1}{2}\sum_{\substack{\beta\in\Sigma^{+}(\hf)\\ \beta|_{\af}=0}}\dim(\gf^{\beta})\beta\,.
$$
Let $\xi\in\tf_{\C}^{*}$ be the Harish-Chandra parameter of some constituent $\sigma_{0}$ of the restriction of $\sigma$ to the connected component of $M$. Then $\xi-\rho_{M}$ is the highest weight of $\sigma_{0}$.

We view $\tf_{\C}^{*}$ and $\af_{\C}^{*}$ as subspaces of $\hf_{\C}^{*}$ by extending the functionals trivially with respect to the decomposition $\hf_{\C}=\tf_{\C}\oplus\af_{\C}$.
We write $p_{\af}$ and $p_{\tf}$ for the restrictions $\af_{\C}$ and $\tf_{\C}$ respectively.
Let $\theta$ be the involutive automorphism on $\hf_{\C}$ that is $1$ on $\tf_{\C}$ and $-1$ on $\af_{\C}$. We denote the adjoint of $\theta$ by $\theta$ as well.

Now assume that $\pi_{\nu,\sigma}$ is not irreducible. We write $\gamma=(\xi,\nu)\in\tf_{\C}^{*}\oplus\af_{\C}^{*}$. By \cite[Theorem 1.1]{SpVo} there exists a $\beta\in\Sigma(\hf)$ such that $\gamma(\beta^{\vee})\in\Z$ and either
\begin{enumerate}[(a)]
\item\label{SpVo conditions item 1} $\gamma(\beta^{\vee})>0$, $\gamma(\theta\beta^{\vee})<0$ and $\theta\beta\neq-\beta$, or
\item\label{SpVo conditions item 2} $\theta\beta=-\beta$.
\end{enumerate}
Note that in both cases (\ref{SpVo conditions item 1}) and (\ref{SpVo conditions item 2}) $p_{\af}\beta$ is non-zero and is in fact a root of $(\gf,\af)$. Therefore, $p_{\af}\beta\in \{\pm\alpha,\pm2\alpha\}$.
Let $k\in\{\pm1,\pm2\}$ be such that $p_{\af}\beta=k\alpha$. Then
\begin{align*}
\nu(\alpha^{\vee})
=\frac{k\|\beta\|^{2}}{\|p_{\af}\beta\|^{2}}\frac{2\langle\nu,p_{\af}\beta\rangle}{\|\beta\|^{2}}
=\frac{k\|\beta\|^{2}}{\|p_{\af}\beta\|^{2}}\Big(\gamma(\beta^{\vee})-\frac{2\langle\xi,p_{\tf}\beta\rangle}{\|\beta\|^{2}}\Big)
\in\frac{k\|\beta\|^{2}}{\|p_{\af}\beta\|^{2}}\Big(\Z-\frac{2\langle\xi,p_{\tf}\beta\rangle}{\|\beta\|^{2}}\Big)\,.
\end{align*}

Let $d$ be the determinant of the Cartan matrix of the root system $\Sigma_{\mf}(\tf_{\C})$ of $\mf_{\C}$ in $\tf_{\C}$.
The lattice $\Lambda_{\mf}(\tf_{\C})$ of integral weights of $\mf_{\C}$ in $\tf_{\C}$ is contained in $\frac{1}{d}\Z[\Sigma_{\mf}(\tf_{\C})]$.
Note that $p_{\tf}\beta,\xi\in\Lambda_{\mf}(\tf_{\C})$.
Let $\mathit{l}$ be the square of the length of the shortest root in $\Sigma(\hf)$. Then $\|\Sigma(\hf)\|^{2}\subseteq\{\mathit{l},2\mathit{l},3\mathit{l}\}$ and $\langle\Sigma(\hf),\Sigma(\hf)\rangle\in\frac{\mathit{l}}{2}\Z$. Therefore,
$$
\langle\Lambda_{\mf}(\tf_{\C}),\Lambda_{\mf}(\tf_{\C})\rangle
\subseteq\frac{1}{d^{2}}\langle\Sigma(\tf_{\C}),\Sigma(\tf_{\C})\rangle
\subseteq\frac{\mathit{l}}{2d^{2}}\,\Z\,,
$$
and since $p_{\tf}\beta,\xi\in\Lambda_{\mf}(\tf_{\C})$,
$$
\frac{2\langle \xi,p_{\tf}\beta\rangle}{\|\beta\|^{2}}\in\frac{1}{6d^{2}}\,\Z\,.
$$
Since $\theta\beta\in\Sigma(\hf)$ and by the Cauchy-Schwartz inequality
$$
\frac{2\langle\beta,\theta\beta\rangle}{\|\beta\|^{2}}\in\{0,\pm1,\pm2\}\,.
$$
Taking into account that $0<\|p_{\af}\beta\|^{2}\leq \|\beta\|^{2}$ we obtain
$$
\frac{\|\beta\|^{2}}{\|p_{\af}\beta\|^{2}}
=\frac{2\|\beta\|^{2}}{\|\beta\|^{2}-\langle\beta,\theta\beta\rangle}
\in \{1,\frac{4}{3},2,4\}
$$
and thus
\begin{align*}
\nu(\alpha^{\vee})
\in\frac{k\|\beta\|^{2}}{\|p_{\af}\beta\|^{2}}\Big(\Z-\frac{2\langle\xi,p_{\tf}\beta\rangle}{\|\beta\|^{2}}\Big)
\subseteq\frac{1}{18d^{2}}\Z\,.
\end{align*}
\end{proof}

\begin{proof}[Proof of Proposition \ref{Prop intertwining operators}]
Let $N\in\N$ be as in Lemma \ref{Lemma irreducibility princ series}.
For $\alpha\in\Pi$ let $G_{\alpha}$ be the connected subgroup of $G$ with Lie algebra generated by the subspace $\gf^{-2\alpha}\oplus\gf^{-\alpha}\oplus\gf^{\alpha}\oplus\gf^{2\alpha}$ of $\gf$. Note that the real rank of $G_{\alpha}$ is equal to $1$. We define the subgroups
$$
A_{\alpha}:=A\cap G_{\alpha}\,,
\quad M_{\alpha}:=M\cap G_{\alpha}\,,
\quad P_{\alpha}:=P\cap G_{\alpha}\,.
$$
Write $\sigma_{\alpha}$ and $\lambda_{\alpha}$ for $\sigma\big|_{M_{\alpha}}$ and $\lambda\big|_{\af_{\alpha}}$ respectively.
Let $I^{0}_{\alpha}(\lambda_{\alpha},\sigma_{\alpha})$ be the standard intertwining operator
$$
I^{0}_{\alpha}(\lambda_{\alpha},\sigma_{\alpha}):
\Ind_{P_{\alpha}}^{G_{\alpha}}(s_{\alpha}\lambda_{\alpha}\otimes s_{\alpha}\sigma_{\alpha})
\to
\Ind_{P_{\alpha}}^{G_{\alpha}}(\lambda_{\alpha}\otimes\sigma_{\alpha})\,.
$$
By equation (17.8) in \cite{KnSt} we have
\begin{equation}\label{relation rank one intertwiners}
I_{\alpha}(\lambda,\sigma)f(e)
=I^{0}_{\alpha}(\lambda_{\alpha},\sigma_{\alpha})\big(f\big|_{G_{\alpha}}\big)(e)
\qquad\big(f\in \pi^{\infty}_{s_{\alpha}\lambda,s_{\alpha}\sigma}\big)\,.
\end{equation}

The poles of the meromorphic family $I^{0}_{\alpha}(\lambda_{\alpha},\sigma_{\alpha})$ are located at the $\lambda_{\alpha}\in\af_{\alpha}^{*}$ such that $\lambda_{\alpha}(\alpha^{\vee})\in-\N_{0}$. See \cite[Theorem 3]{KnSt}.
It follows from (\ref{relation rank one intertwiners}) that $I_{\alpha}(\lambda,\sigma)$ is defined for $\lambda(\alpha^{\vee})\notin -\N_{0}$.

Now assume that $\lambda(\alpha^{\vee})\notin\Z$. Let $\phi_{0}\in C_{c}^{\infty}(\oline N, V_{\sigma})$ be such that $\int_{\oline N\cap s_{\alpha}N}\phi_{0}(\oline n)\,d\oline n\neq0$. Define $\phi\in\pi^{\infty}_{\sigma\lambda},s_{\alpha\sigma}$ by setting $\phi\big|_{\oline N}=\phi_{0}$. Then the integral
$$
\int_{\oline N\cap s_{\alpha}N}\phi(\oline n)\,dn
$$
is absolutely convergent and non-zero. Hence $I_{\alpha}(\lambda,\sigma)\phi(e)$ is non-zero. In particular this shows that
both $I_{\alpha}(\lambda,\sigma)$ and $I^{0}_{\alpha}(\lambda_{\alpha},\sigma_{\alpha})$ are non-zero.

If $I_{\alpha}(\lambda,\sigma)$ is not injective, then there exists a $\phi\in \pi^{\infty}_{s_{\alpha}\lambda,s_{\alpha}\sigma}$ such that $I_{\alpha}(\lambda,\sigma)\phi=0$ and $\phi(e)\neq 0$. It then follows from (\ref{relation rank one intertwiners}) that $I_{\alpha}^{0}(\lambda_{\alpha},\sigma_{\alpha})$ is not injective either. Since $I_{\alpha}^{0}(\lambda_{\alpha},\sigma_{\alpha})$ is non-zero,  $\Ind_{P_{\alpha}}^{G_{\alpha}}(s_{\alpha}\lambda_{\alpha}\otimes s_{\alpha}\sigma_{\alpha})$ is not irreducible.
Similarly, if $I_{\alpha}(\lambda,\sigma)$ is not surjective, then its adjoint $I_{\alpha}(\lambda,\sigma)^{*}=I_{\alpha}(-s_{\alpha}\lambda,s_{\alpha}\sigma^{\vee})$ is not injective, hence it follows that $\Ind_{P_{\alpha}}^{G_{\alpha}}(-\lambda_{\alpha}\otimes\sigma_{\alpha}^{\vee})$ is not irreducible.
It now follows from Lemma \ref{Lemma irreducibility princ series} that if $I_{\alpha}(\lambda,\sigma)$ is not an isomorphism then  $\lambda(\alpha^\vee)\in \frac{1}{N}\Z$.
\end{proof}

\end{document}